\documentclass[final]{siamltex}

\usepackage[top=2.5cm,bottom=2.5cm,right=2.5cm,left=2.5cm]{geometry}
\usepackage{subcaption}
\usepackage[notcite,notref]{showkeys}

\usepackage{endnotes}

\def\sl{\textstyle{\sum}_{\ell=1}^k}

\newcommand{\tsum}{\textstyle\sum}

\usepackage{epsfig}
\usepackage{calc}
\usepackage{amstext}
\usepackage{amsmath}
\usepackage{multicol}
\usepackage{amsfonts}
\usepackage{amssymb}
\usepackage{mathrsfs}
\usepackage{bm}
\usepackage{xcolor}
\usepackage{comment}
\usepackage[scientific-notation=true]{siunitx}

\newcommand{\bbr}{\Bbb{R}}
\newcommand{\beq}{\begin{equation}}
\newcommand{\eeq}{\end{equation}}
\newcommand{\beqa}{\begin{eqnarray}}
\newcommand{\eeqa}{\end{eqnarray}}
\newcommand{\beqas}{\begin{eqnarray*}}
\newcommand{\eeqas}{\end{eqnarray*}}

\usepackage{algorithm}
\usepackage{algpseudocode}

\def\vgap{\vspace*{.1in}}
\newcommand{\nn}{\nonumber}

\title{Auto-conditioned
primal-dual hybrid gradient method and alternating direction method of multipliers \thanks{
This research was partially supported by Air Force Office of Scientific Research grant FA9550-22-1-0447.
Coauthors of this paper are listed according to the alphabetic order.
}}
\author{
Guanghui Lan
	\thanks{H. Milton Stewart School of Industrial and Systems Engineering, Georgia Institute of Technology, Atlanta, GA, 30332.
		(email: {\tt george.lan@isye.gatech.edu}).}
	\and
Tianjiao Li
	\thanks{H. Milton Stewart School of Industrial and Systems Engineering, Georgia Institute of Technology, Atlanta, GA, 30332.
		(email: {\tt tli432@gatech.edu}).}
}

\date{\today}

\begin{document} 

\maketitle

\begin{abstract}
Line search procedures are often employed in primal-dual methods for bilinear saddle point problems, 
especially when the norm of the linear operator is large or difficult to compute. In this paper, we demonstrate that line search is unnecessary by introducing a novel primal-dual method, the auto-conditioned primal-dual hybrid gradient (AC-PDHG) method, which achieves optimal complexity for solving bilinear saddle point problems.
AC-PDHG is fully adaptive to the linear operator, using only past iterates to estimate its norm. We further tailor AC-PDHG to solve linearly constrained problems, providing convergence guarantees for both the optimality gap and constraint violation. Moreover, we explore an important class of linearly constrained problems where both the objective and constraints decompose into two parts. By incorporating the design principles of AC-PDHG into the preconditioned alternating direction method of multipliers (ADMM), 
we propose the auto-conditioned alternating direction method of multipliers (AC-ADMM), which guarantees convergence based solely on one part of the constraint matrix and fully adapts to it, eliminating the need for line search. Finally, we extend both AC-PDHG and AC-ADMM to solve bilinear problems with an additional smooth term. By integrating these methods with a novel acceleration scheme, we attain optimal iteration complexities under the single-oracle setting.
 \end{abstract}

\section{Introduction}
The basic problem of interest in this paper is the bilinear saddle point problem:
\begin{align}\label{eq:bilinear_saddle_point}
\min_{x \in X} \max_{y \in Y} f(x) + \langle Ax, y \rangle - g(y),
\end{align}
where $X \subseteq \bbr^n$ and $Y\subseteq \bbr^m$ are closed convex sets, $A \in \bbr^{m \times n}$ denotes a linear mapping from $\bbr^n$ to $\bbr^m$, and $f: X \rightarrow \bbr$ and $g: Y \rightarrow \bbr$ are proper, lower semicontinuous, and convex functions. 
We assume there exists a pair of saddle point $z^*:=(x^*, y^*) \in Z:= X \times Y$ for problem \eqref{eq:bilinear_saddle_point}. In particular, when $Y= \bbr^m$ and $g^*$ is the Fenchel conjugate of $g$, problem \eqref{eq:bilinear_saddle_point} can be equivalently written as 
\begin{align}\label{eq:bilinear_saddle_point_equiv}
	\min_{x \in X} f(x) + g^*(Ax).
\end{align}
Another important class of problems that can be solved through \eqref{eq:bilinear_saddle_point} is the linearly constrained problem: 
\begin{align}\label{linear_constrained_problem}
\min_{x \in X}\left\{ f(x): Ax = b\right\},
\end{align}
which can be reduced 
to \eqref{eq:bilinear_saddle_point}
by using Lagrange duality.
We also consider a special case of the above linearly constrained problem:
\begin{align}\label{linear_two_parts_0}
    \min_{x \in X, w \in W}\left\{ F(x) + G(y): Bw - Kx = b\right\},
\end{align}
where $X \subseteq \bbr^{n_1}$ and $W \subseteq \bbr^{n_2}$ are closed convex sets, and $B \in \bbr^{m\times n_1}$ and $K\in \bbr^{m\times n_2}$ are linear operators. Clearly, when $B = I$ and $b = 0$, this problem is also equivalent to the composite problem \eqref{eq:bilinear_saddle_point_equiv}. 

Problems \eqref{eq:bilinear_saddle_point}-\eqref{linear_two_parts_0} have a broad range of applications in the field of signal and image processing, machine learning, and statistics; see., e.g., \cite{rudin1992nonlinear, tibshirani2005sparsity, jacob2009group, kolda2009tensor, yang2011alternating, tomioka2011statistical, bouwmans2016handbook} and references therein. Driven by these significant real-world applications, there has been considerable interest in the development of computationally efficient methods for bilinear saddle point and linearly constrained problems. 

\subsection{Smoothing scheme and primal-dual methods for \eqref{eq:bilinear_saddle_point}}
In a seminar work \cite{nesterov2005smooth}, Nesterov introduced a smoothing scheme by subtracting a strongly convex regularizer $\mu_d h(y)$ in \eqref{eq:bilinear_saddle_point}, leading to the formulation 
\begin{align}\label{eq:bilinear_saddle_point_smoothened}
\min_{x \in X} f(x) + H_{\mu_d}(x), \text{  where  } H_{\mu_d}(x):=\max_{y \in Y}  \langle Ax, y \rangle - g(y) - \mu_d h(y).
\end{align}
Then $H_{\mu_d}(x)$ becomes a smooth function with $\tfrac{\|A\|^2}{\mu_d}$-Lipschitz continuous gradient. Assuming $f$ is relatively simple, by choosing an appropriate $\mu_d$ and applying a (composite) accelerated gradient method \cite{nesterov1983method}, the overall iteration complexity can be bounded by
\begin{align}\label{optimal_complexity}
\mathcal{O}\left(\tfrac{\|A\|D_X D_Y}{\epsilon} \right),
\end{align}
 where $D_X$ and $D_Y$ are the diameters of the compact sets $X$ and $Y$. This complexity bound is optimal, as supported by lower bounds in \cite{nemirovsky1992information, nemirovski2004prox}. Nesterov's smoothing scheme has been extensively studied in the literature (e.g., \cite{nesterov2005excessive, auslender2006interior, lan2011primal, pena2008nash, tseng2008accelerated, lan2013bundle, lu2023unified}). 

In contrast to solving \eqref{eq:bilinear_saddle_point} by minimizing the smooth approximation function, primal-dual methods work directly with the saddle-point formulation. In their celebrated work \cite{chambolle2011first}, Chambolle and Pock introduced the primal-dual algorithm (PDA), also known as the primal-dual hybrid gradient (PDHG) method, which achieves the same optimal complexity \eqref{optimal_complexity} but with a stronger criterion based on the primal-dual gap function. When applied to the linearly constrained problem \eqref{linear_constrained_problem}, this method can provide guarantees on both the optimality gap and constraint violation. Further developments of PDHG and its variants can be found in \cite{pock2011diagonal, he2012on, chen2014optimal, liu2021acceleration}. Notably, Chen et al. \cite{chen2014optimal} considered the case when $f$ is smooth and introduced an accelerated variant of PDHG with optimal iteration complexity. 

Additionally, problem \eqref{eq:bilinear_saddle_point} can also be tackled using algorithms for monotone variational inequalities (VIs), e.g., \cite{korpelevich1983extrapolation, nemirovski2004prox, monteiro2010complexity, malitsky2020golden, kotsalis2022simple1, kotsalis2022simple}. Similar to the primal-dual methods mentioned above, VI algorithms, when applied to \eqref{eq:bilinear_saddle_point}, also perform updates in both the primal and dual spaces at each iteration. However, they can at most achieve the complexity of $\mathcal{O}\left(\tfrac{\|A\|(D_X^2 + D_Y^2)}{\epsilon} \right)$, which has a worse dependence on the diameters compared to that for the optimal bound \eqref{optimal_complexity}. 

\subsection{Alternating direction method of multipliers (ADMM) for \eqref{linear_two_parts_0}}
Now we switch our attention to problem \eqref{linear_two_parts_0}, which, by the method of Lagrangian multipliers, can be reformulated as the following saddle-point problem
\begin{align}\label{bilinear_two_lagrangian}
\min_{x \in X, w\in W} \max_{y \in \bbr^m} F(x) + G(w) + \langle K x - Bw +b, y\rangle.
\end{align}
This problem can also be solved using primal-dual methods like PDHG, but their convergence rate will depend on $\|[K, -B]\|$. Alternatively, one can further penalize the constraints by considering the augmented Lagrangian formulation
\begin{align}\label{bilinear_two_augmented_lagrangian}
\min_{x \in X, w\in W} \max_{y \in \bbr^m} F(x) + G(w) + \langle K x - Bw +b, y\rangle + \tfrac{\rho}{2}\|Kx - Bw +b\|^2,
\end{align}
where $\rho$ is a penalty parameter. The idea of this augmented Lagrangian method (ALM) originated from Hestenes \cite{hestenes1969multiplier} and Powell \cite{powell1969method}, and is also known as the method of multipliers; see textbooks \cite{bertsekas1982constrained, nocedal2006numerical}. One influential variant of ALM is the alternating direction method of multipliers (ADMM) \cite{gabay1976dual, glowinski1975approximation}, which updates $x_t$ and $w_t$ alternatively, followed by an update of the Lagrangian multiplier $y_t$ in each iteration,
\begin{align}
x_t &=  \arg \min_{x \in X} F(x) + \langle Kx - Bw_{t-1} + b, y_{t-1}\rangle + \tfrac{\rho}{2}\|Kx - Bw_{t-1} +b\|^2, \label{ADMM_x}\\
w_t &=  \arg \min_{w \in W} G(w) + \langle Kx_t - Bw + b, y_{t-1}\rangle + \tfrac{\rho}{2}\|Kx_2 - Bw +b\|^2, \label{ADMM_w}\\
y_t &=  \arg \min_{w \in W} y_{t-1} - \rho(B w_t - K x_t -b).\label{ADMM_y}
\end{align}
Clearly, ADMM involves more complicated subproblems in steps \eqref{ADMM_x} and \eqref{ADMM_w} compared to primal-dual methods like PDHG. However, several interesting studies \cite{he2012on, monteiro2013iteration, ouyang2013stochastic, LanBook2020}  have shown that the convergence rate of ADMM depends only on one part of the constraint mapping, either $B$ or $K$. 

Meanwhile, one can further reduce the per-iteration cost of ADMM by linearizing the term $\|Kx - Bw_{t-1} +b\|^2$ in step \eqref{ADMM_x}, resulting in the following update:
\begin{align}
x_t &=  \arg \min_{x \in X} F(x) + \langle Kx - Bw_{t-1} + b, y_{t-1}\rangle + \rho\langle K x_{t-1} - B w_{t-1} + b, K x \rangle + \tfrac{\eta}{2}\|x - x_{t-1}\|^2. \label{ADMM_x_preconditioned}
\end{align}
This variant is known as preconditioned ADMM (P-ADMM), and reduces to the primal-dual method \cite{chambolle2011first} when $B=I$. Notably, even when $B \neq I$, the convergence rate of P-ADMM depends only on one part of the constraint, $K$ (see, e.g., \cite{he2012on, ouyang2015accelerated}), but the stepsize parameter $\eta$ needs to be chosen based on $\|K\|$. Additionally, when $F$ is not a simple function, one may also linearize it in the subproblem \eqref{ADMM_x}, leading to the linearized ADMM (L-ADMM) \cite{ye2011computational, ye2011fast, chen2012fast, ouyang2015accelerated}. 

\subsection{Adaptive methods for bilinear saddle point problems}
It is important to note that, when implementing all these first-order methods mentioned above, we often need to estimate many problem parameters. Let us start with the primal-dual method \cite{chambolle2011first} for solving \eqref{eq:bilinear_saddle_point}.  in order to achieve the optimal complexity \eqref{optimal_complexity}, one needs to determine the stepsizes of the algorithm based on the problem parameters $\|A\|$, $D_X$, and $D_Y$. These constants can be unavailable or difficult to compute. For example, it is time-consuming to estimate $\|A\|$ when $A$ is a large-scale dense matrix. Additionally, even when the norm $\|A\|$ is computable, relying on this global constant can lead to overly conservative stepsize choices, and consequently slow down the convergence of the algorithm. Similar issues arise when implementing the P-ADMM or L-ADMM type algorithms for solving problem \eqref{linear_two_parts_0}, especially when $\|K\|$ is hard to estimate. 

To alleviate these concerns, there has been growing interest in the development of adaptive methods for bilinear saddle-point problems. 
One notable research direction is to incorporate line search strategies into primal-dual and ADMM algorithms. 
Specifically, Ouyang et al. \cite{ouyang2015accelerated} introduced line search strategies in P-ADMM and L-ADMM for solving \eqref{linear_two_parts_0}.
Malitsky and Pock \cite{malitsky2018first} proposed a line search for the primal-dual method \cite{chambolle2011first} and established the convergence guarantees in terms of the primal-dual gap of \eqref{eq:bilinear_saddle_point}. However, their complexity $\mathcal{O}(\tfrac{\|A\|(D_X^2 + D_Y^2)}{\epsilon})$ exhibits sub-optimal dependence on $D_X$ and $D_Y$. Meanwhile, although the line search procedure only involves matrix-vector multiplications, it can still be costly in practical implementations. In an effort to reduce the per-iteration cost, Malitsky \cite{malitsky2020golden} proposed the line search free golden ratio method for monotone variational inequalities. However, when applied to problem \eqref{eq:bilinear_saddle_point}, the method still suffers from sub-optimal dependence on $D_X$ and $D_Y$. Moreover, algorithms designed for VIs are generally less efficient than primal-dual methods when applied to bilinear saddle point problems. By incorporating the design ideas from both the golden ratio method \cite{malitsky2020golden} and the primal-dual algorithm~\cite{chambolle2011first}, Chang et al. \cite{chang2022golden} proposed the golden ratio primal-dual algorithm (GRPDA) to solve a linearly constrained problem equivalent to \eqref{eq:bilinear_saddle_point_equiv} with guarantees measured by optimality gap and constraint violation. 
However, GRPDA still requires a line search subroutine in each iteration. 
Vladarean et al. \cite{vladarean2021first} studied problem \eqref{eq:bilinear_saddle_point} in the setting where $f$ is (locally) smooth and proposed a line search free stepsize rule. However, the algorithm is not adaptive to $\|A\|$ and the rate remains sub-optimal with respect to both the Lipschitz constant of $\nabla f$ and the dependence on $D_X$ and $D_Y$.

In a recent work \cite{li2023simple}, we proposed an auto-conditioned fast gradient method (AC-FGM) and showed that it is uniformly optimal for smooth, weakly smooth, and nonsmooth convex optimization problems. The term ``auto-conditioned'' refers to the property that the method can automatically adapt to the problem parameters, such as the Lipschitz constant $L$, without the employment of line search procedures. This development, along with the existing literature on bilinear saddle point problems, motivates us to seek answers for the following questions:
\begin{center}
\emph{Is there an optimal, ``auto-conditioned'' primal-dual method for solving \eqref{eq:bilinear_saddle_point} that can adapt to $\|A\|$ without resorting to line search? And what about an ``auto-conditioned'' P-ADMM type algorithm for solving \eqref{linear_two_parts_0}?
}
\end{center}

\subsection{Contributions and organization}
We make three distinct contributions toward answering the above questions.

    Firstly, we introduce a novel primal-dual algorithm, called \emph{auto-conditioned primal-dual hybrid gradient method (AC-PDHG)} for solving the bilinear saddle point problem \eqref{eq:bilinear_saddle_point}. The major novelty of AC-PDHG lies in two aspects. First, to achieve the optimal rate \eqref{optimal_complexity}, AC-PDHG incorporates the idea of Nesterov's smoothing scheme, but still maintains a primal-dual update structure. Second, for the updates of the primal variable $x_t$, AC-PDHG maintains an additional intertwined sequence of search points $\{\bar x_t\}$ to serve as the ``prox-centers''. This design idea is inspired by the golden ratio method for VIs \cite{malitsky2020golden} and the auto-conditioned fast gradient method for convex optimization \cite{li2023simple}. Combining these two design ideas, AC-PDHG achieves the optimal rate \eqref{optimal_complexity} for the primal-dual gap function, with a line search free stepsize policy fully relies on the local estimates of $\|A\|$. The only input parameters required by AC-PDHG are the diameter of the dual feasible region $D_Y$ and the desired accuracy $\epsilon$. Additionally, we extend AC-PDHG to solve the linearly constrained problem \eqref{linear_constrained_problem}, establishing optimal $\mathcal{O}(\epsilon^{-1})$ complexity bounds for attaining a solution with $\epsilon$-optimality gap and $\epsilon$-constrained violation. Notably, when the primal feasible region is bounded and the diameter $D_X$ is known, the method can be implemented without the estimation of $\|y^*\|$.
    
    Secondly, we propose a novel P-ADMM type algorithm, called \emph{auto-conditioned ADMM (AC-ADMM)} for solving the linearly constrained problem \eqref{linear_two_parts_0}. Similar to AC-PDHG, AC-ADMM incorporates the idea from the smoothing scheme while maintaining alternative updates of the three variables $x_t$, $w_t$, and $y_t$ in each iteration. Specifically, the update rule for the primal variable $x_t$ is similar to \eqref{ADMM_x_preconditioned}, but with the prox-center $x_t$ replaced by $\bar x_t$, following the same design principle as AC-PDHG. The updates for the other primal variable $w_t$ involve the solution of a more complicated subproblem, like \eqref{ADMM_w}. However, at the price of the more complicated subproblem, AC-ADMM guarantees convergence in terms of the optimality gap and constraint violation, with dependence only on the matrix $K$, not $B$. Meanwhile, similar to AC-PDHG, AC-ADMM features a line search free stepsize policy that is fully adaptive to local estimates of $\|K\|$. 
    
    Thirdly, we extend AC-PDHG (and AC-ADMM) to the settings when $f$ (and $F$) are not prox-friendly but $L_f$-smooth (and $L_F$-smooth) functions.
    Specifically, we integrate AC-PDHG with the optimal, line search free AC-FGM algorithm \cite{li2023simple} proposed for smooth convex optimization. A key feature here is instead of linearizing the function $f$ at the search point $x_t$, we linearize it at another intertwined sequence $\widetilde x_t$. Our approach achieves the same complexity bound as \cite{chen2014optimal}:
    \begin{align*}
    \mathcal{O}\left(\sqrt{\tfrac{L_f D_X^2}{\epsilon}} +  \tfrac{\|A\|D_X D_Y}{\epsilon} \right),
    \end{align*}
    which is optimal in the single-oracle setting as supported by the lower bound in \cite{ouyang2021lower}.
    Here single-oracle model means for each search $(x,y)$, the first-order oracle returns $(\nabla f(x), Ax, A^T y)$.
    Meanwhile, the stepsize policy remains line search free and relies only on the local estimates of $L_f$ and $\|A\|$. We extend AC-ADMM to solve \eqref{linear_two_parts_0} with a smooth $F$ in a similar manner. 

The remaining sections of the paper are organized as follows. In Section \ref{sec:PDHG}, we introduce AC-PDHG for solving the bilinear saddle point problem \eqref{eq:bilinear_saddle_point} and the linearly constrained problem \eqref{linear_constrained_problem}. In Section \ref{sec:ADMM}, we propose AC-ADMM for solving problem \eqref{linear_two_parts_0}. Finally, in Section \ref{section:apdhg_aadmm}, we extend AC-PDHG and AC-ADMM to solve problems where the functions $f$ and $F$ are smooth. 

\subsection{Notation}
In this paper, we use the convention that $\tfrac{0}{0} = 0$, and $\tfrac{a}{0} = + \infty$ if $a > 0$. Unless stated otherwise, we let $\langle \cdot, \cdot\rangle$ denote the Euclidean inner product and $\|\cdot\|$ denote the corresponding Euclidean norm ($\ell_2$-norm). Given a vector $x \in \bbr^n$, we denote its $i$-th entry by $x^{(i)}$. Given a matrix $A$, we denote its $(i,j)$-th entry by $A_{i,j}$. We use $\|A\|$ to denote the spectral norm of matrix $A$. 

\section{An adaptive primal-dual hybrid gradient (PDHG) method}\label{sec:PDHG}
In this section, we consider solving problem~\eqref{eq:bilinear_saddle_point} in the simple but important setting, where both functions $f$ and $g$ are ``prox-friendly''. This means that the proximal-mapping problems 
\begin{align}
&\arg\min_{x \in X} \big\{\langle \xi, x\rangle + f(x) + \tfrac{\eta}{2}\|x' - x\|\big\}, \text{ where } \eta>0,~\xi\in \bbr^n, ~x' \in X, \label{eq:prox_x}\\
&\arg\min_{y \in Y} \big\{\langle \xi', y\rangle + g(y) + \tfrac{\eta'}{2}\|y' - y\|\big\}, \text{ where } \eta'>0,~\xi'\in \bbr^m, ~y' \in Y, \label{eq:prox_y}
\end{align}
can be easily solved, either in closed form or through some efficient computational procedures. For convenience, we denote $Z := X \times Y$. For two primal-dual pairs $z:= (x,y), \bar z:= (\bar x, \bar y) \in Z$, we define the primal-dual gap function as
\begin{align}\label{def_Q_mu}
Q(\bar z, z):= f(\bar x) + \langle A \bar x, y\rangle - g(y) - \left[f( x) + \langle A x, \bar y\rangle - g(\bar y) \right].
\end{align}
Note that by definition $z^* = (x^*, y^*)\in Z$ is a saddle point of Ineq.~\eqref{eq:bilinear_saddle_point} if and only if for any $z \in Z$, 
\begin{align*}
f(x^*) + \langle Ax^*, y \rangle - g(y) \leq f(x^*) + \langle Ax^*, y^* \rangle - g(y^*) \leq f(x) + \langle Ax, y^* \rangle - g(y^*).
\end{align*}
It then follows that $z^*$ is a saddle point of Ineq.~\eqref{eq:bilinear_saddle_point} if and only if $Q(z^*, z) \leq 0$ for any $z \in Z$. Consequently, $\max_{z \in Z}Q(\bar z, z)$ is a valid stopping criterion for the saddle point problem \eqref{eq:bilinear_saddle_point} when $Z$ is a bounded set.

Next, we introduce the Auto-Conditioned Primal-Dual Hybrid Gradient Method (AC-PDHG). AC-PDHG involves both primal and dual updates in each iteration. In the primal steps, two sequences of search points are maintained, i.e., $\{x_t\}$, and $\{\bar x_t\}$. The sequence $\{x_t\}$ is the output of the prox-mapping subproblem \eqref{primal_prox-mapping}, while the sequence $\{\bar x_t\}$ is a moving average of $\{x_t\}$ and serves as the prox-centers. This idea was first introduced in the golden ratio method \cite{malitsky2020golden} for variational inequalities, and later adopted in the auto-conditioned fast gradient method (AC-FGM) \cite{li2023simple} for convex optimization. For the dual variables, we update $\{y_t\}$ through another prox-mapping step \eqref{dual_prox-mapping}, which includes an additional regularization term $\tfrac{\mu_d}{2}\|\widetilde y_0 - y\|^2$. These new features in both primal and dual updates distinguish the AC-PDHG method from the classical PDHG method \cite{chambolle2011first}, and play a critical role in  achieving the optimal rate for solving the bilinear saddle point problem \eqref{eq:bilinear_saddle_point} without resorting to line search.
\begin{algorithm}[H]\caption{Auto-Conditioned Primal-Dual Hybrid Gradient Method (AC-PDHG)}\label{alg:ac_primal_dual}
	\begin{algorithmic}
		\State{\textbf{Input}: initial point $x_0  = \bar x_0 \in X$, $\widetilde y_0 \in Y$, nonnegative parameters $\beta_t \in (0, 1)$, $\eta_t \in \bbr_+$, $\tau_t \in \bbr_+$, and $\mu_d > 0$. 
  Let 
  \begin{align}
  y_0 = \arg \min_{y \in Y} \left\{\langle - A x_0, y \rangle + g(y) + \tfrac{\mu_d}{2}\|\widetilde y_0 - y\|^2  \right\}.\label{prox_update_0}
  \end{align}}
		\For{$t=1,\cdots, k$}
		\State{
  \begin{align}
  x_t &= \arg \min_{x \in X} \left\{\eta_t[\langle A^\top y_{t-1}, x \rangle +f(x)] + \tfrac{1}{2}\|\bar x_{t-1} - x\|^2\right\}, \label{primal_prox-mapping}\\
  \bar x_t &= (1-\beta_t) \bar x_{t-1} + \beta_t x_t, \label{primal_prox-center}\\
  y_t & = \arg \min_{y \in Y} \left\{\langle -Ax_t, y\rangle + g(y)+ \tfrac{\mu_d}{2}\|\widetilde y_0 - y\|^2  + \tfrac{\tau_t}{2}\|y_{t-1} - y\|^2\right\}. \label{dual_prox-mapping}
  \end{align}
  }
		\EndFor
	\end{algorithmic}
\end{algorithm}


Our goal is to ensure that AC-PDHG can be implemented without the computation or estimation of the global parameter $\|A\|$. Instead, we compute a local estimate of $\|A\|$ at each iteration, defined as
\begin{align}
L_{A, t} := \tfrac{\|A^\top (y_t - y_{t-1})\|}{\|y_t - y_{t-1}\|}.\label{def_L_A_t}
\end{align}
Then we use the local estimates obtained in the past iterations to determine the algorithm parameters, e.g., $\eta_t$, and $\tau_t$. 
In the next two subsections, we propose the stepsize policy for AC-PDHG and establish its convergence guarantees for solving the bilinear saddle point problem and the linearly constrained problem stated in \eqref{eq:bilinear_saddle_point} and \eqref{linear_constrained_problem} respectively.



\subsection{AC-PDHG for bilinear saddle point problems}\label{subsec:ac-pdhg}
To establish the convergence guarantees for Algorithm~\ref{alg:ac_primal_dual}, we first present the following lemma which characterizes the optimality condition of step \eqref{dual_prox-mapping}, and an important relationship between the primal and dual variables in the first iteration. 
\begin{lemma}\label{dual_three-point_lemma}
Let $\{x_t\}, \{\bar x_t\}$ and $\{y_t\}$ be generated by Algorithm~\ref{alg:ac_primal_dual}, we have for $t \geq 1$,
\begin{align}
&\langle - A x_{t}, y_{t} - y\rangle + \tfrac{\mu_d}{2}\|\widetilde y_0 - y_{t}\|^2 - \tfrac{\mu_d}{2}\|\widetilde y_0 - y\|^2 + g(y_{t}) - g(y) + \tfrac{\mu_d + \tau_{t}}{2}\|y_{t} - y\|^2\nn\\
& \leq \tfrac{\tau_t}{2} \|y_{t-1} - y\|^2 - \tfrac{\tau_t}{2} \|y_{t-1} - y_t\|^2. \label{eq_dual_three-point_lemma}
\end{align}
Moreover, if $\tau_1 = 0$, we have
\begin{align}\label{eq_bound_y_1_y_0}
\|y_1 - y_0\| \leq \tfrac{L_{A, 1}}{\mu_d}\|x_1 - x_0\|.
\end{align}
\end{lemma}
\begin{proof}
First, Ineq. \eqref{eq_dual_three-point_lemma} is known as the ``three-point'' lemma, which follows from the optimality condition of \eqref{dual_prox-mapping}; see a formal proof in Lemma 3.5 of \cite{LanBook2020}. Applying Ineq.~\eqref{eq_dual_three-point_lemma} to the case $t = 1$, noting that $\tau_1 = 0$, and setting $y = y_0$, we obtain
\begin{align*}
\langle - Ax_1, y_1 - y_0 \rangle + \tfrac{\mu_d}{2}\|\widetilde y_0 - y_1\|^2 - \tfrac{\mu_d}{2}\|\widetilde y_0 - y_0\|^2 + g(y_1) - g(y_0) + \tfrac{\mu_d}{2} \|y_1-y_0\|^2 \leq 0.
\end{align*}
Similarly, applying Ineq.~\eqref{eq_dual_three-point_lemma} to the case $t = 0$ with $\tau_0=0$, and setting $y = y_0$ give us
\begin{align*}
\langle - Ax_0, y_0 - y_1 \rangle + \tfrac{\mu_d}{2}\|\widetilde y_0-y_0\|^2 - \tfrac{\mu_d}{2}\|\widetilde y_0 - y_1\|^2 + g(y_0) - g(y_1) + \tfrac{\mu_d}{2} \|y_0-y_1\|^2 \leq 0.
\end{align*}
By summing up the above two inequalities, rearranging the terms, and using the definition of $L_{A,1}$ in \eqref{def_L_A_t}, we have 
\begin{align*}
\langle A^\top (y_1 - y_0), x_1 - x_0\rangle \geq \mu_d \|y_1 - y_0\|^2 = \tfrac{\mu_d}{L_{A, 1}^2}\|A^\top (y_1 - y_0)\|^2.
\end{align*}
Then using Young's inequality provides us with
\begin{align*}
\tfrac{\mu_d}{L_{A, 1}^2}\|A^\top (y_1- y_0)\|^2 \leq \tfrac{\mu_d}{2 L^2_{A,1}}\|A^\top (y_1- y_0)\|^2 + \tfrac{L_{A,1}^2}{2\mu_d}\|x_1-x_0\|^2,
\end{align*}
from which our results follow by rearranging the terms and further applying the definition of $L_{A, 1}$. 
\end{proof}
\vgap

The following proposition states some important conditions on the stepsize parameters, and characterizes the recursive relationship for the iterates generated by Algorithm \ref{alg:ac_primal_dual}. 
\begin{proposition}\label{proposition_1}
Let $\{x_t\}, \{\bar x_t\}$ and $\{y_t\}$ be generated by Algorithm~\ref{alg:ac_primal_dual}. Assume the parameters $\{\tau_t\}$, $\{\eta_t\}$, and $\{\beta_t\}$ satisfy
\begin{align}
\tau_1 &= 0, ~\beta_1 = 0,~\beta_t = \beta > 0,~t\geq 2 \label{cond_3}\\
    \eta_2 &\leq \min\left\{{(1-\beta)\eta_1}, \tfrac{\mu_d}{4 L^2_{A,1}} \right\},\label{cond_1}\\ 
    \eta_t &\leq \min\left\{2(1-\beta)^2 \eta_{t-1},  \tfrac{\tau_{t-2}+\mu_d}{ \tau_{t-1}}\eta_{t-1}, \tfrac{\tau_{t-1}}{4L_{A, t-1}^2}\right\}, ~t \geq 3.
    \label{cond_4}
\end{align}
Then for any $z = (x,y) \in Z$, we have 
\begin{align}\label{eq:proposition_1}
&\left(\tsum_{t=1}^k \eta_t\right) \cdot \left(Q(\widehat z_k, z) -  \tfrac{\mu_d}{2}\|\widetilde y_0 - y\|^2 +  \tfrac{\mu_d}{2}\|\widetilde y_0 - \widehat y_k\|^2  + \tfrac{\mu_d}{2}\|\widetilde y_k - y\|^2\right) \nn\\
 &\leq \tfrac{1}{2\beta}\|x_0 - x\|^2 - \tfrac{1}{2\beta}\|\bar x_{k+1} - x\|^2 -  \tfrac{\eta_2}{2\eta_{1}} \left[\|x_1 - \bar x_1\|^2 + \|x_2 - x_1\|^2 \right] + \tsum_{t=2}^{k+1}\Delta_{t},
\end{align}
where 
\begin{align}\label{def_Delta_t}
\Delta_t &:= \eta_t \langle A^\top (y_{t-1} - y_{t-2}), x_{t-1} - x_t\rangle - \tfrac{\eta_t \tau_{t-1}}{2}\|y_{t-2} - y_{t-1}\|^2 - \tfrac{1}{2}\|x_t - \bar x_{t-1}\|^2,
\end{align}
and 
\begin{align}
\widehat x_k &:= \tfrac{\sum_{t=1}^{k}\eta_{t+1} x_{t}}{\sum_{t=1}^{k}\eta_{t+1}}, \quad \widehat y_k := \tfrac{\sum_{t=1}^{k}\eta_{t+1} y_{t}}{\sum_{t=1}^{k}\eta_{t+1}}, \quad \widehat z_k := (\widehat x_k, \widehat y_k),\label{def_hat_x_k_y_k}\\
\widetilde y_k &:= \tfrac{\sum_{t=1}^{k-1}(\eta_{t+1}(\mu_d + \tau_t) -\eta_{t+2}\tau_{t+1})y_t + \eta_{k+1}(\mu_d + \tau_k)y_k}{\sum_{t=1}^k \eta_{t+1}}.\label{def_tilde_y_k}
\end{align}
\end{proposition}
\begin{proof}
First, by the optimality condition of step \eqref{primal_prox-mapping} and the convexity of $f$, we have for $t \geq 1$,
\begin{align}\label{eq1}
\langle \eta_t A^\top y_{t-1}+ x_t - \bar x_{t-1}, x - x_t\rangle \geq \eta_t[f(x_t) - f(x)], \quad \forall x \in X,
\end{align}
and consequently for $t \geq 2$,
\begin{align}\label{eq2}
\langle \eta_{t-1} A^\top y_{t-2} + x_{t-1} - \bar x_{t-2}, x_t - x_{t-1}\rangle \geq \eta_{t-1}[f(x_{t-1}) - f(x_t)].
\end{align}
By \eqref{primal_prox-center}, we have the relationship $x_{t-1} - \bar x_{t-2} = \tfrac{1}{1-\beta_{t-1}} (x_{t-1}-\bar x_{t-1})$, so we can rewrite \eqref{eq2} as
\begin{align}\label{eq3}
\langle \eta_t A^\top y_{t-2} + \tfrac{\eta_t}{\eta_{t-1}(1-\beta_{t-1})}(x_{t-1} - \bar x_{t-1}), x_t - x_{t-1}\rangle \geq \eta_t[f(x_{t-1}) - f(x_t)]. 
\end{align}
Summing up \eqref{eq1} and \eqref{eq3} and rearranging the terms give us for $t \geq 2$,
\begin{align*}
&\eta_t \langle A^\top y_{t-1}, x - x_{t-1}\rangle + \eta_t \langle A^\top (y_{t-1} - y_{t-2}), x_{t-1}-x_t \rangle\\ 
&+ \langle x_t - \bar x_{t-1}, x - x_t\rangle + \tfrac{\eta_t}{\eta_{t-1}(1-\beta_{t-1})} \langle x_{t-1}-\bar x_{t-1}, x_t - x_{t-1} \rangle \geq \eta_t[f(x_{t-1}) - f(x)].
\end{align*}
By utilizing the fact that $2\langle x - y, z - x\rangle = \|y-z\|^2 - \|x-y\|^2 - \|x-z\|^2$ for the last two terms in the left-hand side of the above inequality, we obtain
\begin{align}\label{eq4}
&\eta_t \langle A^\top y_{t-1}, x - x_{t-1}\rangle + \eta_t \langle A^\top (y_{t-1} - y_{t-2}), x_{t-1}-x_t \rangle\nn\\ 
&+ \tfrac{1}{2}\|\bar x_{t-1}-x\|^2 - \tfrac{1}{2}\|x_{t}-\bar x_{t-1}\|^2 -\tfrac{1}{2}\|x_t - x\|^2\nn\\
&+ \tfrac{\eta_t}{2\eta_{t-1}(1-\beta_{t-1})} \left[\|x_t - y _{t-1}\|^2 - \|x_{t-1} - \bar x_{t-1}\|^2 - \|x_t - x_{t-1}\|^2\right] \geq \eta_t[f(x_{t-1}) - f(x)],
\end{align}
Meanwhile, \eqref{primal_prox-center} also indicates that
\begin{align}\label{eq5}
    \|x_t - x\|^2 &= \|\tfrac{1}{\beta_t} (\bar x_t - x) - \tfrac{1-\beta_t}{\beta_t}(\bar x_{t-1} - x)\|^2\nn\\
    &\overset{(i)}=\tfrac{1}{\beta_t}\|\bar x_t -x\|^2 - \tfrac{1-\beta_t}{\beta_t}\|\bar x_{t-1} - x\|^2 + \tfrac{1}{\beta_t} \cdot \tfrac{1-\beta_t}{\beta_t} \|\bar x_t - \bar x_{t-1}\|^2\nn\\
    &= \tfrac{1}{\beta_t}\|\bar x_t -x\|^2 - \tfrac{1-\beta_t}{\beta_t}\|\bar x_{t-1} - x\|^2 + (1-\beta_t) \|x_t - \bar x_{t-1}\|^2,
\end{align}
where step (i) follows from the fact that $\|\alpha a + (1-\alpha) b\|^2 = \alpha \|a\|^2 + (1-\alpha)\|b\|^2 - \alpha (1-\alpha)\|a-b\|^2,~\forall {\alpha} \in \bbr$. By combining {Ineq.} \eqref{eq4} and {Eq.} \eqref{eq5} and rearranging the terms, we obtain
\begin{align}\label{eq6}
&\eta_t \langle A^\top y_{t-1}, x_{t-1} - x\rangle + \eta_t[f(x_{t-1}) - f(x)]  + \tfrac{\eta_t}{2\eta_{t-1}(1-\beta_{t-1})} \left[\|x_{t-1} - \bar x_{t-1}\|^2 + \|x_t - x_{t-1}\|^2 \right]\nn\\
&\leq \tfrac{1}{2\beta_t}\|\bar x_{t-1} - x\|^2 - \tfrac{1}{2\beta_t} \|\bar x_t - x\|^2 +  \eta_t \langle A^\top (y_{t-1} - y_{t-2}), x_{t-1}-x_t \rangle\nn\\
&\quad - \left[ \tfrac{1}{2} + \tfrac{1 - \beta_t}{2} - \tfrac{\eta_t}{2 \eta_{t-1}(1-\beta_{t-1})} \right]\|x_t - \bar x_{t-1}\|^2.
\end{align}
When $t \geq 3$, by utilizing the fact that $\|x_{t-1} - \bar x_{t-1}\|^2 + \|x_t - x_{t-1}\|^2 \geq \tfrac{1}{2}\|x_t - \bar x_{t-1}\|^2$ and rearranging the terms in Ineq.~\eqref{eq6}, we obtain
\begin{align}\label{eq6_prime}
&\eta_t \langle A^\top y_{t-1}, x_{t-1} - x\rangle + \tfrac{1}{2\beta_t} \|\bar x_t - x\|^2 \nn\\
&\leq \tfrac{1}{2\beta_t}\|\bar x_{t-1} - x\|^2 +  \eta_t \langle A^\top (y_{t-1} - y_{t-2}), x_{t-1}-x_t \rangle - \left[ \tfrac{1}{2} + \tfrac{1 - \beta_t}{2} - \tfrac{\eta_t}{4 \eta_{t-1}(1-\beta_{t-1})} \right]\|x_t - \bar x_{t-1}\|^2\nn\\
&\overset{(i)}\leq \tfrac{1}{2\beta_t}\|\bar x_{t-1} - x\|^2 +  \eta_t \langle A^\top (y_{t-1} - y_{t-2}), x_{t-1}-x_t \rangle- \tfrac{1}{2}\|x_t - \bar x_{t-1}\|^2,
\end{align}
where step (i) follows from $\eta_t \leq 2(1-\beta)^2 \eta_{t-1}$ in \eqref{cond_4}. 
Combining Ineqs.~\eqref{eq6_prime} and \eqref{eq_dual_three-point_lemma}, and rearranging the terms, we have for $t \geq 3$ and any $x\in X,~y \in Y$,
\begin{align}
&\eta_t \left[f(x_{t-1}) + \langle A x_{t-1}, y \rangle - \tfrac{\mu_d}{2}\|\widetilde y_0 - y\|^2 - g(y) - f(x) - \langle Ax, y_{t-1}\rangle + \tfrac{\mu_d}{2}\|\widetilde y_0 - y_{t-1}\|^2 + g(y_{t-1}) \right]\nn\\
& \leq \tfrac{1}{2\beta_t}\|\bar x_{t-1} - x\|^2 - \tfrac{1}{2\beta_t}\|\bar x_{t} - x\|^2 + \tfrac{\eta_t\tau_{t-1}}{2}\|y_{t-2}- y\|^2 - \tfrac{\eta_t(\mu_d + \tau_{t-1})}{2}\|y_{t-1} - y\|^2\nn\\
& \quad + \eta_t \langle A^\top (y_{t-1} - y_{t-2}), x_{t-1}-x_t \rangle - \tfrac{\eta_t \tau_{t-1}}{2}\|y_{t-2} - y_{t-1}\|^2- \tfrac{1}{2}\|x_t - \bar x_{t-1}\|^2.\label{eq9}
\end{align}
When $t=2$, we use the condition $\beta_1 = 0$ in \eqref{eq6} to obtain
\begin{align}
&\eta_2 \langle A^\top y_{1}, x_{1} - x\rangle + \eta_2[f(x_{1}) - f(x)]  + \tfrac{\eta_t}{2\eta_{t-1}} \left[\|x_{1} - \bar x_{1}\|^2 + \|x_2 - x_{1}\|^2 \right]\nn\\
&\leq \tfrac{1}{2\beta_2}\|\bar x_{1} - x\|^2 - \tfrac{1}{2\beta_2} \|\bar x_2 - x\|^2 +  \eta_2 \langle A^\top (y_{1} - y_{0}), x_{1}-x_2 \rangle - \left[ \tfrac{1}{2} + \tfrac{1 - \beta_t}{2} - \tfrac{\eta_t}{2 \eta_{t-1}} \right]\|x_t - \bar x_{t-1}\|^2\nn\\
&\overset{(i)}\leq \tfrac{1}{2\beta_2}\|\bar x_{1} - x\|^2 - \tfrac{1}{2\beta_2} \|\bar x_2 - x\|^2 +  \eta_2 \langle A^\top (y_{1} - y_{0}), x_{1}-x_2 \rangle -\tfrac{1}{2}\|x_2 - \bar x_{1}\|^2,\label{eq9_3}
\end{align}
where step (i) follows from the condition $\eta_2 \leq (1-\beta)\eta_1$ in \eqref{cond_1}.
Then, by combining the above inequality with Ineq.~\eqref{eq_dual_three-point_lemma} for the case $t=1$, and invoking that $\tau_1=0$, we have
\begin{align}\label{eq9_2}
&\eta_2 \left[f(x_1) + \langle A x_{1}, y \rangle - \tfrac{\mu_d}{2}\|\widetilde y_0 - y\|^2 - g(y) - f(x) - \langle Ax, y_{1}\rangle + \tfrac{\mu_d}{2}\|\widetilde y_0 - y_{1}\|^2 + g(y_{1}) \right]\nn\\
&   \quad + \tfrac{\eta_2}{2\eta_{1}} \left[\|x_1 - \bar x_1\|^2 + \|x_2 - x_1\|^2 \right] + \tfrac{\eta_2\mu_d}{2}\|y_{1} - y\|^2+ \tfrac{1}{2\beta_2} \|\bar x_2 - x\|^2\nn\\
&\overset{(i)}\leq \tfrac{1}{2\beta_2}\|\bar x_{1} - x\|^2 +  \eta_2 \langle A^\top (y_1-y_0), x_{1} - x_2 \rangle-\tfrac{1}{2}\|x_2 - \bar x_{1}\|^2.
\end{align}
 By taking the summation of Ineq.~\eqref{eq9_2} and the telescope sum of Ineq.~\eqref{eq9} for $t=3,..., k+1$, and noting that $\beta_t = \beta$ for $t \geq 2$, we obtain
\begin{align}\label{eq_final}
& \tsum_{t=1}^k \eta_{t+1} \left[f(x_t)  +\langle A x_{t}, y \rangle - \tfrac{\mu_d}{2}\|\widetilde y_0 - y\|^2 - g(y)- f(x) - \langle Ax, y_{t}\rangle + \tfrac{\mu_d}{2}\|\widetilde y_0 - y_{t}\|^2 + g(y_{t}) \right]+ \tfrac{1}{2\beta}\|\bar x_{k+1} - x\|^2\nn\\
&  + \tfrac{\eta_2}{2\eta_{1}} \left[\|x_1 - \bar x_1\|^2 + \|x_2 - x_1\|^2 \right] + \tfrac{\eta_{k+1}(\mu_d + \tau_k)}{2}\|y_k - y\|^2 + \tsum_{t=1}^{k-1}\tfrac{\eta_{t+1}(\mu_d + \tau_t) - \eta_{t+2}\tau_{t+1}}{2}\|y_t - y\|^2\nn\\
& \leq \tfrac{1}{2\beta}\|x_0 - x\|^2 + \tsum_{t=2}^{k+1}\Delta_{t},
\end{align}
where $\Delta_t$ is defined in \eqref{def_Delta_t}. 
By utilizing the definition of $\widehat x_k$ and $\widehat y_k$ in \eqref{def_hat_x_k_y_k} and the convexity of $\|\cdot\|^2$, $f$ and $g$, we obtain
\begin{align*}
&\tsum_{t=1}^k \eta_{t+1} \left[f(x_t)  + \langle A x_{t}, y \rangle - \tfrac{\mu_d}{2}\|\widetilde y_0-y\|^2 - g(y) - f(x) - \langle Ax, y_{t}\rangle + \tfrac{\mu_d}{2}\|\widetilde y_0-y_{t}\|^2 + g(y_{t}) \right] \\
&\geq (\tsum_{t=1}^k \eta_{t+1}) \cdot \left[f(\widehat x_k)  +\langle A \widehat x_{k}, y \rangle - \tfrac{\mu_d}{2}\|\widetilde y_0-y\|^2 - g(y) - f(x) - \langle Ax, \widehat y_{k}\rangle + \tfrac{\mu_d}{2}\|\widetilde y_0-\widehat y_{k}\|^2 + g(\widehat y_{k}) \right].
\end{align*}
Then, invoking the definition of $Q(\cdot, \cdot)$ in \eqref{def_Q_mu}, we have
\begin{align*}
& \tsum_{t=1}^k \eta_{t+1} \left[f(x_t)  +\langle A x_{t}, y \rangle - \tfrac{\mu_d}{2}\|\widetilde y_0 - y\|^2 - g(y) - f(x) - \langle Ax, y_{t}\rangle + \tfrac{\mu_d}{2}\|\widetilde y_0 - y_{t}\|^2 + g(y_{t}) \right]\nn\\
& \geq (\tsum_{t=1}^k \eta_{t+1}) \cdot \big(Q(\widehat z_k, z) -  \tfrac{\mu_d}{2}\|\widetilde y_0 - y\|^2 +  \tfrac{\mu_d}{2}\| \widetilde y_0 - \widehat y_k\|^2\big).
\end{align*}
Similarly, we use the convexity of $\|\cdot\|^2$ and the definition of $\widetilde y_k$ to obtain
\begin{align}\label{use_tilde_y_k}
\tfrac{\eta_{k+1}(\mu_d + \tau_k)}{2}\|y_k - y\|^2 + \tsum_{t=1}^{k-1}\tfrac{\eta_{t+1}(\mu_d + \tau_t) - \eta_{t+2}\tau_{t+1}}{2}\|y_t - y\|^2 \geq (\tsum_{t=1}^k\eta_{t+1}) \cdot \tfrac{\mu_d}{2}\|\widetilde y_k - y\|^2.
\end{align}
Substituting the above two inequalities into Ineq.~\eqref{eq_final}, we obtain the desired result.
\end{proof}
\vgap

In the next theorem, we further bound the terms $\tsum_{t=2}^k \Delta_t$ to establish the convergence guarantees for Algorithm~\ref{alg:ac_primal_dual} in terms of the primal-dual gap function $Q(\widehat z_k, z)$. 
\begin{theorem}\label{main_theorem}
Let $\{x_t\}, \{\bar x_t\}$ and $\{y_t\}$ be generated by Algorithm~\ref{alg:ac_primal_dual} with
parameters satisfying \eqref{cond_3}-\eqref{cond_4}. 
Then we have for any $k \geq 2$ and any $z = (x,y) \in Z$, 
\begin{align}\label{eq_main_theorem}
&Q(\widehat z_k, z)- \tfrac{\mu_d}{2}\|\widetilde y_0 - y\|^2 + \tfrac{\mu_d}{2}\|\widetilde y_0 -\widehat y_k \|^2+ \tfrac{\mu_d}{2}\|\widetilde y_k - y\|^2 \nn\\
 &\leq \tfrac{1}{\sum_{t=1}^k \eta_t}\left[\tfrac{\|x_0 - x\|^2}{2\beta} + \left(\tfrac{5\eta_2 L^2_{A, 1}}{4\mu_d}  - \tfrac{\eta_2}{2\eta_1}\right)\|x_1 - x_0\|^2 - \tfrac{\|\bar x_{k+1} - x\|^2}{2\beta}\right].
\end{align}
\end{theorem}
\begin{proof}
 First notice that Ineq.~\eqref{eq:proposition_1} holds in view of Proposition~\ref{proposition_1}. We attempt to provide an upper bound for $\Delta_t$ used in this relation. For the case $t\geq 3$, we have
\begin{align}\label{bound_delta_t}
\Delta_t &\overset{(i)}\leq \eta_t L_{A, t-1}\|y_{t-1} - y_{t-2}\|\|x_{t-1}-x_t\| - \tfrac{\eta_t \tau_{t-1}}{2}\|y_{t-2}-y_{t-1}\|^2 - \tfrac{1}{2}\|x_t - \bar x_{t-1}\|^2\nn\\
& \overset{(ii)}\leq \tfrac{\eta_t L_{A, t-1}^2}{2 \tau_{t-1}} \|x_{t-1} - x_t\|^2 - \tfrac{1}{2}\|x_t - \bar x_{t-1}\|^2 \nn\\
& \overset{(iii)}\leq (1-\beta)^2\cdot \tfrac{\eta_t L_{A, t-1}^2}{\tau_{t-1}}\|x_{t-1}-\bar x_{t-2}\|^2 - \left( \tfrac{1}{2} - \tfrac{\eta_t L_{A, t-1}^2}{\tau_{t-1}}\right)\|x_t - \bar x_{t-1}\|^2\nn\\
& \overset{(iv)}\leq \tfrac{1}{4}\|x_{t-1}-\bar x_{t-2}\|^2 - \tfrac{1}{4}\|x_t - \bar x_{t-1}\|^2,
\end{align}
where step (i) follows from Cauchy-Schwarz inequality, step (ii) follows from Young's inequality, step (iii) follows from the fact that 
\begin{align*}
\|x_{t-1}- x_t\|^2 &= \|x_{t-1} - \bar x_{t-1}+ \bar x_{t-1} - x_t\|^2 = \|(1-\beta)(x_{t-1} - \bar x_{t-2}) + \bar x_{t-1} - x_t\|^2\\
&\leq 2(1-\beta)^2\|x_{t-1} - \bar x_{t-2}\|^2 + 2\|x_t - \bar x_{t-1}\|^2,
\end{align*}
and step (iv) follows from the conditions $\eta_t \leq \tfrac{\tau_{t-1}}{4L_{A, t-1}^2}$ in \eqref{cond_4}. Next, we bound $\Delta_2$. We have
\begin{align}\label{bound_delta_2}
\Delta_2 & \overset{(i)}\leq \eta_2 L_{A,1} \|y_1-y_0\|\|x_1 - x_2\|  - \tfrac{1}{2}\|x_2 -\bar x_1\|^2\nn\\
&\overset{(ii)}\leq \tfrac{\eta_2 L^2_{A, 1}}{\mu_d}\|x_1 - x_0\|(\|x_1 - x_0\| + \|x_2 - x_0\|) -\tfrac{1}{2}\|x_2 - x_0\|^2\nn\\
&\overset{(iii)}\leq \tfrac{5\eta_2 L^2_{A, 1}}{4\mu_d} \|x_1 - x_0\|^2+ \left(\tfrac{\eta_2 L^2_{A, 1}}{\mu_d} -\tfrac{1}{2}\right)\|x_2 - x_0\|^2\nn\\
&\overset{(iv)}\leq \tfrac{5\eta_2 L^2_{A, 1}}{4\mu_d}\|x_1 - x_0\|^2-\tfrac{1}{4}\|x_2 - \bar x_1\|^2,
\end{align}
where step (i) follows from Cauchy-Schwarz inequality and the definition of $L_{A,1}$, step (ii) follows from Ineq.~\eqref{eq_bound_y_1_y_0} and the fact that  $x_0 = \bar x_1$, step (iii) follows from Young's inequality, and step (iv) follows from the condition $\eta_2 \leq \tfrac{\mu_d}{4 L^2_{A,1}}$. By substituting Ineqs.~\eqref{bound_delta_t} and \eqref{bound_delta_2} into Ineq.~\eqref{eq:proposition_1} and rearranging the terms, we obtain 
\begin{align*}
&\left(\tsum_{t=1}^k \eta_t\right) \cdot \left(Q(\widehat z_k, z) - \tfrac{\mu_d}{2}\|\widetilde y_0 -y \|^2 + \tfrac{\mu_d}{2}\|\widetilde y_0 -\widehat y_k \|^2 + \tfrac{\mu_d}{2}\|\widetilde y_k - y\|^2\right) + \tfrac{1}{2\beta}\|\bar x_{k+1} - x\|^2\nn\\
 &\leq \tfrac{1}{2\beta}\|x_0 - x\|^2 + \left(\tfrac{5\eta_2 L^2_{A, 1}}{4\mu_d} - \tfrac{\eta_2}{2\eta_1}\right)\|x_1 - x_0\|^2 - \tfrac{1}{4}\|x_{k+1} - \bar x_k\|^2,
\end{align*}
which completes the proof.
\end{proof}
\vgap

From the stepsize conditions in~\eqref{cond_1} and \eqref{cond_4}, it is not hard to see that the parameters $\eta_t$ and $\tau_t$ can be determined without the employment of a line search procedure. In the next Corollary, we provide a concrete stepsize policy for AC-PDHG which satisfies these conditions, and establish the corresponding convergence guarantees for this algorithm when applied for solving problem \eqref{eq:bilinear_saddle_point}.

\begin{corollary} \label{main_corollary_2}
In the premise of Theorem~\ref{main_theorem}, suppose $\tau_1=0,~ \tau_2 = \mu_d$, and $\beta \in (0,  1 -{\tfrac{\sqrt{6}}{3}}]$. If $\eta_2 = \min \left\{{(1 - \beta) \eta_1},  \tfrac{\mu_d}{4 L_{A, 1}^2}\right\}$, and for $t \geq 3$
\begin{align}
&\eta_t = \min \left\{ \tfrac{4}{3}\eta_{t-1}, \tfrac{\tau_{t-2}+\mu_d}{\tau_{t-1}}\cdot \eta_{t-1}, \tfrac{\tau_{t-1}}{4L^2_{A, t-1}} \right\},\label{def_eta_t}\\
&\tau_t = \tau_{t-1} + \tfrac{\mu_d}{2}\left[\alpha +(1-\alpha) \eta_t\cdot\tfrac{4 L_{A, t-1}^2}{\tau_{t-1}}\right]
\label{def_tau_t}
\end{align}
for some absolute constant in $\alpha \in (0,1]$,
then we have for $t \geq 2$,
\begin{align}\label{eta_lower_bound_2}
\eta_t \geq \tfrac{(3 + \alpha(t-3))\mu_d}{12 \widehat L^2_{A, t-1}},~~ \text{where} ~~\widehat L_{A, t} := \max\{\sqrt{\tfrac{\mu_d}{4(1-\beta)\eta_1}}, L_{A, 1},...,L_{A, t}\}.
\end{align}
Consequently, we have
\begin{align}\label{eq:corollary_2}
& Q(\widehat z_k, z) - \tfrac{\mu_d}{2}\|\widetilde y_0 - y\|^2 + \tfrac{\mu_d}{2}\|\widetilde y_0 -\widehat y_k \|^2 + \tfrac{\mu_d}{2}\|\widetilde y_k - y\|^2 \nn\\
&\leq \tfrac{12 \widehat L^2_{A,k}}{\mu_d(6k + \alpha k(k-3))}\left[\tfrac{\|x_0 - x\|^2 - \|\bar x_{k+1} - x\|^2}{\beta} + \eta_2\big(\tfrac{5 L^2_{A, 1}}{2\mu_d} - \tfrac{1}{\eta_1}\big)\|x_1 - x_0\|^2 \right].
\end{align}
Moreover, if $\max_{x\in X} \|x - x_0\|^2 \leq D_X^2$ and $\max_{y\in Y} \|y - \widetilde y_0\|^2 \leq D_Y^2$, then
\begin{align}\label{convergence_gap_bounded}
\max_{z \in Z} Q(\widehat z_k, z)  \leq  \tfrac{12 \widehat L^2_{A,k}}{\mu_d(6k + \alpha k(k-3))} \left(\tfrac{1}{\beta} + \tfrac{5}{8}\right) D_X^2 + \tfrac{\mu_d}{2}D_Y^2.
\end{align}
\end{corollary}
\begin{proof}
First, by using the condition $\beta\leq 1 - \tfrac{\sqrt{6}}{3}$, we have $2(1-\beta)^2 \geq \tfrac{4}{3}$, which together with \eqref{def_eta_t} ensures the stepsize rule \eqref{cond_4} is satisfied. 
Next, by the definition of $\tau_t$ in \eqref{def_tau_t}, we have for $t \geq 3$
\begin{align}\label{tau_bound}
\tau_t = \tau_2 + \mu_d \cdot \tsum_{j=3}^t [\tfrac{\alpha}{2} + \tfrac{2(1-\alpha)\eta_j  L^2_{A, j-1}}{\tau_{j-1}}] \overset{(i)}\leq \mu_d + \mu_d \cdot \tsum_{j=3}^t (\tfrac{1}{2}) = \tfrac{t \mu_d}{2},
\end{align}
where step (i) follows from the condition $\eta_t \leq \tfrac{\tau_{t-1}}{4L^2_{A, t-1}} $. Therefore, for $t \geq 3$, we always have
\begin{align}\label{bound_tau_ratio}
\tfrac{\tau_{t-1}+\mu_d}{\tau_{t}} \overset{(i)}\geq \tfrac{\tau_{t} +\frac{\mu_d}{2}}{\tau_{t}} \overset{(ii)}\geq \tfrac{t+1}{t},
\end{align} 
where step (i) follows from  $\tau_{t} \leq \tau_{t-1} + \tfrac{\mu_d}{2}$ due to \eqref{def_eta_t} and \eqref{def_tau_t}, and step (ii) follows from Ineq.~\eqref{tau_bound}.
Meanwhile, we have 
\begin{align}
\tau_t = \tau_2 + \mu_d \cdot \tsum_{j=3}^t [\tfrac{\alpha}{2} + \tfrac{2(1-\alpha)\eta_j L^2_{A, j-1}}{\tau_{j-1}}] \geq \mu_d +\tfrac{\alpha (t-2)\mu_d}{2}.
\label{eq:relation_tau_t}
\end{align}
Next, we use an inductive argument to show that $\eta_{t} \geq \tfrac{(3 + \alpha(t-3))\mu_d}{12\widehat L^2_{A, t-1}}$ for $t \geq 2$. First, by the definition of $\widehat L_t$ in \eqref{eta_lower_bound_2}, we can see that $\eta_2 = \min\{{(1-\beta)\eta_1}, {\tfrac{\mu_d}{4 L_{A, 1}^2}} \} = {\tfrac{\mu_d}{4\widehat L^2_{A, 1}}}$ {and $\eta_3 = \min\{\eta_2, \tfrac{\mu_d}{4 L_{A, 2}^2 }\} = \tfrac{\mu_d}{4\widehat L^2_{A, 2}}$} satisfy this condition. Then, we assume that for some $t \geq 3$, $\eta_{t} \geq \tfrac{(3 + \alpha(t-3))\mu_d}{12 \widehat L^2_{A, t-1}}$. Based on the stepsize rule \eqref{def_eta_t}, we have
\begin{align*}
\eta_{t+1} &= \min\big\{\tfrac{4}{3}\eta_{t}, \tfrac{\tau_{t-1}+\mu_d}{\tau_t} \eta_t, \tfrac{ \tau_t}{4L^2_{A, t-1}}\big\}\\
&\overset{(i)}\geq \min\big\{\tfrac{t+1}{t}\eta_t, \tfrac{1+ \alpha(t-2)/2}{4L^2_{A, t-1}/\mu_d}\big\}\\
&\geq \min\big\{\tfrac{t+1}{t} \cdot \tfrac{(3 + \alpha(t-3))\mu_d}{12 \widehat L^2_{A, t-1}}, \tfrac{(2+ \alpha(t-2))\mu_d}{8 L^2_{A, t}}  \big\}\\
& \geq \tfrac{(3 + \alpha(t-2))\mu_d}{12\widehat L^2_{A, t}},
\end{align*}
where step (i) follows from Ineq.~\eqref{bound_tau_ratio} and \eqref{eq:relation_tau_t}, which completes the inductive argument. The result in \eqref{eq:corollary_2} then follows by utilizing these lower bounds on $\eta_t$ and $\tau_t$ in Ineq. \eqref{eq_main_theorem} in Theorem \ref{main_theorem}. 
 Now using the condition $\eta_2 \leq \tfrac{\mu_d}{4 L_{A, 1}^2}$ in Ineq. \eqref{eq:corollary_2}, we have
\begin{align*}
Q(\widehat z_k, z) &\leq \tfrac{12 \widehat L^2_{A,k}}{\mu_d(6k + \alpha k(k-3))}\left[\tfrac{\|x_0 - x\|^2}{\beta} + \tfrac{5}{8}\|x_1 - x_0\|^2 \right] + \tfrac{\mu_d}{2}\|\widetilde y_0 - y\|^2\\
&\leq \tfrac{12 \widehat L^2_{A,k}}{\mu_d(6k + \alpha k(k-3))} \left(\tfrac{1}{\beta} + \tfrac{5}{8}\right) D_X^2 + \tfrac{\mu_d}{2}D_Y^2,
\end{align*}
which implies Ineq.~\eqref{convergence_gap_bounded} after taking the maximum over $z \in Z$. 
\end{proof}
\vgap

A few comments on the stepsize policy and convergence guarantees in Corollary~\ref{main_corollary_2} are in order. 

Firstly, observe that the stepsize policy in Corollary~\ref{main_corollary_2} does not require
any line search procedures. Specifically, the stepsize $\eta_t$ is selected based on $\eta_{t-1}$ and the local estimator $L_{A, t-1}$ computed from the last iterates $y_{t-1}$ and $y_{t-2}$ (see \eqref{def_L_A_t}),
while the parameter $\tau_t$ is determined by $\tau_{t-1}$, $L_{A, t-1}$, the recently used stepsize $\eta_t$,
and a hyper-parameter $\alpha\in (0,1]$. 
It is worth adding some explanation about the role of the hyper-parameter $\alpha$ in the specification of $\tau_t$  in \eqref{def_tau_t}. In fact, one might simply set $\alpha=1$, and in this case, the stepsize policy would reduce to 
\[
\tau_t = \tfrac{t \mu_d}{2} \ \mbox{and} \ \eta_t = \min
\left\{ \tfrac{t}{t-1}\eta_{t-1}, \tfrac{(t-1)\mu_d}{ 8 L^2_{A, t-1}}\right\}.
\]
In view of Ineq.~\eqref{eq:corollary_2}, this choice of $\alpha$ provides the best possible theoretical convergence guarantee. 
However, such a selection of $\alpha$ might result in overly conservative stepsizes for the implementation of AC-PDHG. To illustrate this observation, let us consider an extreme scenario when $L_{A,1} = \|A\|$. In this situation, no matter how small the subsequent local estimators $L_{A, t}$ are, the stepsize $\eta_t$ will be determined by the global constant $\|A\|$, thus preventing the algorithm from utilizing these local estimators. 
To address this issue, we suggest to choose a smaller $\alpha \in (0,1)$ to make AC-PDHG more adaptive to the local structure. Specifically, the last term in the 
definition of $\tau_t$ in \eqref{def_tau_t} depends on the ratio between $\eta_t$ and $\tau_{t-1} /(4L^2_{A, t-1})$. If $\eta_t < \tau_{t-1} /(4L^2_{A, t-1})$ in \eqref{def_eta_t}, the current selection of $\eta_t$ might be too conservative
since it is not determined by the current local estimate $L_{A, t-1}$. In response, $\tau_t$ in \eqref{def_tau_t} will grow slower compared to the case when $\alpha=1$, potentially allowing the stepsize $\eta_t$ to increase faster in subsequent iterations. By contrast, for the relatively more aggressive selection of $\eta_t = \tau_{t-1}/(4L^2_{A, t-1})$ in \eqref{def_eta_t},  $\tau_t$ will be updated by $\tau_t = \tau_{t-1} + \mu_d/2$, same as the case $\alpha = 1$, thereby potentially slowing down the growth of the stepsizes in subsequent iterations. In summary, the introduction of $\alpha$ allows for a more adaptive stepsize rule while still maintaining the optimal convergence rate.

Secondly, we would like to discuss the choice of the initial stepsize $\eta_1$. Conceptually speaking, AC-PDHG can always converge regardless of the specification of $\eta_1$ . However, if $\eta_1$ is too small, 
the estimated Lipschtiz constant $\widehat L_{A, t}$ in \eqref{eta_lower_bound_2} could be dominated by 
the first term $\sqrt{\mu_d/(4(1-\beta)\eta_1)}$ rather
than the desired $\|A\|$, which may slow down the convergence of the algorithm as one can see from \eqref{eq:corollary_2}. To ensure $\widehat L_{A, t} \leq \mathcal{O}(\|A\|)$, we should choose an $\eta_1$ such that $\eta_1^{-1} \leq \mathcal{O}\left(\|A\|^2/\mu_d\right)$. When $X$ is a bounded set, we can ignore the last term $-\eta_2\|x_1-x_0\|^2/\eta_1$ in the right hand side of  \eqref{eq:corollary_2} and set
\begin{align}\label{def_L_0}
\eta_1 = \tfrac{\zeta\mu_d }{4(1-\beta)L^2_{A, 0}}, \text{ where } L_{A, 0} := \tfrac{\|A^\top (\widetilde y_0 - y_0)\|}{\|\widetilde y_0-y_0\|}, ~\zeta>0
\end{align}
so that $\eta_1^{-1} \leq \mathcal{O}\left(\|A\|^2/\mu_d\right)$.
However, when $X$ is unbounded, we need to ensure $\eta_1 \leq 2\mu_d/(5 L_{A, 1}^2)$ in order to remove the last two terms $\eta_2\big(5 L^2_{A, 1}/(2\mu_d) - 1/\eta_1\big)\|x_1 - x_0\|^2$ from the right hand side of Ineq.~\eqref{eq:corollary_2}.
One strategy to find such a stepsize is to perform a simple line search in the first iteration. Taking the stepsize in \eqref{def_L_0} as the starting point, this line search will be terminated in at most $\mathcal{O}(\log(\|A\|/L_{A, 0}))$ steps.

Thirdly, 
we can easily derive  
the iteration complexity of AC-PDHG from Ineq.~\eqref{convergence_gap_bounded} for the simpler case where both $X$ and $Y$ are bounded. More specifically, to obtain an $\epsilon$-optimal solution of
problem~\eqref{eq:bilinear_saddle_point}, i.e., a point $\bar z = (\bar x, \bar y) \in Z$ such that $\max_{z \in Z} Q(\bar z, z) \leq \epsilon$, we need at most
\begin{align*}
\mathcal{O} \left( \sqrt{\tfrac{\|A\|^2 D_X^2}{\mu_d \epsilon}}\right) = \mathcal{O} \left(\tfrac{\|A\|D_X D_Y}{\epsilon}\right)
\end{align*}
number of AC-PDHG iterations by setting $\mu_d = \mathcal{O}(\epsilon/D_Y^2)$. It is well-known that this iteration complexity is optimal for solving a general bilinear saddle point problem. Notably, this optimal complexity is achieved by AC-PDHG without requiring any prior knowledge of $\|A\|$ or $D_X$.

Finally, we show how to specialize the convergence guarantees for PDHG obtained in Corollary~\ref{main_corollary_2} for the case when both $X$ and $Y$ are unbounded. For simplicity, we assume that the initial stepsize $\eta_1$ satisfies $\eta_1^{-1} \leq \mathcal{O}\left(\|A\|^2/\mu_d\right)$ and $\eta_1 \leq 2\mu_d/(5 L_{A, 1}^2)$. Then combining Ineq.~\eqref{eq:corollary_2} with the fact that $\|x_0-x\|^2 - \|\bar x_{k+1} - x\|^2 = \|x_0\|^2 - \|\bar x_{k+1}\|^2 - 2\langle x_0 -\bar x_{k+1},x\rangle$ and $\|\widetilde y_0-y\|^2 - \|\widetilde y_k - y\|^2 = \|\widetilde y_0\|^2 - \|\widetilde y_{k}\|^2 - 2\langle \widetilde y_0 -\widetilde y_{k},y\rangle$, we have
\begin{align}\label{Q_unbounded_feasible_region}
\max_{z \in Z} [Q(\widehat z_k, z) + \langle \delta_{x, k}, x \rangle+ \langle \delta_{y, k}, y \rangle]  \leq \tfrac{12 \widehat L^2_{A,k}}{\beta\mu_d(6k + \alpha k(k-3))}\|x_0\|^2 + \tfrac{\mu_d}{2}\|\widetilde y_0\|^2,
\end{align}
where 
\[
\delta_{x,k} := \tfrac{24 \widehat L^2_{A,k}}{\mu_d(6k + \alpha k(k-3))}(\bar x_{k+1} - x_0) \ \mbox{and} \ \delta_{y,k} := \mu_d (\widetilde y_0 - \widetilde y_k).
\]
Now we prove the convergence of $\|\delta_{x, k}\|$ and $\|\delta_{y, k}\|$. First, by taking $z = z^*$ in Ineq.~\eqref{eq:corollary_2} and using $\|\widetilde y_k - \widetilde y_0\|^2 \leq 2\|\widetilde y_k - y^*\|^2 + 2\|y^* - \widetilde y_0\|^2$, $Q(\widehat z_k, z^*) \geq 0$, we obtain
\begin{align*}
\mu_d\|\widetilde y_k - \widetilde y_0\|^2 \leq \mathcal{O}\left(\tfrac{\widehat L^2_{A, k}}{\mu_d k^2}\|x_0 - x^*\|^2 + \mu_d \|\widetilde y_0 - y^*\|^2 \right).
\end{align*}
Similarly, using $\|\bar x_{k+1} - x_0\|^2 \leq 2\|\bar x_{k+1}- x^*\|^2 + 2\|x^* - x_0\|^2$ in Ineq.~\eqref{eq:corollary_2}, we have
\begin{align*}
\tfrac{\widehat L^2_{A, k}}{\mu_d k^2}\|\bar x_{k+1} - x_0\|^2 \leq \mathcal{O}\left(\tfrac{\widehat L^2_{A, k}}{\mu_d k^2}\|x_0 - x^*\|^2 + \mu_d \|\widetilde y_0 - y^*\|^2 \right).
\end{align*}
Let us denote $D_{x^*}:= \|x_0 - x^*\|$ and $D_{y^*}:= \|\widetilde y_0 - y^*\|$. 
We then conclude from the above three relations that
by setting $\mu_d = \mathcal{O}(\epsilon/D_{y^*}^2)$,  
\begin{align*}
\|\delta_{y, k}\| &\leq \mathcal{O}\left(\tfrac{ \widehat L_{A, k}\sqrt{\mu_d}}{ k }D_{x^*} + \mu_d D_{y^*} \right) = \mathcal{O}\left( \tfrac{\epsilon}{D_{y^*}}\right),\\
\|\delta_{x, k}\| &\leq \mathcal{O}\left(\tfrac{\widehat L^2_{A, k}}{\mu_d k^2}D_{x^*} + \tfrac{\widehat L_{A, k}D_{y^*}}{k} \right) = \mathcal{O}\left( \tfrac{\epsilon}{D_{x^*}}\right)
\end{align*}
for any $k \geq \mathcal{O}(\|A\|D_{x^*}D_{y^*}/\epsilon)$.
Therefore, we obtain a convergence guarantee for AC-PDHG similar to the case when $X$ and $Y$ are bounded, but with a slightly different stopping criterion.

\subsection{AC-PDHG for linearly constrained problems}\label{subsec:linear_constraint} In this subsection, we analyze the convergence of AC-PDHG when applied to the linearly constrained problem \eqref{linear_constrained_problem}. 
This problem can be solved through its Lagrange dual given in the form of
\begin{align}\label{eq:bilinear_problem_lc}
\min_{x\in X} \max_{y \in \bbr^m} f(x)+ \langle Ax - b, y \rangle,
\end{align}
which is equivalent to \eqref{eq:bilinear_saddle_point} with $g(y) = \langle b, y\rangle$ and $Y = \bbr^m$. Throughout this subsection we assume the existence of
an optimal solution $x^*$ of \eqref{linear_constrained_problem} associated with $y^* \in \bbr^m$ such that $(x^*, y^*)$ is a saddle point of~\eqref{eq:bilinear_problem_lc}. 

The following result describes the convergence properties of AC-PDHG for solving the above linearly constrained problem with guarantees in terms of both the optimality gap and constraint violation. 
\begin{proposition}\label{coro_linear_constraint}
Consider Algorithm~\ref{alg:ac_primal_dual} applied to  problem \eqref{linear_constrained_problem} with $\widetilde y_0 = \mathbf{0} \in \bbr^m$ and algorithm parameters satisfing the conditions in Theorem~\ref{main_theorem}. 
Then we have
\begin{align}
f(\widehat x_k) - f(x^*)
& \leq \tfrac{1}{\sum_{t=1}^k \eta_t}\left[\tfrac{\|x_0 - x^*\|^2}{2\beta} + \left(\tfrac{5\eta_2 L^2_{A, 1}}{4\mu_d}  - \tfrac{\eta_2}{2\eta_1}\right)\|x_1 - x_0\|^2 \right], \\
\|A \widehat x_k - b\| &
\leq 2\mu_d \|y^*\| + 2 \sqrt{\tfrac{\mu_d}{\sum_{t=1}^k \eta_t}\left[\tfrac{\|x_0 - x^*\|^2}{2\beta} + \left(\tfrac{5\eta_2 L^2_{A, 1}}{4\mu_d} - \tfrac{\eta_2}{2\eta_1}\right)\|x_1 - x_0\|^2 \right]},
\end{align}
where $\widehat x_k$ and $\widehat y_k$ are defined in \eqref{def_hat_x_k_y_k} and $\widetilde y_k$ is defined in \eqref{def_tilde_y_k}. 
More specifically, if the stepsize parameters are selected according to Corollary~\ref{main_corollary_2}, we have
\begin{align}
f(\widehat x_k) - f(x^*)
& \leq \tfrac{12\widehat L^2_{A,k}}{\mu_d(6k + \alpha k(k-3))} \left[\tfrac{\|x_0 - x^*\|^2}{\beta} + \left(\tfrac{5\eta_2 L^2_{A, 1}}{2\mu_d}  - \tfrac{\eta_2}{2\eta_1}\right)\|x_1 - x_0\|^2 \right], \label{optimality_gap}\\
\|A \widehat x_k - b\| 
&\leq 2\mu_d \|y^*\| + 2 \sqrt{\tfrac{12  \widehat L^2_{A,k}}{6k + \alpha k(k-3)}\left[\tfrac{\|x_0 - x^*\|^2}{\beta} + \left(\tfrac{5\eta_2 L^2_{A, 1}}{2\mu_d} - \tfrac{\eta_2}{2\eta_1}\right)\|x_1 - x_0\|^2 \right]}, \label{constraint_violation}
\end{align}
where $\widehat L_{A, k}$ is defined in \eqref{eta_lower_bound_2}.
\end{proposition}
\begin{proof}
It follows from Ineq.~\eqref{eq_main_theorem} in Theorem~\ref{main_theorem} that
\begin{align}\label{eq_12_pdhg}
&f(\widehat x_k) + \langle A \widehat x_k - b, y\rangle - \tfrac{\mu_d}{2}\|y\|^2 - \left[f( x)  + \langle A  x - b, \widehat y_k\rangle - \tfrac{\mu_d}{2}\|\widehat  y_k\|^2\right] + \tfrac{\mu_d}{2}\|\widetilde y_k - y\|^2 \nn\\
 &\leq \tfrac{1}{\sum_{t=1}^k \eta_t}\left[\tfrac{\|x_0 - x\|^2}{2\beta} + \left(\tfrac{5\eta_2 L^2_{A, 1}}{4\mu_d}  - \tfrac{\eta_2}{2\eta_1}\right)\|x_1 - x_0\|^2 - \tfrac{\|x_{k+1} - \bar x_k\|^2}{4} - \tfrac{\|\bar x_{k+1} - x\|^2}{2\beta}\right].
\end{align}
Using $\|\widetilde y_k - y\|^2 = \|\widetilde y_k\|^2 + \|y\|^2 - 2 \langle \widetilde y_k, y\rangle$ and rearranging the terms in the previous result yield
\begin{align*}
&f(\widehat x_k) + \langle A \widehat x_k - b - \mu_d \widetilde y_k, y\rangle - \left[f( x) +  \langle A  x - b, \widehat y_k\rangle \right] + \tfrac{\mu_d}{2}\|\widehat  y_k\|^2 + \tfrac{\mu_d}{2}\|\widetilde y_k\|^2 \nn\\
 &\leq \tfrac{1}{\sum_{t=1}^k \eta_t}\left[\tfrac{\|x_0 - x\|^2}{2\beta} + \left(\tfrac{5\eta_2 L^2_{A, 1}}{4\mu_d}  - \tfrac{\eta_2}{2\eta_1}\right)\|x_1 - x_0\|^2 - \tfrac{\|x_{k+1} - \bar x_k\|^2}{4} - \tfrac{\|\bar x_{k+1} - x\|^2}{2\beta}\right].
\end{align*}
Taking $x = x^*$ and noting that $Ax^* - b=0$ due to
the optimality of  $(x^*, y^*)$ for \eqref{eq:bilinear_problem_lc}, we have
\begin{align}\label{eq_10_pdhg}
&f(\widehat x_k)  + \langle A \widehat x_k - b - \mu_d \widetilde y_k, y\rangle - f( x^*) + \tfrac{\mu_d}{2}\|\widehat  y_k\|^2 + \tfrac{\mu_d}{2}\|\widetilde y_k\|^2 \nn\\
 &\leq \tfrac{1}{\sum_{t=1}^k \eta_t}\left[\tfrac{\|x_0 - x^*\|^2}{2\beta} + \left(\tfrac{5\eta_2 L^2_{A, 1}}{4\mu_d}  - \tfrac{\eta_2}{2\eta_1}\right)\|x_1 - x_0\|^2 \right],
\end{align}
for any $y\in \bbr^m$. Observe that we must have
\begin{align*}
A \widehat x_k -  b - \mu_d \widetilde y_k = 0,
\end{align*}
since otherwise the left hand side of Ineq.~\eqref{eq_10_pdhg} can be unbounded. Therefore, we conclude from \eqref{eq_10_pdhg} and the above observation that
\begin{align}\label{eq_10_pdhg_1}
&f(\widehat x_k) - f( x^*)  + \tfrac{\mu_d}{2}\|\widehat  y_k\|^2 + \tfrac{\mu_d}{2}\|\widetilde y_k\|^2 \nn\\
 &\leq \tfrac{1}{\sum_{t=1}^k \eta_t}\left[\tfrac{\|x_0 - x^*\|^2}{2\beta} + \left(\tfrac{5\eta_2 L^2_{A, 1}}{4\mu_d}  - \tfrac{\eta_2}{2\eta_1}\right)\|x_1 - x_0\|^2 \right],
\end{align}
and
\begin{align*}
\|A \widehat x_k - b\| =\mu_d \|\widetilde y_k\|,
\end{align*}
which, respectively, 
provide bounds on the optimality gap and constraint violation. Next, we provide an upper bound for $\|\widetilde y_k\|$. By taking $(x,y) = (x^*, y^*)$ in Ineq.~\eqref{eq_12_pdhg}, we have
\begin{align*}
&f(\widehat x_k) +  \langle A \widehat x_k - b, y^*\rangle - \left[f( x^*) + \langle A x^* - b, \widehat y_k\rangle \right] + \tfrac{\mu_d}{2}\|\widehat  y_k\|^2 - \tfrac{\mu_d}{2}\|y^*\|^2 + \tfrac{\mu_d}{2}\|\widetilde y_k - y^*\|^2 \nn\\
 &\leq \tfrac{1}{\sum_{t=1}^k \eta_t}\left[\tfrac{\|x_0 - x^*\|^2}{2\beta} + \left(\tfrac{5\eta_2 L^2_{A, 1}}{4\mu_d}  - \tfrac{\eta_2}{2\eta_1}\right)\|x_1 - x_0\|^2 \right],
\end{align*}
which, in view of the relation $f(\widehat x_k)  + \langle A\widehat x_k -  b, y^*\rangle \geq f( x^*) +\langle A  x^* -  b, \widehat y_k\rangle$ due to 
$(x^*, y^*)$ being a pair of saddle point of \eqref{eq:bilinear_problem_lc}, implies that
\begin{align*}
\tfrac{\mu_d}{4}\|\widetilde y_k\|^2 - \mu_d\|y^*\|^2
& \leq - \tfrac{\mu_d}{2}\|y^*\|^2 + \tfrac{\mu_d}{2}\|\widetilde y_k - y^*\|^2 \nn\\
 &\leq \tfrac{1}{\sum_{t=1}^k \eta_t}\left[\tfrac{\|x_0 - x^*\|^2}{2\beta} + \left(\tfrac{5\eta_2 L^2_{A, 1}}{4\mu_d}  - \tfrac{\eta_2}{2\eta_1}\right)\|x_1 - x_0\|^2 \right].
\end{align*}
As a consequence, we have
\begin{align*}
\mu_d \|\widetilde y_k\|
  &\leq \mu_d \sqrt{4\|y^*\|^2+\tfrac{4}{\mu_d(\sum_{t=1}^k \eta_t)}\left[\tfrac{\|x_0 - x^*\|^2}{2\beta} + \left(\tfrac{5\eta_2 L^2_{A, 1}}{4\mu_d} - \tfrac{\eta_2}{2\eta_1}\right)\|x_1 - x_0\|^2 \right]}\\
&\leq 2\mu_d \|y^*\| + 2 \sqrt{\tfrac{\mu_d}{\sum_{t=1}^k \eta_t}\left[\tfrac{\|x_0 - x^*\|^2}{2\beta} + \left(\tfrac{5\eta_2 L^2_{A, 1}}{4\mu_d}  - \tfrac{\eta_2}{2\eta_1}\right)\|x_1 - x_0\|^2 \right]},
\end{align*}
which completes the proof.
\end{proof}
\vgap

Now let us derive from Proposition~\ref{coro_linear_constraint} the complexity of AC-PDHG for solving the linearly constrained problem \eqref{linear_constrained_problem}
in order to find an $(\epsilon_1,\epsilon_2)$-optimal solution, i.e., a point
$\bar x \in X$ such that $f(\bar x) - f^* \leq \epsilon_1$ and $\|A \bar x - b\|\leq \epsilon_2$. 
For simplicity, we assume the initial stepsize $\eta_1$ satisfies $\eta_1^{-1} \leq \mathcal{O}\left(\|A\|^2/\mu_d\right)$ and $\eta_1 \leq 2\mu_d/(5 L_{A, 1}^2)$, which can be ensured through a simple line search as previously discussed in the comments
after Corollary~\ref{main_corollary_2}.  
In this situation, the convergence guarantees in \eqref{optimality_gap} and \eqref{constraint_violation} can be written as
\begin{align*}
&f(\widehat x_k) - f(x^*) \leq \mathcal{O}\left( \tfrac{\widehat L_{A,k}^2 D_{x^*}^2}{\mu_d k^2}\right),\\
&\|A \widehat x_k - b\|  \leq \mu_d \|\widetilde y_k\| \leq \mathcal{O}\left( \mu_d\|y^*\| + \tfrac{\widehat L_{A,k}D_{x^*}}{k} \right),
\end{align*}
where $D_{x^*} := \|x_0 - x^*\|$ and $D_{y^*} := \|y^*\|$.
Thus, taking $\mu_d = \mathcal{O}(\epsilon_2/D_{y^*})$, AC-PDHG can find an $(\epsilon_1, \epsilon_2)$-optimal solution of \eqref{linear_constrained_problem} in at most 
\begin{align} \label{eq:PDHG_Linear1}
\mathcal{O}\left( \tfrac{\|A\|D_{x^*}}{\epsilon_2} +\tfrac{\|A\|D_{x^*}\sqrt{D_{y^*}}}{\sqrt{\epsilon_1\epsilon_2}}\right)
\end{align}
iterations.
In particular, for the case when $\epsilon_1 =\epsilon$ and $\epsilon_2 = \epsilon/D_{y^*}$, which corresponds to a higher requirement on the feasibility (as long as $D_{y^*} \geq 1$), 
the above complexity bound reduces to
\begin{align} \label{eq:PDHG_Linear2}
\mathcal{O}\left( \tfrac{\|A\|D_{x^*} D_{y^*}}{\epsilon}\right)
\end{align}
with $\mu_d$ chosen as $\mathcal{O}(\epsilon/D^2_{y^*})$.

\vspace{0.1in}

It should be noted that 
to achieve the above complexity bounds 
in \eqref{eq:PDHG_Linear1} and \eqref{eq:PDHG_Linear2} requires the input of $\mu_d$, which necessitates an
estimate $\widehat D_{y^*}$ for the unknown quantity $D_{y^*}$.
If this estimate $\widehat D_{y^*}$ is too conservative (i.e., $\widehat D_{y^*} \gg D_{y^*}$), the convergence of AC-PDHG may be slowed down. Conversely, if the estimate is too aggressive (i.e., $\widehat D_{y^*} \ll D_{y^*}$), we might fail to achieve the desired accuracy for the constraints. However,  when the primal feasible region $X$ is bounded (i.e., $\max_{x, y}\|x-y\|\leq D_X$), we can employ a simple ``guess and check'' procedure, as shown below, to determine $\mu_d$ without requiring a precise estimation of $D_{y^*}$. 
\begin{algorithm}[h]\caption*{A guess and check procedure for AC-PDHG:}\label{alg:ac_pdhg_guess}
	\begin{algorithmic}
		\State{\textbf{Input}: Choose $\widehat D^{(0)}$ such that $\widehat D^{(0)} \leq D_{y^*}$. Fix $\epsilon_1 >0$ and $\epsilon_2>0$.}
		\For{$i=0, 1, 2,\cdots$}
		\State{Let $\widehat D^{(i)} := \widehat D^{(0)}\cdot 2^{i}, ~\mu^{(i)} := \tfrac{\epsilon_2}{4\widehat D^{(i)}}$. Run Algorithm~\ref{alg:ac_primal_dual} with $\mu_d = \mu^{(i)}$ until we find the first iterate $\widehat k \geq 3$ such that
  \begin{align*}
  &\mathcal{E}_{1}(\widehat k) \leq \epsilon_1 \text{  and  }\min\left\{\|A \widehat x_{\widehat k}- b\|, \mathcal{E}_{2}(\widehat k)\right\} \leq \epsilon_2,\\
  &\text{where } \mathcal{E}_{1}(k) := \tfrac{12 \widehat L_{A, k}^2 D_X^2}{\mu^{(i)}\beta(6k + \alpha k (k-3))}, \text{  }\mathcal{E}_{2}(k):= 4 \sqrt{\tfrac{12  \widehat L_{A, k}^2 D_X^2}{\beta(6k + \alpha k(k-3))}},  \text{ for $\widehat L_{A, k}$ defined in \eqref{eta_lower_bound_2}.}
  \end{align*}
  }
  \State{If $\|A \widehat x_{\widehat k}- b\| \leq \epsilon_2$: Terminate.}
		\EndFor
\State{\textbf{Output}: $(\widehat x, \widehat w) = (\widehat x_{\widehat k}, \widehat w_{\widehat k})$, $\widehat D_Y= \widehat D^{(i)}$.}
	\end{algorithmic}
\end{algorithm}

 We provide a brief explanation of this procedure. For each $\widehat D^{(i)}$, we check whether it is possible to attain the desired optimality gap and constraint violation. On one hand, if  $\mathcal{E}_{1}(\widehat k) \leq \epsilon_1$ and $\|A \widehat x_{\widehat k}- b\| \leq \epsilon_2$, we can terminate the algorithm with the desired accuracy achieved. On the other hand, if $\mathcal{E}_{1}(\widehat k) \leq \epsilon_1$ and $\mathcal{E}_{2}(\widehat k) \leq \epsilon_2$, we conclude from \eqref{constraint_violation} that $2\mu^{(i)} \|y^*\| > \epsilon_2/2$, thus $\widehat D^{(i)}$ should be further enlarged. This guess and check procedure is guaranteed to terminate, and upon termination, we must have  $\widehat D_Y \leq 2 D_{y^*}$.
As a result, an $(\epsilon_1, \epsilon_2)$-optimal solution of \eqref{linear_constrained_problem} can be found within at most 
$$\mathcal{O}\left( \tfrac{\|A\|D_{X}}{\epsilon_2}\log(\tfrac{D_{y^*}}{\widehat D^{(0)}}) +\tfrac{\|A\|D_{X}\sqrt{D_{y^*}}}{\sqrt{\epsilon_1\epsilon_2}}\right)$$
iterations. This procedure does not require the input of any other problem parameters than the diameter of the primal feasible region $D_X$ as well as the desired accuracy $\epsilon_1$ and $\epsilon_2$.

\section{An adaptive alternating direction method of multipliers (ADMM)}\label{sec:ADMM}
In this section, we consider an important class of linearly constrained problem in the form of
\begin{align}\label{bilinear_two_parts}
\min_{x\in X, w \in W} ~&F(x) + G(w) \nn\\
\text{s.t.} ~& B w - K x = b.
\end{align}
where $X \subseteq \bbr^{n_1}$ and $W \subseteq \bbr^{n_2}$ are closed convex sets, $b \in \bbr^m$, $K\in \bbr^{m \times n_1}$ and $B\in\bbr^{m \times n_2}$. 

Similar to Section~\ref{sec:PDHG}, we solve \eqref{bilinear_two_parts} by considering the primal-dual formulation
\begin{align}\label{eq:bilinear_problem_ADMM}
\min_{x\in X, w \in W} \max_{y \in \bbr^m} F(x) + G(w) + \langle Kx-Bw+b, y \rangle.
\end{align}
We assume there exists an optimal solution $(x^*, w^*) \in X\times K$ for \eqref{bilinear_two_parts}, along with a dual variable $y^*$ such that $(x^*, w^*, y^*) \in Z:= X \times K \times \bbr^m$ is a saddle point for problem \eqref{eq:bilinear_problem_ADMM}. 
With a slight abuse of notation, we still define here
\begin{align}\label{def_Q_mu_admm}
Q(\bar z, z):= &F(\bar x) + G(\bar w) + \langle K \bar x - B \bar w + b, y\rangle - \left[F( x) + G(w) + \langle K  x - B w + b, \bar y\rangle \right],
\end{align}
where $\bar z := (\bar x, \bar w, \bar y) \in Z$ and $z = (x,w,y) \in Z$. It follows from the convex-concave structure of \eqref{eq:bilinear_problem_ADMM} that $z^*$ is a saddle point of \eqref{eq:bilinear_problem_ADMM} if and only if $Q(z^*, z) \leq 0$ for any $z \in Z$.

Clearly, we can apply the AC-PDHG method in Algorithm~\ref{alg:ac_primal_dual} to solve problem~\eqref{bilinear_two_parts}, and the complexity will depend on $\|K\| + \|B\|$. However, if $\|B\|\gg \|K\|$, the convergence of AC-PDHG is dominated by $\|B\|$, no matter how small $\|K\|$ is. To address this issue, we propose a preconditioned ADMM-type method whose convergence rate depends solely on $\|K\|$, at the cost of solving a more complicated subproblem in each iteration.
\begin{algorithm}[h]\caption{Auto-Conditioned Alternating Direction Method of Multipliers (AC-ADMM)}\label{alg:ac_admm}
	\begin{algorithmic}
		\State{\textbf{Input}: initial point $x_0 = \bar x_0 \in X$, nonnegative parameters $\beta_t \in (0, 1)$, $\eta_t \in \bbr_+$, and $\tau_t \in \bbr_+$. 
  Let 
  \begin{align}
  w_0 &= \arg \min_{w \in W}  \left\{G(w) +  \tfrac{1}{2\mu_d}\|Bw - K x_0 -b\|^2 \right\}, \label{primal_update_0}\\
  y_0 &= ( K x_0 - B w_0 +b)/{\mu_d}.\label{dual_update_0}
  \end{align}}
		\For{$t=1,\cdots, k$}
		\State{
  \begin{align}
  x_t &= \arg \min_{x \in X} \left\{\eta_t [\langle K^\top y_{t-1} , x \rangle + F(x)] + \tfrac{1}{2}\|\bar x_{t-1} - x\|^2\right\}, \label{primal_x_update}\\
  \bar x_t &= (1-\beta_t) \bar x_{t-1} + \beta_t x_t, \label{primal_x-center}\\
  w_t &= \arg \min_{w \in W}  \left\{G(w) + (\tau_t + \mu_d)^{-1}\left(-\tau_t \langle y_{t-1}, Bw \rangle  + \tfrac{1}{2}\|Bw - K x_t -b\|^2 \right)\right\}, \label{primal_w_update}\\
  y_t & = \left[\tau_t y_{t-1}  - (B w_t - Kx_t - b)\right]/(\tau_t + \mu_d). \label{dual_update}
  \end{align}
  }
		\EndFor
	\end{algorithmic}
\end{algorithm}

Similar to AC-PDHG, the AC-ADMM method (Algorithm~\ref{alg:ac_admm}) updates both primal and dual variables in each iteration. The major difference exists in that AC-ADMM maintains the updates of two separate sequences of primal variables, $\{x_t\}$ and $\{w_t\}$. Specifically, the update of $w_t$ in \eqref{primal_w_update} requires solving a more complicated subproblem involving an augmented Lagrangian term $\|B w -Kx_t -b\|^2$. However, unlike the classical ADMM where the update for both $x_t$ and $w_t$ involves augmented Lagrangian terms, the update of $x_t$ in AC-ADMM is based on a prox-mapping step similar to \eqref{primal_prox-mapping} in Algorithm~\ref{alg:ac_primal_dual}.
This subproblem can be computed more efficiently than those involving augmented Lagrangian terms.  
Additionally, as in AC-PDHG, we use a moving averaging sequence $\{\bar x_t\}$ as the prox-centers.  
It should also be noted that the explicit update rule \eqref{dual_update} for the dual variable $y_t$ can be equivalently expressed in the following form of prox-mapping with a regularization term $\tfrac{\mu_d}{2}\|y\|^2$:
\begin{align} \label{eq:ADMM_equivalent_dual}
y_t = \arg\min_{y \in \bbr^m}\left\{ \langle B w_t - Kx_t - b,y \rangle + \tfrac{\mu_d}{2}\|y\|^2 + \tfrac{\tau_t}{2} \|y_{t-1} - y\|^2 \right\},
\end{align}
which highlights a more explicit connection between AC-PDHG and AC-ADMM.

Our goal in the remaining part of this section is to establish the convergence of AC-ADMM equipped with an adaptive stepsize policy. Instead of relying on the global constant $\|K\|$, we compute the local estimate of $\|K\|$ at the $t$-th iteration according to
\begin{align} 
L_{K, t} := \tfrac{\|K^\top (y_t - y_{t-1})\|}{\|y_t - y_{t-1}\|} \label{def_L_K_t}
\end{align}
and use it to determine the selection of
$\eta_t$ and $\tau_t$ for the subsequent iterations.

We first show a result to characterize the optimality condition of steps \eqref{primal_w_update} and \eqref{dual_update} and to provide an important relation between the primal and dual variables for the first iteration. 
\begin{lemma}\label{lemma_optimality_conditions}
Let $\{x_t\}$, $\{w_t\}$ and $\{y_t\}$ be generated by Algorithm~\ref{alg:ac_admm}. We have for $t \geq 0$,
\begin{align}\label{optimality_condition_w}
G(w) - G(w_t) - \langle  B^\top y_t, w - w_t\rangle \geq 0,
\end{align}
and 
\begin{align}
&\langle B w_{t} - Kx_{t} - b, y_{t} - y\rangle + \tfrac{\mu_d}{2}\|y_{t}\|^2 - \tfrac{\mu_d}{2}\|y\|^2 + \tfrac{\mu_d + \tau_{t}}{2}\|y_{t} - y\|^2\nn\\
& = \tfrac{\tau_t}{2} \|y_{t-1} - y\|^2 - \tfrac{\tau_t}{2} \|y_{t-1} - y_t\|^2. \label{eq_dual_three-point_lemma_admm}
\end{align}
Moreover, if $\tau_1 = 0$, we have
\begin{align}\label{eq_bound_y_1_y_0_K}
\|y_1 - y_0\| \leq \tfrac{L_{K, 1}}{\mu_d}\|x_1 - x_0\|.
\end{align}
\end{lemma}
\begin{proof}
First, notice that the updates in \eqref{primal_update_0} and \eqref{dual_update_0} can be viewed as \eqref{primal_w_update} and \eqref{dual_update}, respectively, with $\tau_0=0$. By the optimality condition of \eqref{primal_w_update} and the convexity of $G$, we have for $t \geq 0$,
\begin{align*}
G(w) - G(w_t) + \langle - \tau_t B^\top y_{t-1}  + B^\top(Bw_t - K x_t -b), w - w_t \rangle/(\tau_t + \mu_d) \geq 0.
\end{align*}
By substituting \eqref{dual_update} into the above inequality, we obtain
\begin{align*}
G(w) - G(w_t) - \langle B^\top y_t, w - w_t \rangle \geq 0,
\end{align*}
which completes the proof of Ineq. \eqref{optimality_condition_w}. 
Eq. \eqref{eq_dual_three-point_lemma_admm} follows from either direct calculations or an immediate application of the ``three-point'' lemma (e.g., Lemma 3.5 of \cite{LanBook2020}, all the inequalities in the proof hold with equality) to the equivalent expression of \eqref{dual_update} in \eqref{eq:ADMM_equivalent_dual}.

It remains to prove Ineq. \eqref{eq_bound_y_1_y_0_K}. Taking $t=1$, $y=y_0$ and $\tau_1=0$ in \eqref{eq_dual_three-point_lemma_admm} yields
\begin{align*}
\langle B w_1 - K x_1 - b, y_1 - y_0 \rangle \leq \tfrac{\mu_d}{2}\|y_0\|^2 - \tfrac{\mu_d}{2}\|y_1\|^2 - \tfrac{\mu_d}{2}\|y_1- y_0\|^2.
\end{align*}
Similarly, we can set $t=0$, $\tau_0 =0$ and $y=y_1$ in \eqref{eq_dual_three-point_lemma_admm} to obtain
\begin{align*}
\langle B w_0 - K x_0 - b, y_0 - y_1\rangle \leq \tfrac{\mu_d}{2}\|y_1\|^2 - \tfrac{\mu_d}{2}\|y_0\|^2 - \tfrac{\mu_d}{2}\|y_1 - y_0\|^2.
\end{align*}
By summing up the above two inequalities, we have
\begin{align}\label{inter_proof_1}
\langle (K x_1 - Bw_1) - (Kx_0 - Bw_0), y_1- y_0\rangle \geq \mu_d \|y_1 - y_0\|^2.
\end{align}
Meanwhile, Ineq.~\eqref{optimality_condition_w} indicates that
\begin{align*}
G(w_0) - G(w_1) + \langle -B^{\top} y_1, w_0 - w_1\rangle &\geq 0,\\
G(w_1) - G(w_0) + \langle -B^{\top} y_0, w_1 - w_0\rangle &\geq 0.
\end{align*}
Summing up the above two inequalities provides us with 
\begin{align}\label{inter_proof_2}
\langle B^\top (y_1- y_0), w_1 - w_0\rangle \geq 0.
\end{align}
Then we sum up Ineqs.~\eqref{inter_proof_1} and \eqref{inter_proof_2} and obtain
\begin{align*}
\langle K^\top (y_1 - y_0), x_1 - x_0\rangle \geq \mu_d \|y_1 - y_0\|^2 = \tfrac{\mu_d}{L_{K, 1}^2}\|K^\top (y_1 - y_0)\|^2,
\end{align*}
which, together with Young's inequality, indicates that
\begin{align*}
\tfrac{\mu_d}{L_{K, 1}^2}\|K^\top (y_1- y_0)\|^2 \leq \tfrac{\mu_d}{2 L^2_{K,1}}\|K^\top (y_1- y_0)\|^2 + \tfrac{L_{K,1}^2}{2\mu_d}\|x_1-x_0\|^2.
\end{align*}
Finally, we obtain Ineq.~\eqref{eq_bound_y_1_y_0_K} by rearranging the terms and further use the definition of $L_{K, 1}$. 
\end{proof}
\vgap 

We are now ready to establish the main convergence properties of AC-ADMM applied to problem \eqref{bilinear_two_parts} in terms of both the optimality gap and constraint violation. 
\begin{theorem}\label{theorem_admm}
Let $\{x_t\}, \{\bar x_t\}, \{w_t\}$ and $\{y_t\}$ be generated by Algorithm~\ref{alg:ac_admm} with the parameters $\{\tau_t\}$, $\{\eta_t\}$, and $\{\beta_t\}$ satisfying conditions \eqref{cond_3} and
\begin{align}
    \eta_2 &\leq \min\left\{{(1-\beta)\eta_1},  \tfrac{\mu_d}{4 L^2_{K,1}}\right\},\label{cond_1_admm}\\ 
    \eta_t &\leq \min\left\{2(1-\beta)^2 \eta_{t-1}, \tfrac{\tau_{t-2}+\mu_d}{ \tau_{t-1}}\eta_{t-1}, \tfrac{\tau_{t-1}}{4L_{K, t-1}^2} \right\}, ~t \geq 3.
    \label{cond_4_admm}
\end{align}
Then we have
\begin{align}
F(\widehat x_k) + G(\widehat w_k) &- [F( x^*) + G(w^*)] \nn\\
& \leq \tfrac{1}{\sum_{t=1}^k \eta_t}\left[\tfrac{\|x_0 - x^*\|^2}{2\beta} + \left(\tfrac{5\eta_2 L^2_{K, 1}}{4\mu_d}  - \tfrac{\eta_2}{2\eta_1}\right)\|x_1 - x_0\|^2 \right], \label{eq:opt_gap}\\
\|K \widehat x_k - B \widehat w_k + b\| &\leq 2\mu_d \|y^*\| + 2 \sqrt{\tfrac{\mu_d}{\sum_{t=1}^k \eta_t}\left[\tfrac{\|x_0 - x^*\|^2}{2\beta} + \left(\tfrac{5\eta_2 L^2_{K, 1}}{4\mu_d}  - \tfrac{\eta_2}{2\eta_1}\right)\|x_1 - x_0\|^2 \right]},\label{eq:const_vio}
\end{align}
where $\widetilde y_k$ is defined in \eqref{def_tilde_y_k} and $\widehat z_k := (\widehat x_k, \widehat w_k, \widehat y_k)$
is given by
\begin{align}
\widehat x_k &:= \tfrac{\sum_{t=1}^{k}\eta_{t+1} x_{t}}{\sum_{t=1}^{k}\eta_{t+1}}, \quad \widehat y_k := \tfrac{\sum_{t=1}^{k}\eta_{t+1} y_{t}}{\sum_{t=1}^{k}\eta_{t+1}},\ \mbox{and} \ \widehat w_k := \tfrac{\sum_{t=1}^{k}\eta_{t+1} w_{t}}{\sum_{t=1}^{k}\eta_{t+1}}.\label{def_hat_x_k_y_k_admm}
\end{align}
\end{theorem}
\begin{proof}
First, recall that \eqref{eq6_prime} in the proof of Proposition \ref{proposition_1} follows from the update rules \eqref{primal_prox-mapping} and \eqref{primal_prox-center} in Algorithm~\ref{alg:ac_primal_dual}. Since the update rules \eqref{primal_x_update} and \eqref{primal_x-center} in AC-ADMM resembles \eqref{primal_prox-mapping} and \eqref{primal_prox-center} respectively, with $A$ replaced by $K$, we can show similarly that for $t \geq 3$,
\begin{align}\label{eq6_prime_admm}
&\eta_t \langle K^\top y_{t-1}, x_{t-1} - x\rangle + \eta_t[F(x_{t-1}) - F(x)] + \tfrac{1}{2\beta_t} \|\bar x_t - x\|^2 \nn\\
&\leq \tfrac{1}{2\beta_t}\|\bar x_{t-1} - x\|^2 +  \eta_t \langle K^\top (y_{t-1} - y_{t-2}), x_{t-1}-x_t \rangle- \tfrac{1}{2}\|x_t - \bar x_{t-1}\|^2.
\end{align}
By combining Ineq.~\eqref{eq6_prime_admm} with Ineq.~\eqref{optimality_condition_w}, we have
\begin{align*}
&\eta_t \left[ \langle K^\top y_{t-1}, x_{t-1} - x\rangle + F(x_{t-1}) - F(x) + G(w_{t-1}) - G(w) - \langle B^\top y_{t-1}, w_{t-1} - w\rangle \right]  \\
&\leq \tfrac{1}{2\beta_t}\|\bar x_{t-1} - x\|^2 - \tfrac{1}{2\beta_t} \|\bar x_t - x\|^2 +  \eta_t \langle K^\top (y_{t-1} - y_{t-2}), x_{t-1}-x_t \rangle- \tfrac{1}{2}\|x_t - \bar x_{t-1}\|^2
\end{align*}
or equivalently,
\begin{align*}
&\eta_t \left[  F(x_{t-1})-F(x) + G(w_{t-1}) - G(w)+ \langle Kx_{t-1} - Bw_{t-1}+b, y \rangle - \langle Kx-Bw+b, y_{t-1}\rangle \right]  \\
&\leq \tfrac{1}{2\beta_t}\|\bar x_{t-1} - x\|^2 - \tfrac{1}{2\beta_t} \|\bar x_t - x\|^2 +  \eta_t \langle K^\top (y_{t-1} - y_{t-2}), x_{t-1}-x_t \rangle- \tfrac{1}{2}\|x_t - \bar x_{t-1}\|^2\\
&\quad  + \eta_t  \langle - Kx_{t-1} + Bw_{t-1} -b, y_{t-1}\rangle.
\end{align*}
Substituting Ineq.~\eqref{eq_dual_three-point_lemma_admm} into the above inequality and rearranging the terms, we have
\begin{align}
&\eta_t \big[ F(x_{t-1}) - F(x) + G(w_{t-1}) - G(w) + \langle Kx_{t-1} - Bw_{t-1}+b, y \rangle- \langle Kx-Bw+b, y_{t-1}\rangle  \nn\\
&\quad - \tfrac{\mu_d}{2}\|y\|^2 + \tfrac{\mu_d}{2}\|y_{t-1}\|^2 \big] + \tfrac{\eta_t(\mu_d+\tau_{t-1})}{2}\|y_{t-1} - y\|^2 - \tfrac{\eta_t \tau_{t-1}}{2}\|y_{t-2}-y\|^2  \nn\\
&\leq \tfrac{1}{2\beta_t}\|\bar x_{t-1} - x\|^2 - \tfrac{1}{2\beta_t} \|\bar x_t - x\|^2  +  \eta_t \langle K^\top (y_{t-1} - y_{t-2}), x_{t-1}-x_t \rangle\nn\\
&\quad - \tfrac{1}{2}\|x_t - \bar x_{t-1}\|^2 - \tfrac{\eta_t \tau_{t-1}}{2}\|y_{t-1} - y_{t-2}\|^2.\label{eq9_admm}
\end{align}
Moreover, for the case $t=2$, by utilizing $\beta_1 = 0$ and $\eta_2 \leq (1-\beta)\eta_1$, we can show the following inequality as an analog of Ineq.~\eqref{eq9_3} in the proof of Proposition~\ref{proposition_1}:
\begin{align}\label{eq6_admm}
&\eta_2 \langle K^\top y_{1}, x_{1} - x\rangle + \eta_2[F(x_1) - F(x)] + \tfrac{1}{2\beta_2} \|\bar x_2 - x\|^2 + \tfrac{\eta_2}{2\eta_{1}} \left[\|x_{1} - \bar x_{1}\|^2 + \|x_2 - x_{1}\|^2 \right]\nn\\
&\leq \tfrac{1}{2\beta_2}\|\bar x_{1} - x\|^2 +  \eta_2 \langle K^\top (y_{1} - y_{0}), x_{1}-x_2 \rangle - \tfrac{1}{2}\|x_2 - \bar x_{1}\|^2.
\end{align}
Then, combining the above inequality with Ineqs.~\eqref{optimality_condition_w} and \eqref{eq_dual_three-point_lemma_admm}, using $\tau_1 =  0$, and rearranging the terms yield
\begin{align}
&\eta_2 \left[F(x_1) - F(x)+ G(w_{1}) - G(w) + \langle Kx_{1} - Bw_{1}+b, y \rangle - \langle Kx-Bw+b, y_{1}\rangle - \tfrac{\mu_d}{2}\|y\|^2 + \tfrac{\mu_d}{2}\|y_{1}\|^2 \right]  \nn\\
& \quad  + \tfrac{1}{2\beta_2} \|\bar x_2 - x\|^2 + \tfrac{\eta_2\mu_d}{2}\|y_{1} - y\|^2 + \tfrac{\eta_2}{2\eta_{1}} \left[\|x_{1} - \bar x_{1}\|^2 + \|x_2 - x_{1}\|^2 \right]   \nn\\
&\leq \tfrac{1}{2\beta_2}\|x_0 - x\|^2 +  \eta_2 \langle K^\top (y_{1} - y_{0}), x_{1}-x_2 \rangle- \tfrac{1}{2}\|x_2 - \bar x_{1}\|^2  - \tfrac{1}{2}\|x_2 - \bar x_{1}\|^2, \label{eq9_2_admm}
\end{align}
By taking the summation of Ineq.~\eqref{eq9_2_admm} and the telescope sum of Ineq.~\eqref{eq9_admm} for $t=3,..., k+1$, and noting that $\beta_t = \beta$ for $t \geq 2$, we obtain
\begin{align}\label{eq_final_admm}
& \tsum_{t=1}^k \eta_{t+1} \left[F(x_t) + G(w_t) + \langle K x_{t} - B w_t + b, y \rangle - \tfrac{\mu_d}{2}\|y\|^2 - F(x) - G(w) - \langle Kx - Bw +b, y_{t}\rangle + \tfrac{\mu_d}{2}\|y_{t}\|^2 \right]\nn\\
& + \tfrac{1}{2\beta}\|\bar x_{k+1} - x\|^2 + \tfrac{\eta_2}{2\eta_{1}} \left[\|x_1 - \bar x_1\|^2 + \|x_2 - x_1\|^2 \right] + \tfrac{\eta_{k+1}(\mu_d + \tau_k)}{2}\|y_k - y\|^2\nn\\
&+ \tsum_{t=1}^{k-1}\tfrac{\eta_{t+1}(\mu_d + \tau_t) - \eta_{t+2}\tau_{t+1}}{2}\|y_t - y\|^2\nn\\
& \leq \tfrac{1}{2\beta}\|x_0 - x\|^2 + \tsum_{t=2}^{k+1}\Delta_{t},
\end{align}
where
\begin{align}
\Delta_t &:= \eta_t \langle K^\top (y_{t-1} - y_{t-2}), x_{t-1} - x_t\rangle - \tfrac{\eta_t \tau_{t-1}}{2}\|y_{t-2} - y_{t-1}\|^2  - \tfrac{1}{2}\|x_t - \bar x_{t-1}\|^2.\label{def_Delta_t_admm}
\end{align}
Then, by using exactly the same arguments in the proof of Proposition \ref{proposition_1} (after Ineq.~\eqref{eq_final}), we obtain
\begin{align}\label{eq:proposition_2}
&\left(\tsum_{t=1}^k \eta_t\right) \cdot \left(Q(\widehat z_k, z) - \tfrac{\mu_d}{2}\|y\|^2 + \tfrac{\mu_d}{2}\|\widehat y_k\|^2+ \tfrac{\mu_d}{2}\|\widetilde y_k - y\|^2\right) + \tfrac{1}{2\beta}\|\bar x_{k+1} - x\|^2 + \tfrac{\eta_2}{2\eta_{1}} \left[\|x_1 - \bar x_1\|^2 + \|x_2 - x_1\|^2 \right]\nn\\
 &\leq \tfrac{1}{2\beta}\|x_0 - x\|^2 + \tsum_{t=2}^{k+1}\Delta_{t}.
\end{align}
Bounding $\Delta_t$ using the same arguments in the proof of Theorem~\ref{main_theorem} with $A$ replaced by $K$, we have
\begin{align}\label{eq_12_admm}
&F(\widehat x_k) + G(\widehat w_k) + \langle K \widehat x_k - B \widehat w_k + b, y\rangle - \tfrac{\mu_d}{2}\|y\|^2 - \left[F( x) + G(w) + \langle K  x - B w + b, \widehat y_k\rangle - \tfrac{\mu_d}{2}\|\widehat  y_k\|^2\right]  \nn\\
 &\leq \tfrac{1}{\sum_{t=1}^k \eta_t}\left[\tfrac{\|x_0 - x\|^2}{2\beta} + \left(\tfrac{5\eta_2 L^2_{K, 1}}{4\mu_d}  - \tfrac{\eta_2}{2\eta_1}\right)\|x_1 - x_0\|^2 - \tfrac{\|x_{k+1} - \bar x_k\|^2}{4} - \tfrac{\|\bar x_{k+1} - x\|^2}{2\beta}\right] - \tfrac{\mu_d}{2}\|\widetilde y_k - y\|^2.
\end{align}
Using the fact $\|\widetilde y_k - y\|^2 = \|\widetilde y_k\|^2 + \|y\|^2 - 2 \langle \widetilde y_k, y\rangle$ and rearranging the terms in \eqref{eq_12_admm}, we have
\begin{align*}
&F(\widehat x_k) + G(\widehat w_k) + \langle K \widehat x_k - B \widehat w_k + b - \mu_d \widetilde y_k, y\rangle - \left[F( x) + G(w) + \langle K  x - B w + b, \widehat y_k\rangle \right] + \tfrac{\mu_d}{2}\|\widehat  y_k\|^2 + \tfrac{\mu_d}{2}\|\widetilde y_k\|^2 \nn\\
 &\leq \tfrac{1}{\sum_{t=1}^k \eta_t}\left[\tfrac{\|x_0 - x\|^2}{2\beta} + \left(\tfrac{5\eta_2 L^2_{K, 1}}{4\mu_d}  - \tfrac{\eta_2}{2\eta_1}\right)\|x_1 - x_0\|^2 - \tfrac{\|x_{k+1} - \bar x_k\|^2}{4} - \tfrac{\|\bar x_{k+1} - x\|^2}{2\beta}\right].
\end{align*}
Taking $(x, w) = (x^*, w^*)$ and invoking that $Kx^* - Bw^* +b=0$ due to the optimality of $(x^*, w^*, y^*)$ for \eqref{eq:bilinear_problem_ADMM}, we conclude that, for any $y\in \bbr^m$,
\begin{align}\label{eq_10_admm}
&F(\widehat x_k) + G(\widehat w_k) + \langle K \widehat x_k - B \widehat w_k + b - \mu_d \widetilde y_k, y\rangle - [F( x^*) + G(w^*)] + \tfrac{\mu_d}{2}\|\widehat  y_k\|^2 + \tfrac{\mu_d}{2}\|\widetilde y_k\|^2 \nn\\
 &\leq \tfrac{1}{\sum_{t=1}^k \eta_t}\left[\tfrac{\|x_0 - x^*\|^2}{2\beta} + \left(\tfrac{5\eta_2 L^2_{K, 1}}{4\mu_d} - \tfrac{\eta_2}{2\eta_1}\right)\|x_1 - x_0\|^2 \right].
\end{align}
This inequality implies
\begin{align*}
K \widehat x_k - B \widehat w_k + b - \mu_d \widetilde y_k = 0,
\end{align*}
since otherwise its left hand side can be unbounded. Therefore, we arrive at
\begin{align}\label{eq_10_admm_1}
&F(\widehat x_k) + G(\widehat w_k) - [F( x^*) + G(w^*)] + \tfrac{\mu_d}{2}\|\widehat  y_k\|^2 + \tfrac{\mu_d}{2}\|\widetilde y_k\|^2 \nn\\
 &\leq \tfrac{1}{\sum_{t=1}^k \eta_t}\left[\tfrac{\|x_0 - x^*\|^2}{2\beta} + \left(\tfrac{5\eta_2 L^2_{K, 1}}{4\mu_d}  - \tfrac{\eta_2}{2\eta_1}\right)\|x_1 - x_0\|^2 \right],
\end{align}
and
\begin{align*}
\|K \widehat x_k - B \widehat w_k + b\| \leq \mu_d \|\widetilde y_k\|,
\end{align*}
which provides guarantees for the optimality gap and the constraint violation, respectively. Finally, we provide an upper bound for $\|\widetilde y_k\|$. By taking $(x,w,y) = (x^*, w^*, y^*)$ in Ineq.~\eqref{eq_12_admm}, we have
\begin{align*}
&F(\widehat x_k) + G(\widehat w_k) + \langle K \widehat x_k - B \widehat w_k + b, y^*\rangle - \left[F( x^*) + G(w^*) + \langle K  x^* - B w^* + b, \widehat y_k\rangle \right]  \nn\\
&\quad  + \tfrac{\mu_d}{2}\|\widehat  y_k\|^2 - \tfrac{\mu_d}{2}\|y^*\|^2 + \tfrac{\mu_d}{2}\|\widetilde y_k - y^*\|^2\nn\\
 &\leq \tfrac{1}{\sum_{t=1}^k \eta_t}\left[\tfrac{\|x_0 - x^*\|^2}{2\beta} + \left(\tfrac{5\eta_2 L^2_{K, 1}}{4\mu_d}  - \tfrac{\eta_2}{2\eta_1}\right)\|x_1 - x_0\|^2 \right],
\end{align*}
which, in view of the fact that
$F(\widehat x_k) + G(\widehat w_k) + \langle K \widehat x_k - B \widehat w_k + b, y^*\rangle \geq F( x^*) + G(w^*) +\langle K  x^* - B w^* + b, \widehat y_k\rangle$ due to 
 $(x^*, w^*, y^*)$ being a saddle point of \eqref{eq:bilinear_problem_ADMM}, then implies that
\begin{align*}
\tfrac{\mu_d}{4}\|\widetilde y_k\|^2 - \mu_d\|y^*\|^2
& \leq  - \tfrac{\mu_d}{2}\|y^*\|^2 + \tfrac{\mu_d}{2}\|\widetilde y_k - y^*\|^2 \nn\\
 &\leq \tfrac{1}{\sum_{t=1}^k \eta_t}\left[\tfrac{\|x_0 - x^*\|^2}{2\beta} + \left(\tfrac{5\eta_2 L^2_{K, 1}}{4\mu_d}  - \tfrac{\eta_2}{2\eta_1}\right)\|x_1 - x_0\|^2 \right].
\end{align*}
As a consequence, we obtain
\begin{align*}
\mu_d \|\widetilde y_k\|
  &\leq \mu_d \sqrt{4\|y^*\|^2+\tfrac{4}{\mu_d(\sum_{t=1}^k \eta_t)}\left[\tfrac{\|x_0 - x^*\|^2}{2\beta} + \left(\tfrac{5\eta_2 L^2_{K, 1}}{4\mu_d}  - \tfrac{\eta_2}{2\eta_1}\right)\|x_1 - x_0\|^2 \right]}\\
&\leq 2\mu_d \|y^*\| + 2 \sqrt{\tfrac{\mu_d}{\sum_{t=1}^k \eta_t}\left[\tfrac{\|x_0 - x^*\|^2}{2\beta} + \left(\tfrac{5\eta_2 L^2_{K, 1}}{4\mu_d}  - \tfrac{\eta_2}{2\eta_1}\right)\|x_1 - x_0\|^2 \right]},
\end{align*}
which completes the proof.
\end{proof}
\vgap

Analogous to Corollary~\ref{main_corollary_2}, below we establish the convergence guarantees for AC-ADMM equipped with a concrete adaptive stepsize policy.
\begin{corollary} \label{main_corollary_admm}
In the premise of Theorem~\ref{theorem_admm}, suppose $\tau_1=0,~ \tau_2 = \mu_d$, and $\beta \in (0,  1 -{\tfrac{\sqrt{6}}{3}}]$. Also assume that $\eta_2 = \min \left\{{(1 - \beta) \eta_1},  \tfrac{\mu_d}{4 L_{K, 1}^2}\right\}$, and for $t \geq 3$
\begin{align}
&\eta_t = \min \left\{ \tfrac{4}{3}\eta_{t-1}, \tfrac{\tau_{t-2}+\mu_d}{\tau_{t-1}}\cdot \eta_{t-1}, \tfrac{\tau_{t-1}}{ 4L^2_{K, t-1}} \right\},\label{def_eta_t_admm}\\
&\tau_t = \tau_{t-1} + \tfrac{\mu_d}{2}\left[ \alpha  +(1-\alpha)\eta_t\cdot \tfrac{4L^2_{K, t-1}}{\tau_{t-1}}\right] \label{def_tau_t_admm}
\end{align}
for some absolute constant $\alpha \in (0,1]$.
Then we have for $t \geq 2$,
\begin{align}\label{eta_lower_bound_2_admm}
\eta_t \geq \tfrac{(3 + \alpha(t-3))\mu_d}{12 \widehat L^2_{K, t-1}},~~ \text{where} ~~\widehat L_{K, t} := \max\{\sqrt{\tfrac{\mu_d}{4(1-\beta)\eta_1}}, L_{K, 1},...,L_{K, t}\}.
\end{align}
Consequently, we have
\begin{align}
&F(\widehat x_k) + G(\widehat w_k) - F(x^*) - G(w^*)
 \leq \tfrac{12\widehat L^2_{K,k}}{\mu_d(6k + \alpha k(k-3))}\left[\tfrac{\|x_0 - x^*\|^2}{\beta} + \eta_2\left(\tfrac{5 L^2_{K, 1}}{2\mu_d} - \tfrac{1}{2\eta_1}\right)\|x_1 - x_0\|^2 \right], \label{optimality_gap_admm}\\
&\|K \widehat x_k - B \widehat w_k + b\| 
\leq 2\mu_d \|y^*\| + 2 \sqrt{\tfrac{12\widehat L^2_{K,k}}{6k + \alpha k(k-3)}\left[\tfrac{\|x_0 - x^*\|^2}{\beta} + \eta_2\left(\tfrac{5 L^2_{K, 1}}{2\mu_d} - \tfrac{1}{2\eta_1}\right)\|x_1 - x_0\|^2 \right]}.\label{constraint_violation_admm}
\end{align}
\end{corollary}
\begin{proof}
The proof of \eqref{eta_lower_bound_2_admm} follows from the same arguments in the proof of Corollary~\ref{main_corollary_2} with $A$ replaced by $K$. Then the results in \eqref{optimality_gap_admm} and \eqref{constraint_violation_admm}
follow by
by further applying the stepsize policy in \eqref{def_eta_t_admm} and \eqref{def_tau_t_admm} to Ineqs. \eqref{eq:opt_gap} and \eqref{eq:const_vio}. 
\end{proof}
\vgap

The bounds in Corollary~\ref{main_corollary_admm} merit some comments. Firstly, the selection of the algorithm parameters in AC-ADMM, as well as its rate of convergence, depends solely on  $K$ rather than $[K, -B]$. This advantage comes at the cost of solving a more complicated subproblem in step \eqref{primal_w_update}. 

Secondly, the stepsize policy in Corollary~\ref{main_corollary_admm} closely mirrors that of AC-PDHG in Corollary~\ref{main_corollary_2}, with $A$ replaced by $K$. Additionally, when $X$ is unbounded, we can also take an intial line search step to ensure that $\eta_1 \leq 2\mu_d/(5 L_{A, 1}^2)$ and consequently eliminate the term $\eta_2\left[5 L^2_{K, 1}/(2\mu_d) - 1/(2\eta_1)\right]\|x_1 - x_0\|^2$ in \eqref{optimality_gap_admm}-\eqref{constraint_violation_admm}.

Thirdly, we can derive from Corollary~\ref{main_corollary_admm} the iteration complexity of AC-ADMM for computing
an $(\epsilon_1,\epsilon_2)$-optimal solution of problem \eqref{bilinear_two_parts}, i.e.,
a pair of solution $(\bar x, \bar w) \in X \times W$ satisfying $F(\bar x) + G(\bar w) - F(x^*) - G(w^*) \leq \epsilon_1$ and $\|K \bar x - B \bar w + b\|\leq \epsilon_2$.
For simplicity, let us assume that $\eta_1 \leq \tfrac{2\mu_d}{5 L_{A, 1}^2}$ and use define the shorthand notation $D_{x^*}:=\|x_0 - x^*\|$ and $D_{y^*} := \|y^*\|$. Then the relations in \eqref{optimality_gap_admm} and \eqref{constraint_violation_admm} reduce to
\begin{align*}
&F(\widehat x_k) + G(\widehat w_k) - F(x^*) - G(w^*) \leq \mathcal{O}\left( \tfrac{\widehat L_{K,k}^2 D_{x^*}^2}{\mu_d k^2}\right),\\
&\|K \widehat x_k- B \widehat w_k + b\|  \leq \mu_d \|\widetilde y_k\| \leq \mathcal{O}\left( \mu_d\|y^*\| + \tfrac{\widehat L_{K,k}D_{x^*}}{k} \right).
\end{align*}
As a result, by setting $\mu_d = \mathcal{O}(\epsilon_2/D_{y^*})$, we can bound the iteration complexity of AC-ADMM by
\begin{align*}
&\mathcal{O}\left( \tfrac{\|K\|D_{x^*}}{\epsilon_2} +\tfrac{\|K\|D_{x^*}\sqrt{D_{y^*}}}{\sqrt{\epsilon_1\epsilon_2}}\right), 
\end{align*}
which reduces to
\begin{align*}
&\mathcal{O}\left( \tfrac{\|K\|D_{x^*} D_{y^*}}{\epsilon}\right)
\end{align*}
for the specific case when $(\epsilon_1, \epsilon_2)= (\epsilon, \epsilon/D_{y^*})$.
This latter complexity bound improves upon the previously best-known complexity  $$\mathcal{O}\left( \tfrac{\|K\|D_{x^*}^2 + D_{y^*}^2 + \|K\|D_{x^*}D_{y^*}}{\epsilon}\right)$$ for the preconditioned ADMM method in \cite{ouyang2015accelerated} in terms of their dependence on
$D_{x^*}$ and $D_{y^*}$. 

Finally, when implementing the AC-ADMM method, we need to substitute $D_{y^*}$ with its estimate $\widehat D_{y^*}$. 
Similar to the AC-PDHG method,
if the feasible region $X$ is bounded, i.e., $\max_{x, y}\|x-y\|\leq D_X$, we can employ a ``guess and check'' procedure as shown below.
Using an argument similar to the one for the AC-PDHG method, 
we can show that this procedure can 
output an $(\epsilon_1,\epsilon_2)$-optimal solution of problem \eqref{bilinear_two_parts}, given by $(\widehat x, \widehat w)$,
within at most $\mathcal{O}\left(\log(\tfrac{D_{y^*}}{\widehat D^{(0)}})\right)$ loops, and that the total number of AC-ADMM iterations
can be bounded by
$$\mathcal{O}\left( \tfrac{\|K\|D_{X}}{\epsilon_2}\log(\tfrac{D_{y^*}}{\widehat D^{(0)}}) +\tfrac{\|K\|D_{X}\sqrt{D_{y^*}}}{\sqrt{\epsilon_1\epsilon_2}}\right).$$

\begin{algorithm}[H]\caption*{A guess and check procedure for AC-ADMM:}\label{alg:ac_admm_guess}
	\begin{algorithmic}
		\State{\textbf{Input}: Choose $\widehat D^{(0)}$ such that $\widehat D^{(0)} \leq D_{y^*}$. Fix $\epsilon_1, \epsilon_2$.}
		\For{$i=0, 1, 2,\cdots$}
		\State{Let $\widehat D^{(i)} := \widehat D^{(0)}\cdot 2^{i}, ~\mu^{(i)} := \tfrac{\epsilon_2}{4D^{(i)}}$. Run Algorithm~\ref{alg:ac_admm} with $\mu_d = \mu^{(i)}$ until the first $\widehat k \geq 3$ such that
  \begin{align*}
  &\mathcal{E}_{1}(\widehat k) \leq \epsilon_1, \text{  and  }\min\left\{\|K \widehat x_{\widehat k}- B \widehat w_{\widehat k} + b\|, \mathcal{E}_{2}(\widehat k)\right\} \leq \epsilon_2,\\
  &\text{where } \mathcal{E}_{1}(k) := \tfrac{12 \widehat L_{K, k}^2 D_X^2}{\mu^{(i)}\beta(6k + \alpha k (k-3))}, \text{  }\mathcal{E}_{2}(k):= 4 \sqrt{\tfrac{12  \widehat L_{K, k}^2 D_X^2}{\beta(6k + \alpha k(k-3))}},  \text{ for $\widehat L_{K, k}$ defined in \eqref{eta_lower_bound_2_admm}.}
  \end{align*}
  }
  \State{If $\|K \widehat x_{\widehat k}- B \widehat w_{\widehat k} + b\| \leq \epsilon_2$: Terminate.}
		\EndFor
\State{\textbf{Output}: $(\widehat x, \widehat w) = (\widehat x_{\widehat k}, \widehat w_{\widehat k})$, $\widehat D_Y= \widehat D^{(i)}$.}
	\end{algorithmic}
\end{algorithm}

\section{Adaptive PDHG and ADMM with acceleration}\label{section:apdhg_aadmm}
The AC-PDHG and AC-ADMM methods developed in Sections \ref{sec:PDHG} and \ref{sec:ADMM} require the assumption that $f(x)$ and $F(x)$ are ``prox-friendly'' (see discussion above \eqref{eq:prox_x}). In this section, we propose accelerated versions of AC-PDHG and AC-ADMM to handle the situation when $f(x)$ and $F(x)$ are not ``prox-friendly'' but general smooth convex functions. 

\subsection{Auto-conditioned accelerated PDHG for bilinear saddle point problems}
In this subsection, we consider the bilinear saddle point problem problem~\eqref{eq:bilinear_saddle_point}, where $g:Y \rightarrow \bbr$ is prox-friendly, and $f:X\rightarrow \bbr$ is a smooth convex function satisfying
\begin{align}\label{eq:smoothness}
\tfrac{1}{2L_f}\|\nabla f(x) -\nabla f(y) \|^2\leq f(y) - f(x) - \langle \nabla f(x), y -x \rangle \leq \tfrac{L_f}{2}\|x -y \|^2, \quad \forall x,y \in X.
\end{align}
To solve this problem, we propose an auto-conditioned accelerated primal-dual hybrid gradient (AC-APDHG) method as shown in Algorithm~\ref{alg:ac_primal_dual_acce}. In comparison to AC-PDHG in Algorithm~\ref{alg:ac_primal_dual}, AC-APDHG introduces an additional sequence of search points $\{\widetilde x_t\}$ in \eqref{primal_output_series_acce}, another sequence taking as the weighted average of $\{x_t\}$. In the prox-mapping step \eqref{primal_prox-mapping_acce}, we compute the gradient at the sequence $\{ \widetilde x_t\}$ and use it to linearize the function $f(x)$. This design of AC-APDHG was inspired the accelerated primal-dual method in \cite{chen2014optimal} that incorporates an acceleration scheme into the PDHG method. However, by utilizing  the specific acceleration method in \cite{li2023simple}, AC-APDHG is not only simpler than the primal-dual method in \cite{chen2014optimal}, but also does not require the input of the operator norm $\|A\|$ and the Lipschitz constant $L_f$.

\begin{algorithm}[h]\caption{Auto-Conditioned Accelerated Primal-Dual Hybrid Gradient Method (AC-APDHG)}\label{alg:ac_primal_dual_acce}
	\begin{algorithmic}
		\State{\textbf{Input}: initial point $x_0 = \widetilde x_0 = \bar x_0 \in X$, $\widetilde y_0 \in Y$, nonnegative parameters $\beta_t \in (0, 1)$, $\eta_t \in \bbr_+$, $\tau_t \in \bbr_+$ and $\widetilde \tau_t \in \bbr_+$. \\
  Let 
  \begin{align}
  y_0 = \arg \min_{y \in Y} \left\{\langle - A x_0, y \rangle+ g(y) + \tfrac{\mu_d}{2}\|\widetilde y_0 - y\|^2  \right\}.\label{prox_update_0_acce}
  \end{align}}
		\For{$t=1,\cdots, k$}
		\State{
  \begin{align}
  x_t &= \arg \min_{x \in X} \left\{\eta_t\langle A^\top y_{t-1} + \nabla f(\widetilde x_{t-1}), x \rangle + \tfrac{1}{2}\|\bar x_{t-1} - x\|^2\right\}, \label{primal_prox-mapping_acce}\\
  \bar x_t &= (1-\beta_t) \bar x_{t-1} + \beta_t x_t, \label{primal_prox-center_acce}\\
  \widetilde x_t &= (x_{t} + \widetilde\tau_t \widetilde x_{t-1})/({1 + \widetilde \tau_t}), \label{primal_output_series_acce}\\
  y_t & = \arg \min_{y \in Y} \left\{\langle -Ax_t, y\rangle+ g(y) + \tfrac{\mu_d}{2}\|\widetilde y_0 - y\|^2  + \tfrac{\tau_t}{2}\|y_{t-1} - y\|^2\right\}. \label{dual_prox-mapping_acce}
  \end{align}
  }
		\EndFor
	\end{algorithmic}
\end{algorithm}

In place of the global constants $L_f$ and $\|A\|$, AC-APDHG utilizes local estimators of $\|A\|$, as defined in \eqref{def_L_A_t}, as well as local estimators of $L_f$ defined as
\begin{align}
L_{f, 1} := \tfrac{\|\nabla f(\widetilde x_1) - \nabla f(\widetilde x_0)\|}{\|\widetilde x_1 - \widetilde x_0\|}\label{def_L_f_1}
\end{align}
and
\begin{align}\label{def_L_f_t}
L_{f, t} := \tfrac{\|\nabla f(\widetilde x_t) - \nabla f(\widetilde x_{t-1})\|^2}{2[f(\widetilde x_{t-1}) - f(\widetilde x_t) - \langle\nabla f(\widetilde x_t), \widetilde x_{t-1} - \widetilde x_t \rangle]}, t\ge 2
\end{align}
to determine a few algorithmic parameters. In particular, the stepsize parameters ($\eta_t, \tau_t$ and $\tilde \tau_t$) used in the $t$-th iteration of AC-APDHG will be specified based on the estimators $L_{A, i}$ and $ L_{f, i}$, $i=1, t-1$, computed in the previous iterations,  but not on $L_{A, t}$ or $L_{f, t}$ obtained in the current iteration. 

\vspace{0.1in}

Analogous to Proposition~\ref{proposition_1}, we now specify the stepsize conditions that Algrithm~\ref{alg:ac_primal_dual_acce} needs to satisfy, along with a recursive relationship for the primal-dual gap function $Q(\widehat z_k, z)$.

\begin{proposition}\label{proposition_1_acce}
Let $\{x_t\}, \{\widetilde x_t\}, \{\bar x_t\}$ and $\{y_t\}$ be generated by Algorithm~\ref{alg:ac_primal_dual} with the parameters $\{\tau_t\}$, $\{\eta_t\}$, $\{\beta_t\}$ and $\{\widetilde \tau_t\}$ satisfying
\begin{align}
\tau_1 &= 0, ~\widetilde \tau_1 = 0, ~\beta_1 = 0,~\beta_t = \beta > 0,~t\geq 2, \label{cond_3_acce}\\
    \eta_2 &\leq \min\left\{{(1-\beta)\eta_1}, \left( \tfrac{4 L^2_{A,1}}{\mu_d} + 4 L_{f, 1}\right)^{-1} \right\},\label{cond_1_acce}\\ 
    \eta_t &\leq \min\left\{2(1-\beta)^2 \eta_{t-1}, \tfrac{\widetilde\tau_{t-2}+1}{\widetilde \tau_{t-1}}\eta_{t-1}, \tfrac{\tau_{t-2}+\mu_d}{ \tau_{t-1}}\eta_{t-1}, \left(\tfrac{4L_{A, t-1}^2}{\tau_{t-1}} +  \tfrac{4L_{f, t-1}}{\widetilde \tau_{t-1}}\right)^{-1}\right\}, ~t \geq 3.
    \label{cond_4_acce}
\end{align}
We have for any $z = (x,y) \in Z$,
\begin{align}\label{eq:proposition_1_acce}
&\left(\tsum_{t=1}^k \eta_t\right) \cdot \left(Q(\widehat z_k, z)- \tfrac{\mu_d}{2}\|\widetilde y_0 - y\|^2 + \tfrac{\mu_d}{2}\|\widetilde y_0 - \widehat y_k\|^2+ \tfrac{\mu_d}{2}\|\widetilde y_k - y\|^2\right) + \tfrac{\eta_2}{2\eta_{1}} \left[\|x_1 - \bar x_1\|^2 + \|x_2 - x_1\|^2 \right]\nn\\
 &\leq \tfrac{1}{2\beta}\|x_0 - x\|^2 - \tfrac{1}{2\beta}\|\bar x_{k+1} - x\|^2 + \tsum_{t=2}^{k+1}\widetilde\Delta_{t},
\end{align}
where $\widehat z_k$ is defined in \eqref{def_hat_x_k_y_k},  $\widetilde y_k$ is defined in \eqref{def_tilde_y_k}, and
\begin{align}\label{def_tilde_Delta_t}
\widetilde \Delta_t &:= \eta_t \langle A^\top (y_{t-1} - y_{t-2}), x_{t-1} - x_t\rangle - \tfrac{\eta_t \tau_{t-1}}{2}\|y_{t-2} - y_{t-1}\|^2 + \eta_t \langle \nabla f(\widetilde x_{t-1}) - \nabla f(\widetilde x_{t-2}), x_{t-1} - x_t \rangle \nn\\
& ~~\quad - \tfrac{\eta_t \widetilde \tau_{t-1}}{2 L_{f, t-1}}\|\nabla f(\widetilde x_{t-1}) - \nabla f(\widetilde x_{t-2})\|^2 - \tfrac{1}{2}\|x_t - \bar x_{t-1}\|^2.
\end{align}

\end{proposition}
\begin{proof}
First, by the optimality condition of step \eqref{primal_prox-mapping_acce}, we have for $t \geq 1$,
\begin{align}\label{eq1_acce}
\langle \eta_t [A^\top y_{t-1} + \nabla f(\widetilde x_{t-1})]+ x_t - \bar x_{t-1}, x - x_t\rangle \geq 0, \quad \forall x \in X,
\end{align}
and consequently for $t \geq 2$,
\begin{align}\label{eq2_acce}
\langle \eta_{t-1} [A^\top y_{t-2} + \nabla f(\widetilde x_{t-2})]+ x_{t-1} - \bar x_{t-2}, x_t - x_{t-1}\rangle \geq 0.
\end{align}
Using the relation $x_{t-1} - \bar x_{t-2} = \tfrac{1}{1-\beta_{t-1}} (x_{t-1}-\bar x_{t-1})$ due to \eqref{primal_prox-center_acce}, we can equivalently write \eqref{eq2_acce} as
\begin{align}\label{eq3_acce}
\langle \eta_t [A^\top y_{t-2} + \nabla f(\widetilde x_{t-2})] + \tfrac{\eta_t}{\eta_{t-1}(1-\beta_{t-1})}(x_{t-1} - \bar x_{t-1}), x_t - x_{t-1}\rangle \geq 0. 
\end{align}
By summing up \eqref{eq1_acce} and \eqref{eq3_acce} and rearranging the terms, we obtain for $t \geq 2$,
\begin{align*}
&\eta_t \langle A^\top y_{t-1} + \nabla f(\widetilde x_{t-1}), x - x_{t-1}\rangle + \eta_t \langle A^\top (y_{t-1} - y_{t-2}) + [\nabla f(\widetilde x_{t-1}) - \nabla f(\widetilde x_{t-2})], x_{t-1}-x_t \rangle\\ 
&+ \langle x_t - \bar x_{t-1}, x - x_t\rangle + \tfrac{\eta_t}{\eta_{t-1}(1-\beta_{t-1})} \langle x_{t-1}-\bar x_{t-1}, x_t - x_{t-1} \rangle \geq 0.
\end{align*}
Then, by using the same arguments in Ineqs.~\eqref{eq4} and \eqref{eq5} in the proof of Proposition~\ref{proposition_1}, we obtain an analog of \eqref{eq6}, stating that
\begin{align}\label{eq6_acce}
&\eta_t \langle A^\top y_{t-1} + \nabla f(\widetilde x_{t-1}), x_{t-1} - x\rangle + \tfrac{\eta_t}{2\eta_{t-1}(1-\beta_{t-1})} \left[\|x_{t-1} - \bar x_{t-1}\|^2 + \|x_t - x_{t-1}\|^2 \right]\nn\\
&\leq \tfrac{1}{2\beta_t}\|\bar x_{t-1} - x\|^2 - \tfrac{1}{2\beta_t} \|\bar x_t - x\|^2 +  \eta_t \langle A^\top (y_{t-1} - y_{t-2}) + [\nabla f(\widetilde x_{t-1}) - \nabla f(\widetilde x_{t-2})], x_{t-1}-x_t \rangle\nn\\
&\quad - \left[ \tfrac{1}{2} + \tfrac{1 - \beta_t}{2} - \tfrac{\eta_t}{2 \eta_{t-1}(1-\beta_{t-1})} \right]\|x_t - \bar x_{t-1}\|^2.
\end{align}
When $t \geq 3$, by utilizing the fact that $\|x_{t-1} - \bar x_{t-1}\|^2 + \|x_t - x_{t-1}\|^2 \geq \tfrac{1}{2}\|x_t - \bar x_{t-1}\|^2$ and rearranging the terms in Ineq.~\eqref{eq6_acce}, we have
\begin{align}\label{eq6_prime_acce}
&\eta_t \langle A^\top y_{t-1} + \nabla f(\widetilde x_{t-1}), x_{t-1} - x\rangle  + \tfrac{1}{2\beta_t} \|\bar x_t - x\|^2 \nn\\
&\leq \tfrac{1}{2\beta_t}\|\bar x_{t-1} - x\|^2 +  \eta_t \langle A^\top (y_{t-1} - y_{t-2}) + [\nabla f(\widetilde x_{t-1}) - \nabla f(\widetilde x_{t-2})], x_{t-1}-x_t \rangle\nn\\
&\quad - \left[ \tfrac{1}{2} + \tfrac{1 - \beta_t}{2} - \tfrac{\eta_t}{4 \eta_{t-1}(1-\beta_{t-1})} \right]\|x_t - \bar x_{t-1}\|^2\nn\\
&\overset{(i)}\leq \tfrac{1}{2\beta_t}\|\bar x_{t-1} - x\|^2 +  \eta_t \langle A^\top (y_{t-1} - y_{t-2}) + 
[\nabla f(\widetilde x_{t-1}) - \nabla f(\widetilde x_{t-2})], x_{t-1}-x_t \rangle- \tfrac{1}{2}\|x_t - \bar x_{t-1}\|^2,
\end{align}
where step (i) follows from $\eta_t \leq 2(1-\beta)^2 \eta_{t-1}$ in \eqref{cond_4_acce}. On the other hand, using
the fact $\widetilde x_{t-1}-x_{t-1} = \widetilde \tau_{t-1}(\widetilde x_{t-2}-\widetilde x_{t-1})$ due to \eqref{primal_output_series_acce} and
 the definition of $L_{f,t}$ in \eqref{def_L_f_t}, we have
\begin{align}
&\widetilde \tau_{t-1}f(\widetilde x_{t-1}) + \langle \nabla f(\widetilde x_{t-1}), \widetilde x_{t-1} - x \rangle  - \langle \nabla f(\widetilde x_{t-1}), x_{t-1}-x \rangle \nn \\
&= \widetilde \tau_{t-1}\left[f(\widetilde x_{t-2}) - \tfrac{\|\nabla f(\widetilde x_{t-1}) - \nabla f(\widetilde x_{t-2})\|^2}{2L_{f, t-1}}\right].\label{eq_L_f_t}
\end{align}
Also notice that Lemma~\ref{dual_three-point_lemma} still holds for Algorithm~\ref{alg:ac_primal_dual_acce}, since the dual variable $y_t$ are updated in the same way for both Algorithm~\ref{alg:ac_primal_dual} and Algorithm~\ref{alg:ac_primal_dual_acce}. Thus, by
combining Ineq.~\eqref{eq6_prime_acce}, Eq. \eqref{eq_L_f_t}, and Ineq.~\eqref{eq_dual_three-point_lemma} in Lemma~\ref{dual_three-point_lemma}, and rearranging the terms, we have for $t \geq 3$ and any $(x,y)\in X \times Y$,
\begin{align}
&\eta_t \left[\langle A x_{t-1}, y \rangle - \tfrac{\mu_d}{2}\|\widetilde y_0 - y\|^2 - g(y) - \langle Ax, y_{t-1}\rangle + \tfrac{\mu_d}{2}\|\widetilde y_0 - y_{t-1}\|^2 + g(y_{t-1}) \right]\nn\\
& + \eta_t  [\widetilde \tau_{t-1} f(\widetilde x_{t-1}) + \langle \nabla f(\widetilde x_{t-1}), \widetilde x_{t-1} - x \rangle - \widetilde \tau_{t-1} f(\widetilde x_{t-2})]  \nn\\
& \leq \tfrac{1}{2\beta_t}\|\bar x_{t-1} - x\|^2 - \tfrac{1}{2\beta_t}\|\bar x_{t} - x\|^2 + \tfrac{\eta_t\tau_{t-1}}{2}\|y_{t-2}- y\|^2 - \tfrac{\eta_t(\mu_d + \tau_{t-1})}{2}\|y_{t-1} - y\|^2\nn\\
& \quad + \eta_t \langle A^\top (y_{t-1} - y_{t-2}) + [\nabla f(\widetilde x_{t-1}) - \nabla f(\widetilde x_{t-2})], x_{t-1}-x_t \rangle - \tfrac{\eta_t \tau_{t-1}}{2}\|y_{t-2} - y_{t-1}\|^2\nn\\
&\quad -\tfrac{\eta_t \widetilde \tau_{t-1}}{2L_{f,t-1}}\|\nabla f(\widetilde x_{t-1}) - \nabla f(\widetilde x_{t-2})\|^2 - \tfrac{1}{2}\|x_t - \bar x_{t-1}\|^2.\label{eq9_acce}
\end{align}
For the case when $t=2$, by using the condition $\beta_1 = 0$ in \eqref{eq6_acce}, we obtain
\begin{align}
&\eta_2 \langle A^\top y_{1}, x_{1} - x\rangle + \eta_2\langle \nabla f(\widetilde x_1), x_1 - x \rangle  + \tfrac{\eta_t}{2\eta_{t-1}} \left[\|x_{1} - \bar x_{1}\|^2 + \|x_2 - x_{1}\|^2 \right]\nn\\
&\leq \tfrac{1}{2\beta_2}\|\bar x_{1} - x\|^2 - \tfrac{1}{2\beta_2} \|\bar x_2 - x\|^2 +  \eta_2 \langle A^\top (y_{1} - y_{0}) +  [\nabla f(\widetilde x_{1}) - \nabla f(\widetilde x_{0})], x_{1}-x_2 \rangle\nn\\
&\quad - \left[ \tfrac{1}{2} + \tfrac{1 - \beta_t}{2} - \tfrac{\eta_t}{2 \eta_{t-1}} \right]\|x_t - \bar x_{t-1}\|^2\nn\\
&\overset{(i)}\leq \tfrac{1}{2\beta_2}\|\bar x_{1} - x\|^2 - \tfrac{1}{2\beta_2} \|\bar x_2 - x\|^2 +  \eta_2 \langle A^\top (y_{1} - y_{0})+  [\nabla f(\widetilde x_{1}) - \nabla f(\widetilde x_{0})], x_{1}-x_2 \rangle -\tfrac{1}{2}\|x_2 - \bar x_{1}\|^2,\label{eq9_3_acce}
\end{align}
where step (i) follows from the condition $\eta_2 \leq (1-\beta)\eta_1$ in \eqref{cond_1_acce}.
Then, by combining the above inequality with Ineq.~\eqref{eq_dual_three-point_lemma} for $t=1$, and using $\tau_1=0$, we have
\begin{align}\label{eq9_2_acce}
&\eta_2 \left[\langle A x_{1}, y \rangle - \tfrac{\mu_d}{2}\|\widetilde y_0 - y\|^2 - g(y) - \langle Ax, y_{1}\rangle + \tfrac{\mu_d}{2}\|\widetilde y_0 - y_{1}\|^2 + g(y_{1}) \right]\nn\\
&   \quad + \eta_2\langle \nabla f(\widetilde x_1), x_1 - x \rangle + \tfrac{\eta_2}{2\eta_{1}} \left[\|x_1 - \bar x_1\|^2 + \|x_2 - x_1\|^2 \right] + \tfrac{\eta_2\mu_d}{2}\|y_{1} - y\|^2+ \tfrac{1}{2\beta_2} \|\bar x_2 - x\|^2\nn\\
&\leq \tfrac{1}{2\beta_2}\|\bar x_{1} - x\|^2 +  \eta_2 \langle A^\top (y_1-y_0)+  [\nabla f(\widetilde x_{1}) - \nabla f(\widetilde x_{0})], x_{1} - x_2 \rangle-\tfrac{1}{2}\|x_2 - \bar x_{1}\|^2.
\end{align}
By taking the summation of Ineq.~\eqref{eq9_2_acce} and the telescope sum of Ineq.~\eqref{eq9_acce} for $t=3,..., k+1$, using $\langle \nabla f(x_t), x_t - x\rangle \geq f(x_t) - f(x)$ due to the convexity of $f$, and noting that $\beta_t = \beta$ for $t \geq 2$, we obtain
\begin{align}\label{eq_final_acce}
& \tsum_{t=1}^k \eta_{t+1} \left[\langle A x_{t}, y \rangle - \tfrac{\mu_d}{2}\|\widetilde y_0 - y\|^2 - g(y) - \langle Ax, y_{t}\rangle + \tfrac{\mu_d}{2}\|\widetilde y_0 - y_{t}\|^2 + g(y_{t}) \right]\nn\\
& + \tsum_{t=1}^{k-1} \left[(\widetilde \tau_t + 1)\eta_{t+1} - \widetilde \tau_{t+1}\eta_{t+2}\right]\cdot (f(\widetilde x_t) - f(x)) + (\widetilde \tau_k +1)\eta_{k+1}(f(\widetilde x_k)-f(x)) + \tfrac{1}{2\beta}\|\bar x_{k+1} - x\|^2\nn\\
&  + \tfrac{\eta_2}{2\eta_{1}} \left[\|x_1 - \bar x_1\|^2 + \|x_2 - x_1\|^2 \right] + \tfrac{\eta_{k+1}(\mu_d + \tau_k)}{2}\|y_k - y\|^2 + \tsum_{t=1}^{k-1}\tfrac{\eta_{t+1}(\mu_d + \tau_t) - \eta_{t+2}\tau_{t+1}}{2}\|y_t - y\|^2\nn\\
& \leq \tfrac{1}{2\beta}\|x_0 - x\|^2 + \tsum_{t=2}^{k+1}\widetilde\Delta_{t},
\end{align}
where $\widetilde \Delta_t$ is defined in \eqref{def_tilde_Delta_t}.

Next, we provide lower bounds for the terms in the left-hand side of the above inequality. It follows from the condition $\eta_t \leq \eta_{t-1} (\widetilde \tau_{t-2}+1) / \widetilde \tau_{t-1}$, the fact $x_t = (\widetilde \tau_t +1)\widetilde x_t - \widetilde \tau_t \widetilde x_{t-1}$, the definition of $\widehat x_k$ in \eqref{def_hat_x_k_y_k}, and  the convexity of $f$ that
\begin{align*}
\tsum_{t=1}^{k-1} \left[(\widetilde \tau_t + 1)\eta_{t+1} - \widetilde \tau_{t+1}\eta_{t+2}\right]\cdot (f(\widetilde x_t) - f(x)) + (\widetilde \tau_k +1)\eta_{k+1}(f(\widetilde x_k)-f(x)) \geq (\tsum_{t=1}^k \eta_{t+1}) \cdot (f(\widehat x_k) -f(x)).
\end{align*}
Moreover, by utilizing the definition of $\widehat y_k$ in \eqref{def_hat_x_k_y_k}, the convexity of $\|\cdot\|^2$, $f$ and $g$, and the definition of $Q(\cdot, \cdot)$, we have
\begin{align*}
& \tsum_{t=1}^k \eta_{t+1} \left[\langle A x_{t}, y \rangle - \tfrac{\mu_d}{2}\|\widetilde y_0 - y\|^2 - g(y) - \langle Ax, y_{t}\rangle + \tfrac{\mu_d}{2}\|\widetilde y_0 - y_{t}\|^2 + g(y_{t}) \right]\nn\\
& + \tsum_{t=1}^{k-1} \left[(\widetilde \tau_t + 1)\eta_{t+1} - \widetilde \tau_{t+1}\eta_{t+2}\right]\cdot (f(\widetilde x_t) - f(x)) + (\widetilde \tau_k +1)\eta_{k+1}(f(\widetilde x_k)-f(x))\nn\\
& \geq (\tsum_{t=1}^k \eta_{t+1}) \cdot \big( Q(\widehat z_k, z) - \tfrac{\mu_d}{2}\|\widetilde y_0 - y\|^2 + \tfrac{\mu_d}{2}\|\widetilde y_0 - \widehat y_k\|^2\big).
\end{align*}
The desired result then follows by substituting the above two inequalities and Ineq.~\eqref{use_tilde_y_k} into ~\eqref{eq_final_acce}.
\end{proof}
\vgap

We are now ready to 
describe the main convergence properties of AC-APDHG for solving the bilinear saddle point problem \eqref{eq:bilinear_saddle_point} with smooth $f$. 
\begin{theorem}\label{main_theorem_acce}
Let $\{x_t\}, \{\widetilde x_t\}, \{\bar x_t\}$ and $\{y_t\}$ be generated by Algorithm~\ref{alg:ac_primal_dual_acce} with algorithmic parameters satisfying \eqref{cond_3_acce}-\eqref{cond_4_acce}.
Then we have for any $z = (x,y) \in Z$ and $k \geq 2$, 
\begin{align}\label{eq_main_theorem_acce}
&Q(\widehat z_k, z)- \tfrac{\mu_d}{2}\|\widetilde y_0 - y\|^2 + \tfrac{\mu_d}{2}\|\widetilde y_0 - \widehat y_k\|^2 + \tfrac{\mu_d}{2}\|\widetilde y_k - y\|^2 \nn\\
 &\leq \tfrac{1}{\sum_{t=1}^k \eta_t}\left[\tfrac{\|x_0 - x\|^2}{2\beta} + \left(\tfrac{5\eta_2 L^2_{A, 1}}{4\mu_d} + \tfrac{5\eta_2 L_{f,1}}{4} - \tfrac{\eta_2}{2\eta_1}\right)\|x_1 - x_0\|^2 - \tfrac{\|x_{k+1} - \bar x_k\|^2}{4} - \tfrac{\|\bar x_{k+1} - x\|^2}{2\beta}\right],
\end{align}
where $\widehat z_k$ and $\widetilde y_k$ are defined in \eqref{def_hat_x_k_y_k} and \eqref{def_tilde_y_k}, respectively.
\end{theorem}

\begin{proof}
In view of Proposition~\ref{proposition_1_acce}, it suffices to upper bound $\widetilde\Delta_t$. For the case $t\geq 3$, we have
\begin{align}\label{bound_delta_t_acce}
\widetilde\Delta_t &\overset{(i)}\leq \eta_t L_{A, t-1}\|y_{t-1} - y_{t-2}\|\|x_{t-1}-x_t\| - \tfrac{\eta_t \tau_{t-1}}{2}\|y_{t-2}-y_{t-1}\|^2 + \eta_t\|\nabla f(\widetilde x_{t-1}) - \nabla f(\widetilde x_{t-2})\|\|x_{t-1}-x_t\|\nn\\
& \quad - \tfrac{\eta_t \widetilde \tau_{t-1}}{2 L_{f, t-1}}\|\nabla f(\widetilde x_{t-1}) - \nabla f(\widetilde x_{t-2})\|^2 - \tfrac{1}{2}\|x_t - \bar x_{t-1}\|^2\nn\\
& \overset{(ii)}\leq \left(\tfrac{\eta_t L_{A, t-1}^2}{2 \tau_{t-1}} + \tfrac{\eta_t L_{f,t-1}}{2\widetilde \tau_{t-1}} \right) \|x_{t-1} - x_t\|^2 - \tfrac{1}{2}\|x_t - \bar x_{t-1}\|^2 \nn\\
& \overset{(iii)}\leq (1-\beta)^2\left(\tfrac{\eta_t L_{A, t-1}^2}{\tau_{t-1}} + \tfrac{\eta_t L_{f,t-1}}{\widetilde \tau_{t-1}} \right)\|x_{t-1}-\bar x_{t-2}\|^2 - \left( \tfrac{1}{2} - \tfrac{\eta_t L_{A, t-1}^2}{\tau_{t-1}} - \tfrac{\eta_t L_{f,t-1}}{\widetilde \tau_{t-1}} \right)\|x_t - \bar x_{t-1}\|^2\nn\\
& \overset{(iv)}\leq \tfrac{1}{4}\|x_{t-1}-\bar x_{t-2}\|^2 - \tfrac{1}{4}\|x_t - \bar x_{t-1}\|^2,
\end{align}
where step (i) follows from Cauchy-Schwarz inequality and the definition of $L_{f, t}$ in \eqref{def_L_f_t}, step (ii) follows from Young's inequality, step (iii) follows from the fact that 
\begin{align*}
\|x_{t-1}- x_t\|^2  = \|(1-\beta)(x_{t-1} - \bar x_{t-2}) + \bar x_{t-1} - x_t\|^2\leq 2(1-\beta)^2\|x_{t-1} - \bar x_{t-2}\|^2 + 2\|x_t - \bar x_{t-1}\|^2,
\end{align*}
and step (iv) follows from the conditions $\eta_t \leq \left(\tfrac{4L_{A, t-1}^2}{\tau_{t-1}} + \tfrac{4L_{f, t-1}}{\widetilde \tau_{t-1}}\right)^{-1}$ in \eqref{cond_4_acce}. For $\widetilde\Delta_2$, we have
\begin{align}\label{bound_delta_2_acce}
\Delta_2 & \overset{(i)}\leq \eta_2 L_{A,1} \|y_1-y_0\|\|x_1 - x_2\| + \eta_2 L_{f,1}\|\widetilde x_1 - \widetilde x_0\|\|x_1 - x_2\| - \tfrac{1}{2}\|x_2 -\bar x_1\|^2\nn\\
&\overset{(ii)}\leq \left(\tfrac{\eta_2 L^2_{A, 1}}{\mu_d} + \eta_2 L_{f,1}\right)\|x_1 - x_0\|(\|x_1 - x_0\| + \|x_2 - x_0\|) -\tfrac{1}{2}\|x_2 - x_0\|^2\nn\\
&\overset{(iii)}\leq \tfrac{5}{4}\left(\tfrac{\eta_2 L^2_{A, 1}}{\mu_d} + \eta_2 L_{f, 1}\right)\|x_1 - x_0\|^2+ \left(\tfrac{\eta_2 L^2_{A, 1}}{\mu_d} + \eta_2 L_{f, 1} -\tfrac{1}{2}\right)\|x_2 - x_0\|^2\nn\\
&\overset{(iv)}\leq \tfrac{5}{4}\left(\tfrac{\eta_2 L^2_{A, 1}}{\mu_d} + \eta_2 L_{f,1}\right)\|x_1 - x_0\|^2-\tfrac{1}{4}\|x_2 - \bar x_1\|^2,
\end{align}
where step (i) follows from Cauchy-Schwarz inequality and the definition of $L_{A,1}$ and $L_{f,1}$, step (ii) follows from Ineq.~\eqref{eq_bound_y_1_y_0} and the fact that $\widetilde x_1 = x_1$, $\widetilde x_0 = x_0 = \bar x_1$, step (iii) follows from Young's inequality, and step (iv) follows from the condition $\eta_2 \leq \left( 4 L^2_{A,1}/\mu_d + 4 L_{f, 1}\right)^{-1}$. By substituting Ineqs.~\eqref{bound_delta_t_acce} and \eqref{bound_delta_2_acce} into Ineq.~\eqref{eq:proposition_1_acce} and rearranging the terms, we obtain the desired result.
\end{proof}
\vgap

We provide below an adaptive stepsize policy
for AC-APDHG that only requires the local estimators $L_{A, t}$ and $L_{f,t}$ computed from previous iterations,
and the role of the hyper-parameter $\alpha$ is similar to that 
for AC-PDHG in Corollary \ref{main_corollary_2}. 

\begin{corollary} \label{main_corollary_2_acce}
In the premise of Theorem~\ref{main_theorem_acce}, suppose $\tau_1=0,~ \tau_2 = \mu_d$, $\beta \in (0,  1 -{\tfrac{\sqrt{6}}{3}}]$,   $\widetilde \tau_1 = 0$, and $\widetilde \tau_2 = 1$.  
Also assume that $\eta_2 = \min \left\{{(1 - \beta) \eta_1},  \tfrac{1}{4 L_{A, 1}^2/\mu_d + 4  L_{f,1}}\right\}$, and that for $t \geq 3$
\begin{align}
&\eta_t = \min \left\{ \tfrac{4}{3}\eta_{t-1}, \tfrac{\tau_{t-2}+\mu_d}{\tau_{t-1}}\cdot \eta_{t-1}, \tfrac{\tau_{t-1}}{4 \mu_d L_{f,t-1} + 4L^2_{A, t-1}} \right\},\label{def_eta_t_acce}\\
&\tau_t = \tau_{t-1} + \tfrac{\mu_d}{2}\left[\alpha +(1-\alpha) \eta_t\cdot\tfrac{4 \mu_d L_{f,t-1} + 4L^2_{A, t-1}}{\tau_{t-1}}\right],\label{def_tau_t_acce}\\
&\widetilde \tau_t = \tfrac{\tau_t}{\mu_d}, \label{def_tilde_tau_t_acce}
\end{align}
where $\alpha \in (0,1]$ denotes an absolute constant.
Then we have for $t \geq 2$,
\begin{align}\label{eta_lower_bound_2_acce}
\eta_t \geq \tfrac{3 + \alpha(t-3)}{12 \widehat L_{t-1}},~~ \text{where} ~~\widehat L_{t} := \max\{\tfrac{{1}}{4(1-\beta)\eta_1}, \tfrac{L_{A, 1}^2}{\mu_d} + L_{f,1},...,\tfrac{L_{A, t}^2}{\mu_d} +  L_{f,t}\}.
\end{align}
Consequently, we have
\begin{align}\label{eq:corollary_2_acce}
& Q(\widehat z_k, z) - \tfrac{\mu_d}{2}\|\widetilde y_0 - y\|^2 +\tfrac{\mu_d}{2}\|\widetilde y_0 - \widehat y_k\|^2 + \tfrac{\mu_d}{2}\|\widetilde y_k - y\|^2 \nn\\
&\leq \tfrac{12 \widehat L_k}{6k + \alpha k(k-3)}\left[\tfrac{\|x_0 - x\|^2 - \|\bar x_{k+1} - x\|^2}{\beta} + \eta_2\big(\tfrac{5 L^2_{A, 1}}{2\mu_d} + \tfrac{5 L_{f,1}}{2} - \tfrac{1}{\eta_1}\big)\|x_1 - x_0\|^2 \right].
\end{align}
Moreover, if $\max_{x\in X} \|x - x_0\|^2 \leq D_X^2$ and $\max_{y\in Y} \|y - \widetilde y_0\|^2 \leq D_Y^2$, then
\begin{align}\label{convergence_gap_bounded_acce}
\max_{z \in Z} Q(\widehat z_k, z)  \leq  \tfrac{12 \widehat L_k}{6k + \alpha k(k-3)} \left(\tfrac{1}{\beta} + \tfrac{5}{8}\right) D_X^2 + \tfrac{\mu_d}{2}D_Y^2.
\end{align}
\end{corollary}
\begin{proof}
By the condition $\widetilde \tau_t = \tau_t/\mu_d$ in \eqref{def_tilde_tau_t_acce}, the two requirements $\eta_t \leq \tfrac{\widetilde\tau_{t-2}+1}{\widetilde \tau_{t-1}}\eta_{t-1}$ and $ \eta_t \leq \tfrac{\tau_{t-2}+\mu_d}{ \tau_{t-1}}\eta_{t-1}$ in \eqref{cond_4_acce} become equivalent. As a consequence, all the results in Corollary~\ref{main_corollary_2_acce} follows from the same arguments used in the proof of Corollary~\ref{main_corollary_2}, with $L^2_{A, 1}/\mu_d$ replaced by $\widehat L_t$. 
\end{proof}
\vgap

By Corollary~\ref{main_corollary_2_acce}, when $f$ is a general smooth convex function, we can still obtain convergence guarantees similar to those in Corollary~\ref{main_corollary_2}, by applying AC-APDHG with stepsizes controlled by both the local norm $L_{A,t}$ and the local Lipschitz constant $L_{f,t}$. When both $X$ and $Y$ are bounded, the iteration complexity for AC-APDHG to compute an
$\epsilon$-optimal solution of problem~~\eqref{eq:bilinear_saddle_point} can be bounded by
\begin{align*}
\mathcal{O} \left(\sqrt{\tfrac{L_f D_X^2}{\epsilon}} + \sqrt{\tfrac{\|A\|^2 D_X^2}{\mu_d \epsilon}}\right) = \mathcal{O} \left(\sqrt{\tfrac{L_f D_X^2}{\epsilon}} + \tfrac{\|A\|D_X D_Y}{\epsilon}\right)
\end{align*}
with $\mu_d = \mathcal{O}(\epsilon/D_Y^2)$.
The above complexity bound is achieved without any prior knowledge of $D_X$, $L_f$ and $\|A\|$; the only required inputs are the desired accuracy $\epsilon$ and the diameter of the dual feasible region $D_Y$.
It is worth noting that this complexity bound is optimal in terms of the oracle $(\nabla f(x), Ax, A^\top y)$, as supported by the lower bound in~\cite{ouyang2021lower} under the single-oracle setting.
However, the number of gradient computations can be potentially further improved by using gradient sliding techniques~\cite{lan2016gradient, lan2023graph, lan2022accelerated} as we allow separate oracles for computing $\nabla f(x)$
and $(Ax, A^Ty)$.
It should also be observed that the guarantees in Corollary~\ref{main_corollary_2_acce} can be extended to the case where $X$ and $Y$ are potentially unbounded sets, using the same arguments as in Ineq. \eqref{Q_unbounded_feasible_region} and the associated remark.

\subsection{AC-APDHG for linearly constrained problems}
In this subsection, we extend the AC-APDHG method to solve the linearly constrained problem \eqref{linear_constrained_problem}, where $f$ is a smooth convex function satisfying \eqref{eq:smoothness}. 
The following result is analogous to Proposition~\ref{coro_linear_constraint} and 
we skip its proof due to its similarity to the argument used to prove Proposition~\ref{coro_linear_constraint}.
\begin{proposition}\label{coro_linear_constraint_acce}
Consider Algorithm~\ref{alg:ac_primal_dual_acce} applied to problem \eqref{linear_constrained_problem} with $\widetilde y_0 = \mathbf{0} \in \bbr^m$ and algorithm parameters selected according to Corollary~\ref{main_corollary_2_acce}. Then we have
\begin{align}
f(\widehat x_k) - f(x^*)
& \leq \tfrac{12 \widehat L_k}{6k + \alpha k(k-3)} \left[\tfrac{\|x_0 - x^*\|^2}{\beta} + \left(\tfrac{5\eta_2 L^2_{A, 1}}{2\mu_d} + \tfrac{5\eta_2 L_{f, 1}}{2} - \tfrac{\eta_2}{2\eta_1}\right)\|x_1 - x_0\|^2 \right], \label{optimality_gap_acce}\\
\|A \widehat x_k - b\| 
&\leq 2\mu_d \|y^*\| + 2 \sqrt{\tfrac{12 \mu_d \widehat L_k}{6k + \alpha k(k-3)}\left[\tfrac{\|x_0 - x^*\|^2}{\beta} + \left(\tfrac{5\eta_2 L^2_{A, 1}}{2\mu_d} + \tfrac{5\eta_2 L_{f, 1}}{2} - \tfrac{\eta_2}{2\eta_1}\right)\|x_1 - x_0\|^2 \right]}, \label{constraint_violation_acce}
\end{align}
where $\widehat x_k$ and $\widehat y_k$ are defined in \eqref{def_hat_x_k_y_k}, $\widetilde y_k$ is defined in \eqref{def_tilde_y_k}, and $\widehat L_k$ is defined in \eqref{eta_lower_bound_2_acce}.
\end{proposition}
\vgap

 Now let us take a closer look at the convergence guarantees in Proposition~\ref{coro_linear_constraint_acce}. 
For simplicity, we assume that a line search step is applied in the first iteration to eliminate the last term 
\[
\left(\tfrac{5\eta_2 L^2_{A, 1}}{4\mu_d} + \tfrac{5\eta_2 L_{f, 1}}{4} - \tfrac{\eta_2}{2\eta_1}\right)\|x_1 - x_0\|^2\]
appearing in both \eqref{optimality_gap_acce} and
\eqref{constraint_violation_acce}. As a consequence, we have
\begin{align*}
&f(\widehat x_k) - f(x^*) \leq \mathcal{O}\left( \tfrac{\widehat L_{f, k} D_{x^*}^2}{k^2}  + \tfrac{\widehat L_{A,k}^2 D_{x^*}^2}{\mu_d k^2}\right),\\
&\|A \widehat x_k - b\|  \leq \mu_d \|\widetilde y_k\| \leq \mathcal{O}\left( \mu_d\|y^*\| + \sqrt{\mu_d \widehat L_{f,k}}\cdot\tfrac{D_{x^*}}{k} + \tfrac{\widehat L_{A,k}D_{x^*}}{k} \right)
\end{align*}
where $D_{x^*} := \|x_0 - x^*\|$ and $D_{y^*} := \|y^*\|$.
Therefore, by setting $\mu_d = \mathcal{O}(\epsilon_2/D_{y^*})$, we can bound the iteration complexity of AC-APDHG for computing an
$(\epsilon_1, \epsilon_2)$-optimal solution of problem \eqref{linear_constrained_problem} by
\begin{align*}
\mathcal{O}\left( \sqrt{\tfrac{L_f D_{x^*}^2}{\epsilon_1}}+ \sqrt{\tfrac{L_f D_{x^*}^2}{\epsilon_2 D_{y^*}}}+ \tfrac{\|A\|D_{x^*}}{\epsilon_2} +\tfrac{\|A|D_{x^*}\sqrt{D_{y^*}}}{\sqrt{\epsilon_1\epsilon_2}}\right),
\end{align*}
which reduces to
\begin{align*}
\mathcal{O}\left( \sqrt{\tfrac{L_F D_{x^*}^2}{\epsilon}}+\tfrac{\|A\|D_{x^*} D_{y^*}}{\epsilon}\right)
\end{align*}
when $(\epsilon_1, \epsilon_2)= (\epsilon, \epsilon/D_{y^*})$.
Moreover, if $X$ is a bounded set, we can employ the ``guess and check'' procedure, as outlined in Section~\ref{subsec:linear_constraint}, to search for a proper estimate of $D_{y^*}$.

\subsection{Auto-conditioned accelerated ADMM}
In this subsection, we consider again the special linearly constrained problem~\eqref{bilinear_two_parts} but assume  that $F:X\rightarrow \bbr$ is a general smooth convex function satisfying
\begin{align}
\tfrac{1}{2L_F}\|\nabla F(x) -\nabla F(y) \|^2\leq F(y) - F(x) - \langle \nabla F(x), y -x \rangle \leq \tfrac{L_F}{2}\|x -y \|^2, \quad \forall x,y \in X.
\end{align}
Similar to the AC-APDHG method, we propose the following auto-conditioned accelerated alternative direction of multipliers (AC-AADMM) by properly modifying the AC-ADMM method in Algorithm~\ref{alg:ac_admm}. The key difference between AC-ADMM and AC-AADMM exists in that AC-AADMM uses an additional sequence $\{\widetilde x_t\}$ for gradient computation and approximation of the function $F(x)$ in the prox-mapping update \eqref{primal_x_update_acce}. 

\begin{algorithm}[h]\caption{Auto-Conditioned Accelerated Alternating Direction Method of Multipliers (AC-AADMM)}\label{alg:ac_admm_acce}
	\begin{algorithmic}
		\State{\textbf{Input}: initial point $x_0 = \bar x_0 = \widetilde x_0 \in X$, nonnegative parameters $\beta_t \in (0, 1)$, $\eta_t \in \bbr_+$, $\tau_t \in \bbr_+$ and $\widetilde \tau_t \in \bbr_+$. 
  Let 
  \begin{align}
  w_0 &= \arg \min_{w \in W}  \left\{G(w) +  \tfrac{1}{2\mu_d}\|Bw - K x_0 -b\|^2 \right\}, \label{primal_update_0_acce}\\
  y_0 &= ( K x_0 - B w_0 +b)/{\mu_d}.\label{dual_update_0_acce}
  \end{align}}
		\For{$t=1,\cdots, k$}
		\State{
  \begin{align}
  x_t &= \arg \min_{x \in X} \left\{\eta_t \langle K^\top y_{t-1} + \nabla F(\widetilde x_{t-1} , x \rangle  + \tfrac{1}{2}\|\bar x_{t-1} - x\|^2\right\}, \label{primal_x_update_acce}\\
  \bar x_t &= (1-\beta_t) \bar x_{t-1} + \beta_t x_t, \label{primal_x-center_acce}\\
  \widetilde x_t &= (x_{t} + \widetilde\tau_t \widetilde x_{t-1})/({1 + \widetilde \tau_t}), \label{primal_x_series_acce}\\
  w_t &= \arg \min_{w \in W}  \left\{G(w) + (\tau_t + \mu_d)^{-1}\left(-\tau_t \langle y_{t-1}, Bw \rangle  + \tfrac{1}{2}\|Bw - K x_t -b\|^2 \right)\right\}, \label{primal_w_update_acce}\\
  y_t & = \left[\tau_t y_{t-1}  - (B w_t - Kx_t - b)\right]/(\tau_t + \mu_d). \label{dual_update_acce}
  \end{align}
  }
		\EndFor
	\end{algorithmic}
\end{algorithm}

AC-AADMM is also related to the accelerated ADMM method suggested in \cite{ouyang2015accelerated}, since both these methods incorporate acceleration schemes
into the ADMM method. However, AC-AADMM does not require the input of global problem parameters including $\|K\|$ and $L_F$ as required by the accelerated ADMM method in \cite{ouyang2015accelerated}. 
Similar to AC-APDHG, the selection of stepsize parameters $\eta_t$, $\tau_t$ and $\widetilde \tau_t$ in AC-AADMM will be determined by previously computed local estimators $L_{K, i}$ and $L_{F, i}$, $i =1, \ldots, t-1$, where $L_{K, t}$ is defined in \eqref{def_L_K_t} and
\begin{align}
L_{F, t} := \begin{cases}
    \tfrac{\|\nabla F(\widetilde x_1) - \nabla F(\widetilde x_0)\|}{\|\widetilde x_1 - \widetilde x_0\|}, & \text{for $t=1$},\\
     \tfrac{\|\nabla F(\widetilde x_t) - \nabla F(\widetilde x_{t-1})\|^2}{2[F(\widetilde x_{t-1}) - F(\widetilde x_t) - \langle\nabla F(\widetilde x_t), \widetilde x_{t-1} - \widetilde x_t \rangle]}, &\text{for $t\geq 2$}.\label{def_L_F_t}
\end{cases}
\end{align}

The following result, which is an analog of Theorem~\ref{theorem_admm}, establishes the convergence properties for AC-AADMM applied to problem~\eqref{bilinear_two_parts} in terms of both optimality gap and constraint violation. 
\begin{theorem}\label{theorem_admm_acce}
Let $\{x_t\}, \{\widetilde x_t\}, \{\bar x_t\}, \{w_t\}$ and $\{y_t\}$ be generated by Algorithm~\ref{alg:ac_admm} with the parameters $\{\tau_t\}$, $\{\eta_t\}$, $\{\widetilde \tau_t\}$ and $\{\beta_t\}$ satisfying conditions \eqref{cond_3_acce} and
\begin{align}
    \eta_2 &\leq \min\left\{{(1-\beta)\eta_1}, \left( \tfrac{4 L^2_{K,1}}{\mu_d} + 4 L_{F, 1}\right)^{-1} \right\},\label{cond_1_admm_acce}\\ 
    \eta_t &\leq \min\left\{2(1-\beta)^2 \eta_{t-1}, \tfrac{\widetilde\tau_{t-2}+1}{\widetilde \tau_{t-1}}\eta_{t-1}, \tfrac{\tau_{t-2}+\mu_d}{ \tau_{t-1}}\eta_{t-1}, \left(\tfrac{4L_{K, t-1}^2}{\tau_{t-1}} + \tfrac{4L_{F, t-1}}{\widetilde \tau_{t-1}}\right)^{-1}\right\}, ~t \geq 3.
    \label{cond_4_admm_acce}
\end{align}
Then for any $k\geq 2$, we have 
\begin{align}
F(\widehat x_k) + G(\widehat w_k) &- [F( x^*) + G(w^*)] \nn\\
& \leq \tfrac{1}{\sum_{t=1}^k \eta_t}\left[\tfrac{\|x_0 - x^*\|^2}{2\beta} + \left(\tfrac{5\eta_2 L^2_{K, 1}}{4\mu_d} + \tfrac{5\eta_2 L_{F, 1}}{4} - \tfrac{\eta_2}{2\eta_1}\right)\|x_1 - x_0\|^2 \right], \label{eq:opt_gap_acce}\\
\|K \widehat x_k - B \widehat w_k + b\| 
&\leq 2\mu_d \|y^*\| + 2 \sqrt{\tfrac{\mu_d}{\sum_{t=1}^k \eta_t}\left[\tfrac{\|x_0 - x^*\|^2}{2\beta} + \left(\tfrac{5\eta_2 L^2_{K, 1}}{4\mu_d} + \tfrac{5\eta_2 L_{F, 1}}{4} - \tfrac{\eta_2}{2\eta_1}\right)\|x_1 - x_0\|^2 \right]},\label{eq:const_vio_acce}
\end{align}
where $\widehat x_k$ and $\widehat w_k$ are defined in \eqref{def_hat_x_k_y_k_admm}, and $\widetilde y_k$ is defined in \eqref{def_tilde_y_k}.
\end{theorem}
\begin{proof}
First, recall that Ineq. \eqref{eq6_prime_acce} in the proof of Proposition \ref{proposition_1_acce} was derived from the update rules \eqref{primal_prox-mapping_acce} and \eqref{primal_prox-center_acce} in Algorithm~\ref{alg:ac_primal_dual_acce}. Since the update rules \eqref{primal_x_update_acce} and \eqref{primal_x-center_acce} in AC-AADMM resembles \eqref{primal_prox-mapping_acce} and \eqref{primal_prox-center_acce}, respectively, with $A$ and $f$ replaced by $K$ and $F$,  we can show that for any $t \geq 3$,
\begin{align}\label{eq6_prime_admm_acce}
&\eta_t \langle K^\top y_{t-1} +\nabla F(\widetilde x_{t-1}), x_{t-1} - x\rangle  + \tfrac{1}{2\beta_t} \|\bar x_t - x\|^2 \nn\\
&\leq \tfrac{1}{2\beta_t}\|\bar x_{t-1} - x\|^2 +  \eta_t \langle K^\top (y_{t-1} - y_{t-2}) + [\nabla F(\widetilde x_{t-1}) - \nabla F(\widetilde x_{t-2})], x_{t-1}-x_t \rangle- \tfrac{1}{2}\|x_t - \bar x_{t-1}\|^2.
\end{align}
Also notice that Lemma~\ref{lemma_optimality_conditions} still holds for AC-AADMM since the updates of $w_t$ and $y_t$ remain unchanged. By combining Ineq.~\eqref{eq6_prime_admm_acce} and Ineq.~\eqref{optimality_condition_w} in Lemma~\ref{lemma_optimality_conditions} and rearranging the terms, we obtain
\begin{align*}
&\eta_t \left[ \langle \nabla F(\widetilde x_{t-1}), x_{t-1} - x\rangle + G(w_{t-1}) - G(w)+ \langle Kx_{t-1} - Bw_{t-1}+b, y \rangle - \langle Kx-Bw+b, y_{t-1}\rangle \right]  \\
&\leq \tfrac{1}{2\beta_t}\|\bar x_{t-1} - x\|^2 - \tfrac{1}{2\beta_t} \|\bar x_t - x\|^2 +  \eta_t \langle K^\top (y_{t-1} - y_{t-2}) + [\nabla F(\widetilde x_{t-1}) - \nabla F(\widetilde x_{t-2})], x_{t-1}-x_t \rangle- \tfrac{1}{2}\|x_t - \bar x_{t-1}\|^2\\
&\quad  + \eta_t  \langle - Kx_{t-1} + Bw_{t-1} -b, y_{t-1}, y\rangle.
\end{align*}
Substituting~\eqref{eq_dual_three-point_lemma_admm} into the above inequality and rearranging the terms, we have
\begin{align*}
&\eta_t \big[ \langle \nabla F(\widetilde x_{t-1}), x_{t-1} - x\rangle+ G(w_{t-1}) - G(w) + \langle Kx_{t-1} - Bw_{t-1}+b, y \rangle - \langle Kx-Bw+b, y_{t-1}\rangle \\
&\quad - \tfrac{\mu_d}{2}\|y\|^2 + \tfrac{\mu_d}{2}\|y_{t-1}\|^2 \big] + \tfrac{\eta_t(\mu_d+\tau_{t-1})}{2}\|y_{t-1} - y\|^2 - \tfrac{\eta_t \tau_{t-1}}{2}\|y_{t-2}-y\|^2  \\
&\leq \tfrac{1}{2\beta_t}\|\bar x_{t-1} - x\|^2 - \tfrac{1}{2\beta_t} \|\bar x_t - x\|^2  +  \eta_t \langle K^\top (y_{t-1} - y_{t-2}) + [\nabla F(\widetilde x_{t-1}) - \nabla F(\widetilde x_{t-2})], x_{t-1}-x_t \rangle\\
&\quad - \tfrac{1}{2}\|x_t - \bar x_{t-1}\|^2 - \tfrac{\eta_t \tau_{t-1}}{2}\|y_{t-1} - y_{t-2}\|^2.
\end{align*}
Moreover, we derive from the definition of $L_{F,t}$ in \eqref{def_L_F_t}
and 
the fact $\widetilde x_{t-1}-x_{t-1} = \widetilde \tau_{t-1}(\widetilde x_{t-2}-\widetilde x_{t-1})$ due to \eqref{primal_x_series_acce} that 
\begin{align*}
&\widetilde \tau_{t-1}F(\widetilde x_{t-1}) + \langle \nabla F(\widetilde x_{t-1}), \widetilde x_{t-1} - x \rangle - \langle \nabla F(\widetilde x_{t-1}), x_{t-1}-x \rangle \\
&= \widetilde \tau_{t-1}\left[F(\widetilde x_{t-2}) - \tfrac{\|\nabla F(\widetilde x_{t-1}) - \nabla F(\widetilde x_{t-2})\|^2}{2L_{F, t-1}}\right] .
\end{align*}
Combining the above two relations and rearranging the terms, we obtain
\begin{align}
&\eta_t \left[G(w_{t-1}) - G(w) + \langle Kx_{t-1} - Bw_{t-1}+b, y \rangle - \langle Kx-Bw+b, y_{t-1}\rangle - \tfrac{\mu_d}{2}\|y\|^2 + \tfrac{\mu_d}{2}\|y_{t-1}\|^2 \right]  \nn\\
& + \eta_t[\widetilde \tau_{t-1} F(\widetilde x_{t-1}) + \langle \nabla F(\widetilde x_{t-1}), \widetilde x_{t-1} - x \rangle - \widetilde \tau_{t-1} F(\widetilde x_{t-2})]+ \tfrac{1}{2\beta_t} \|\bar x_t - x\|^2 + \tfrac{\eta_t(\mu_d+\tau_{t-1})}{2}\|y_{t-1} - y\|^2 - \tfrac{\eta_t \tau_{t-1}}{2}\|y_{t-2}-y\|^2   \nn\\
&\leq \tfrac{1}{2\beta_t}\|\bar x_{t-1} - x\|^2 +  \eta_t \langle K^\top (y_{t-1} - y_{t-2}) +[ \nabla F(\widetilde x_{t-1}) - \nabla F(\widetilde x_{t-2})], x_{t-1}-x_t \rangle- \tfrac{1}{2}\|x_t - \bar x_{t-1}\|^2\nn\\
&\quad - \tfrac{\eta_t \tau_{t-1}}{2}\|y_{t-1} - y_{t-2}\|^2  - \tfrac{\eta_t \widetilde \tau_{t-1}}{2L_{F, t-1}}\|\nabla F(\widetilde x_{t-1}) - \nabla F(\widetilde x_{t-2})\|^2.\label{eq9_admm_acce}
\end{align}
On the other hand, for $t=2$, similar to Ineq.~\eqref{eq9_3_acce} in the proof of Proposition~\ref{proposition_1_acce}, we can show that
\begin{align}\label{eq6_admm_acce}
&\eta_2 \langle K^\top y_{1} + \nabla F(\widetilde x_{1}), x_{1} - x\rangle  + \tfrac{1}{2\beta_2} \|\bar x_2 - x\|^2 + \tfrac{\eta_2}{2\eta_{1}(1-\beta_{1})} \left[\|x_{1} - \bar x_{1}\|^2 + \|x_2 - x_{1}\|^2 \right]\nn\\
&\leq \tfrac{1}{2\beta_2}\|\bar x_{1} - x\|^2 +  \eta_2 \langle K^\top (y_{1} - y_{0}) + [\nabla F(\widetilde x_{1}) - \nabla F(\widetilde x_{0})], x_{1}-x_2 \rangle  - \tfrac{1}{2}\|x_2 - \bar x_{1}\|^2.
\end{align}
Then, combining the above inequality with Ineqs.~\eqref{optimality_condition_w} and \eqref{eq_dual_three-point_lemma_admm} in Lemma~\ref{lemma_optimality_conditions}, using $\tau_1 = \widetilde \tau_1  = 0$, and rearranging the terms yield
\begin{align}
&\eta_2 \left[G(w_{1}) - G(w) + \langle Kx_{1} - Bw_{1}+b, y \rangle - \langle Kx-Bw+b, y_{1}\rangle - \tfrac{\mu_d}{2}\|y\|^2 + \tfrac{\mu_d}{2}\|y_{1}\|^2 \right]  \nn\\
& + \eta_t\langle \nabla F(\widetilde x_1), \widetilde x_1 - x \rangle + \tfrac{1}{2\beta_2} \|\bar x_2 - x\|^2 + \tfrac{\eta_2\mu_d}{2}\|y_{1} - y\|^2 + \tfrac{\eta_2}{2\eta_{1}(1-\beta_{1})} \left[\|x_{1} - \bar x_{1}\|^2 + \|x_2 - x_{1}\|^2 \right]   \nn\\
&\leq \tfrac{1}{2\beta_2}\|x_0 - x\|^2 +  \eta_2 \langle K^\top (y_{1} - y_{0}) + [\nabla F(\widetilde x_{1}) - \nabla F(\widetilde x_{0})], x_{1}-x_2 \rangle- \tfrac{1}{2}\|x_2 - \bar x_{1}\|^2  - \tfrac{1}{2}\|x_2 - \bar x_{1}\|^2. \label{eq9_2_admm_acce}
\end{align}
By taking the summation of Ineq.~\eqref{eq9_2_admm_acce} and the telescope sum of Ineq.~\eqref{eq9_admm_acce} for $t=3,..., k+1$, using $\langle \nabla F(x_t), x_t - x\rangle \geq F(x_t) - F(x)$ due to the convexity of $F$, and noting that $\beta_t = \beta$ for $t \geq 2$, we obtain
\begin{align}\label{eq_final_admm_acce}
& \tsum_{t=1}^k \eta_{t+1} \left[ G(w_t) + \langle K x_{t} - B w_t + b, y \rangle - \tfrac{\mu_d}{2}\|y\|^2 - G(w) - \langle Kx - Bw +b, y_{t}\rangle + \tfrac{\mu_d}{2}\|y_{t}\|^2 \right]\nn\\
& + \tsum_{t=1}^{k-1} \left[(\widetilde \tau_t + 1)\eta_{t+1} - \widetilde \tau_{t+1}\eta_{t+2}\right]\cdot (F(\widetilde x_t) - F(x)) + (\widetilde \tau_k +1)\eta_{k+1}(F(\widetilde x_k)-F(x))+ \tfrac{1}{2\beta}\|\bar x_{k+1} - x\|^2\nn\\
&  + \tfrac{\eta_2}{2\eta_{1}} \left[\|x_1 - \bar x_1\|^2 + \|x_2 - x_1\|^2 \right] + \tfrac{\eta_{k+1}(\mu_d + \tau_k)}{2}\|y_k - y\|^2 + \tsum_{t=1}^{k-1}\tfrac{\eta_{t+1}(\mu_d + \tau_t) - \eta_{t+2}\tau_{t+1}}{2}\|y_t - y\|^2\nn\\
& \leq \tfrac{1}{2\beta}\|x_0 - x\|^2 + \tsum_{t=2}^{k+1}\widetilde\Delta_{t},
\end{align}
where
\begin{align}
\widetilde\Delta_t &:= \eta_t \langle K^\top (y_{t-1} - y_{t-2}), x_{t-1} - x_t\rangle - \tfrac{\eta_t \tau_{t-1}}{2}\|y_{t-2} - y_{t-1}\|^2 +  \eta_t \langle \nabla F(\widetilde x_{t-1}) - \nabla F(\widetilde x_{t-2}), x_{t-1} - x_t \rangle \nn\\
& ~~\quad - \tfrac{\eta_t \widetilde \tau_{t-1}}{2 L_{F, t-1}}\|\nabla F(\widetilde x_{t-1}) - \nabla F(\widetilde x_{t-2})\|^2 - \tfrac{1}{2}\|x_t - \bar x_{t-1}\|^2\label{def_Delta_t_admm_acce}
\end{align}
Now by using exactly the same arguments in the proof of Proposition \ref{proposition_1_acce} (after Ineq.~\eqref{eq_final_acce}), we obtain
\begin{align}\label{eq:proposition_2_acce}
&\left(\tsum_{t=1}^k \eta_t\right) \cdot \left(Q(\widehat z_k, z)- \tfrac{\mu_d}{2}\|y\|^2 +\tfrac{\mu_d}{2}\| 
\widehat y_{k}\|^2 + \tfrac{\mu_d}{2}\|\widetilde y_k - y\|^2\right) + \tfrac{\eta_2}{2\eta_{1}} \left[\|x_1 - \bar x_1\|^2 + \|x_2 - x_1\|^2 \right]\nn\\
 &\leq \tfrac{1}{2\beta}\|x_0 - x\|^2 - \tfrac{1}{2\beta}\|\bar x_{k+1} - x\|^2  + \tsum_{t=2}^{k+1}\widetilde\Delta_{t}.
\end{align}
Next, after bounding $\widetilde\Delta_t$ using the same arguments in the proof of Theorem~\ref{main_theorem_acce} with $f$ replaced by $F$ and $A$ replaced by $K$, we arrive at 
\begin{align}\label{eq_12_admm_acce}
&F(\widehat x_k) + G(\widehat w_k) + \langle K \widehat x_k - B \widehat w_k + b, y\rangle - \tfrac{\mu_d}{2}\|y\|^2 - \left[F( x) + G(w) + \langle K  x - B w + b, \widehat y_k\rangle - \tfrac{\mu_d}{2}\|\widehat  y_k\|^2\right] \nn\\
 &\leq \tfrac{1}{\sum_{t=1}^k \eta_t}\left[\tfrac{\|x_0 - x\|^2}{2\beta} + \left(\tfrac{5\eta_2 L^2_{K, 1}}{4\mu_d} +\tfrac{5\eta_2 L_{F, 1}}{4} - \tfrac{\eta_2}{2\eta_1}\right)\|x_1 - x_0\|^2 - \tfrac{\|x_{k+1} - \bar x_k\|^2}{4} - \tfrac{\|\bar x_{k+1} - x\|^2}{2\beta}\right] - \tfrac{\mu_d}{2}\|\widetilde y_k - y\|^2 .
\end{align}
Then the results in \eqref{eq:opt_gap_acce} and \eqref{eq:const_vio_acce} follow immediately by utilizing the same arguments after \eqref{eq_12_admm} in the proof of Theorem~\ref{theorem_admm}. 
\end{proof}
\vgap

The next corollary provides a concrete stepsize policy for AC-AADMM.
We skip its proof, as it follows the same arguments used in the proof of Corollary~\ref{main_corollary_2_acce}, with $A$ and $f$ replaced by $K$ and $F$.
\begin{corollary} \label{main_corollary_admm_acce}
In the premise of Theorem~\ref{theorem_admm_acce}, suppose $\tau_1=0,~ \tau_2 = \mu_d$, $\beta \in (0,  1 -{\tfrac{\sqrt{6}}{3}}]$, and $\widetilde \tau_1 = 0, ~\widetilde \tau_2 = 1$. Also assume that $\eta_2 = \min \left\{{(1 - \beta) \eta_1},  \tfrac{1}{4 L_{K, 1}^2/\mu_d + 4 L_{F,1}}\right\}$, and for $t \geq 3$
\begin{align}
&\eta_t = \min \left\{ \tfrac{4}{3}\eta_{t-1}, \tfrac{\tau_{t-2}+\mu_d}{\tau_{t-1}}\cdot \eta_{t-1}, \tfrac{\tau_{t-1}}{4  \mu_d L_{F,t-1} + 4L^2_{K, t-1}} \right\},\label{def_eta_t_admm_acce}\\
&\tau_t = \tau_{t-1} + \tfrac{\mu_d}{2}\left[\alpha +(1-\alpha) \eta_t\cdot\tfrac{4 \mu_d L_{F,t-1} + 4L^2_{K, t-1}}{\tau_{t-1}}\right],\label{def_tau_t_admm_acce}\\
&\widetilde \tau_t = \tfrac{\tau_t}{\mu_d}, \label{def_tilde_tau_t_admm_acce}
\end{align}
where $\alpha \in (0,1]$ denotes an absolute constant. 
Then we have for any $t \geq 2$,
\begin{align}\label{eta_lower_bound_2_admm_acce}
\eta_t \geq \tfrac{3 + \alpha(t-3)}{12 \widetilde L_{t-1}},~~ \text{where} ~~\widetilde L_{t} := \max\{\tfrac{{1}}{4(1-\beta)\eta_1}, L_{K, 1}^2/\mu_d +  L_{F,1},...,L_{K, t}^2/\mu_d + L_{F,t}\},
\end{align}
and consequently, for any $k\ge 1$,
\begin{align}
F(\widehat x_k) + G(\widehat w_k) &- F(x^*) - G(w^*)
 \leq \tfrac{12 \widehat L_k}{6k + \alpha k(k-3)} \left[\tfrac{\|x_0 - x^*\|^2}{\beta} + \eta_2\left(\tfrac{5 L^2_{K, 1}}{2\mu_d} +  \tfrac{5 L_{F, 1}}{2} - \tfrac{1}{2\eta_1}\right)\|x_1 - x_0\|^2 \right], \label{optimality_gap_admm_acce}\\
\|K \widehat x_k - B \widehat w_k + b\| &\leq \mu_d \|\widetilde y_k\|\nn\\
&\leq 2\mu_d \|y^*\| + 2 \sqrt{\tfrac{12 \mu_d \widehat L_k}{6k + \alpha k(k-3)}\left[\tfrac{\|x_0 - x^*\|^2}{\beta} + \eta_2\left(\tfrac{5 L^2_{K, 1}}{2\mu_d} + \tfrac{5L_{F, 1}}{2} - \tfrac{1}{2\eta_1}\right)\|x_1 - x_0\|^2 \right]}.\label{constraint_violation_admm_acce}
\end{align}
\end{corollary}
\vgap

We now derive from Corollary~\ref{main_corollary_admm_acce}
the iteration complexity of AC-AADMM. Similarly to AC-APDHG, by performing an initial line search, we can eliminate the last term 
\[
\left(\tfrac{5\eta_2 L^2_{K, 1}}{4\mu_d} + \tfrac{5\eta_2 L_{F, 1}}{4} - \tfrac{\eta_2}{2\eta_1}\right)\|x_1 - x_0\|^2
\]
from both \eqref{optimality_gap_admm_acce} and \eqref{optimality_gap_admm_acce}, and reduce these bounds to
\begin{align*}
&F(\widehat x_k)  + G(\widehat w_k)- F(x^*) - G(w^*) \leq \mathcal{O}\left( \tfrac{\widehat L_{F, k} D_{x^*}^2}{k^2}  + \tfrac{\widehat L_{K,k}^2 D_{x^*}^2}{\mu_d k^2}\right),\\
&\|K \widehat x_k - B\widehat w_k + b\|  \leq \mu_d \|\widetilde y_k\| \leq \mathcal{O}\left( \mu_d\|y^*\| + \sqrt{\mu_d \widehat L_{F,k}}\cdot\tfrac{D_{x^*}}{k} + \tfrac{\widehat L_{K,k}D_{x^*}}{k} \right).
\end{align*}
As a result, by setting $\mu_d = \mathcal{O}(\epsilon_2/D_{y^*})$,
we can bound the iteration complexity of AC-AADMM to find
an $(\epsilon_1,\epsilon_2)$-optimal solution of problem~\eqref{bilinear_two_parts} by
\begin{align*}
\mathcal{O}\left( \sqrt{\tfrac{L_F D_{x^*}^2}{\epsilon_1}}+ \sqrt{\tfrac{L_F D_{x^*}^2}{\epsilon_2 D_{y^*}}}+ \tfrac{\|K\|D_{x^*}}{\epsilon_2} +\tfrac{\|K|D_{x^*}\sqrt{D_{y^*}}}{\sqrt{\epsilon_1\epsilon_2}}\right),
\end{align*}
which, 
for $(\epsilon_1, \epsilon_2):= (\epsilon, \epsilon/D_{y^*})$,
simplifies to 
\begin{align*}
\mathcal{O}\left( \sqrt{\tfrac{L_F D_{x^*}^2}{\epsilon}}+\tfrac{\|K\|D_{x^*} D_{y^*}}{\epsilon}\right).
\end{align*}
It is worth noting that this complexity bound improves the complexity of the accelerated P-ADMM in \cite{ouyang2015accelerated} in terms of their dependence on $D_{x^*}$ and $D_{y^*}$.

\section{Concluding remarks}
In this paper, we develop novel auto-conditioned primal-dual methods for a few classes of bilinear saddle point problems. First, we introduce the AC-PDHG method for solving problem \eqref{eq:bilinear_saddle_point}, which is fully adaptive to $\|A\|$ and achieves the optimal complexity in terms of the primal-dual gap function. We then extend AC-PDHG to solve linearly constrained problems \eqref{linear_constrained_problem}, establishing convergence guarantees for both optimality gap and constraint violation. Next, we propose the AC-ADMM approach for solving an important class of linearly constrained problems in the form of \eqref{bilinear_two_parts}. While it involves more complicated subproblems for $w_t$, AC-ADMM attains convergence guarantees that depend only on one part of the constraint matrix, $K$, and adapts to the local estimates of $\|K\|$. Finally, we extend AC-PDHG to tackle bilinear saddle point problems where $f$ is a smooth function. By incorporating the acceleration scheme of the AC-FGM algorithm \cite{li2023simple} into AC-PDHG, we obtain the optimal iteration complexity in terms of both smooth and bilinear components. We also extend AC-ADMM to solve \eqref{bilinear_two_parts} with a smooth $F$ in a similar manner.

\renewcommand\refname{Reference}
\bibliographystyle{siam} 
\bibliography{arxiv_version}

\newcommand{\noopsort}[1]{} \newcommand{\printfirst}[2]{#1}
  \newcommand{\singleletter}[1]{#1} \newcommand{\switchargs}[2]{#2#1}
\begin{thebibliography}{10}

\bibitem{auslender2006interior}
{\sc A.~Auslender and M.~Teboulle}, {\em Interior gradient and proximal methods
  for convex and conic optimization}, SIAM Journal on Optimization, 16 (2006),
  pp.~697--725.

\bibitem{bertsekas1982constrained}
{\sc D.~P. Bertsekas}, {\em {Constrained Optimization and Lagrange Multiplier
  Methods}}, Academic Press, 1982.

\bibitem{bouwmans2016handbook}
{\sc T.~Bouwmans, N.~S. Aybat, and E.-h. Zahzah}, {\em Handbook of robust
  low-rank and sparse matrix decomposition: Applications in image and video
  processing}, CRC Press, 2016.

\bibitem{chambolle2011first}
{\sc A.~Chambolle and T.~Pock}, {\em A first-order primal-dual algorithm for
  convex problems with applications to imaging}, Journal of mathematical
  imaging and vision, 40 (2011), pp.~120--145.

\bibitem{chang2022golden}
{\sc X.-K. Chang, J.~Yang, and H.~Zhang}, {\em Golden ratio primal-dual
  algorithm with linesearch}, SIAM Journal on Optimization, 32 (2022),
  pp.~1584--1613.

\bibitem{chen2012fast}
{\sc Y.~Chen, W.~Hager, F.~Huang, D.~Phan, X.~Ye, and W.~Yin}, {\em Fast
  algorithms for image reconstruction with application to partially parallel
  {MR} imaging}, SIAM Journal on Imaging Sciences, 5 (2012), pp.~90--118.

\bibitem{chen2014optimal}
{\sc Y.~Chen, G.~Lan, and Y.~Ouyang}, {\em Optimal primal-dual methods for a
  class of saddle point problems}, SIAM Journal on Optimization, 24 (2014),
  pp.~1779--1814.

\bibitem{gabay1976dual}
{\sc D.~Gabay and B.~Mercier}, {\em A dual algorithm for the solution of
  nonlinear variational problems via finite element approximation}, Computers
  \& Mathematics with Applications, 2 (1976), pp.~17--40.

\bibitem{glowinski1975approximation}
{\sc R.~Glowinski and A.~Marroco}, {\em Sur l'approximation, par
  {\'e}l{\'e}ments finis d'ordre un, et la r{\'e}solution, par
  p{\'e}nalisation-dualit{\'e} d'une classe de probl{\`e}mes de dirichlet non
  lin{\'e}aires}, ESAIM: Mathematical Modelling and Numerical
  Analysis-Mod{\'e}lisation Math{\'e}matique et Analyse Num{\'e}rique, 9
  (1975), pp.~41--76.

\bibitem{he2012on}
{\sc B.~He and X.~Yuan}, {\em On the o(1/n) convergence rate of the
  douglas-rachford alternating direction method}, SIAM Journal on Numerical
  Analysis, 50 (2012), pp.~700--709.

\bibitem{hestenes1969multiplier}
{\sc M.~R. Hestenes}, {\em Multiplier and gradient methods}, Journal of
  Optimization Theory and Applications, 4 (1969), pp.~303--320.

\bibitem{jacob2009group}
{\sc L.~Jacob, G.~Obozinski, and J.-P. Vert}, {\em Group lasso with overlap and
  graph lasso}, in Proceedings of the 26th annual international conference on
  machine learning, 2009, pp.~433--440.

\bibitem{kolda2009tensor}
{\sc T.~G. Kolda and B.~W. Bader}, {\em Tensor decompositions and
  applications}, SIAM review, 51 (2009), pp.~455--500.

\bibitem{korpelevich1983extrapolation}
{\sc G.~M. Korpelevich}, {\em Extrapolation gradient methods and relation to
  modified lagrangeans}, Ekonomika i Matematicheskie Metody, 19 (1983),
  pp.~694--703.
\newblock in Russian; English translation in Matekon.

\bibitem{kotsalis2022simple1}
{\sc G.~Kotsalis, G.~Lan, and T.~Li}, {\em Simple and optimal methods for
  stochastic variational inequalities, i: Operator extrapolation}, SIAM Journal
  on Optimization, 32 (2022), pp.~2041--2073.

\bibitem{kotsalis2022simple}
\leavevmode\vrule height 2pt depth -1.6pt width 23pt, {\em Simple and optimal
  methods for stochastic variational inequalities, ii: Markovian noise and
  policy evaluation in reinforcement learning}, SIAM Journal on Optimization,
  32 (2022), pp.~1120--1155.

\bibitem{lan2013bundle}
{\sc G.~Lan}, {\em Bundle-level type methods uniformly optimal for smooth and
  nonsmooth convex optimization}, Mathematical Programming,  (2013), pp.~1--45.

\bibitem{lan2016gradient}
\leavevmode\vrule height 2pt depth -1.6pt width 23pt, {\em Gradient sliding for
  composite optimization}, Mathematical Programming, 159 (2016), pp.~201--235.

\bibitem{LanBook2020}
\leavevmode\vrule height 2pt depth -1.6pt width 23pt, {\em First-order and
  stochastic optimization methods for machine learning}, vol.~1, Springer,
  2020.

\bibitem{lan2011primal}
{\sc G.~Lan, Z.~Lu, and R.~D. Monteiro}, {\em Primal-dual first-order methods
  with {$\mathcal {O}(1/\epsilon)$} iteration-complexity for cone programming},
  Mathematical Programming, 126 (2011), pp.~1--29.

\bibitem{lan2022accelerated}
{\sc G.~Lan and Y.~Ouyang}, {\em Accelerated gradient sliding for structured
  convex optimization}, Computational Optimization and Applications, 82 (2022),
  pp.~361--394.

\bibitem{lan2023graph}
{\sc G.~Lan, Y.~Ouyang, and Y.~Zhou}, {\em Graph topology invariant gradient
  and sampling complexity for decentralized and stochastic optimization}, SIAM
  journal on Optimization, 33 (2023), pp.~1647--1675.

\bibitem{li2023simple}
{\sc T.~Li and G.~Lan}, {\em A simple uniformly optimal method without line
  search for convex optimization}, arXiv preprint arXiv:2310.10082,  (2023).

\bibitem{liu2021acceleration}
{\sc Y.~Liu, Y.~Xu, and W.~Yin}, {\em Acceleration of primal--dual methods by
  preconditioning and simple subproblem procedures}, Journal of Scientific
  Computing, 86 (2021), p.~21.

\bibitem{lu2023unified}
{\sc H.~Lu and J.~Yang}, {\em On a unified and simplified proof for the ergodic
  convergence rates of ppm, pdhg and admm}, arXiv preprint arXiv:2305.02165,
  (2023).

\bibitem{malitsky2020golden}
{\sc Y.~Malitsky}, {\em Golden ratio algorithms for variational inequalities},
  Mathematical Programming, 184 (2020), pp.~383--410.

\bibitem{malitsky2018first}
{\sc Y.~Malitsky and T.~Pock}, {\em A first-order primal-dual algorithm with
  linesearch}, SIAM Journal on Optimization, 28 (2018), pp.~411--432.

\bibitem{monteiro2010complexity}
{\sc R.~D. Monteiro and B.~F. Svaiter}, {\em On the complexity of the hybrid
  proximal extragradient method for the iterates and the ergodic mean}, SIAM
  Journal on Optimization, 20 (2010), pp.~2755--2787.

\bibitem{monteiro2013iteration}
{\sc R.~D. Monteiro and B.~F. Svaiter}, {\em Iteration-complexity of
  block-decomposition algorithms and the alternating direction method of
  multipliers}, SIAM Journal on Optimization, 23 (2013), pp.~475--507.

\bibitem{nemirovski2004prox}
{\sc A.~Nemirovski}, {\em Prox-method with rate of convergence ${O}(1/t)$ for
  variational inequalities with {L}ipschitz continuous monotone operators and
  smooth convex-concave saddle point problems}, SIAM Journal on Optimization,
  15 (2004), pp.~229--251.

\bibitem{nemirovsky1992information}
{\sc A.~Nemirovsky}, {\em Information-based complexity of linear operator
  equations}, Journal of Complexity, 8 (1992), pp.~153--175.

\bibitem{nesterov1983method}
{\sc Y.~Nesterov}, {\em A method for unconstrained convex minimization problem
  with the rate of convergence {$O(1/k^2)$}}, in Doklady an ussr, vol.~269,
  1983, pp.~543--547.

\bibitem{nesterov2005excessive}
{\sc Y.~Nesterov}, {\em Excessive gap technique in nonsmooth convex
  minimization}, SIAM Journal on Optimization, 16 (2005), pp.~235--249.

\bibitem{nesterov2005smooth}
{\sc Y.~Nesterov}, {\em Smooth minimization of non-smooth functions},
  Mathematical programming, 103 (2005), pp.~127--152.

\bibitem{nocedal2006numerical}
{\sc J.~Nocedal and S.~J. Wright}, {\em Numerical optimization}, Springer
  Science+ Business Media, 2006.

\bibitem{ouyang2013stochastic}
{\sc H.~Ouyang, N.~He, L.~Tran, and A.~G. Gray}, {\em Stochastic alternating
  direction method of multipliers}, in Proceedings of the 30th International
  Conference on Machine Learning (ICML-13), 2013, pp.~80--88.

\bibitem{ouyang2015accelerated}
{\sc Y.~Ouyang, Y.~Chen, G.~Lan, and E.~Pasiliao~Jr}, {\em An accelerated
  linearized alternating direction method of multipliers}, SIAM Journal on
  Imaging Sciences, 8 (2015), pp.~644--681.

\bibitem{ouyang2021lower}
{\sc Y.~Ouyang and Y.~Xu}, {\em Lower complexity bounds of first-order methods
  for convex-concave bilinear saddle-point problems}, Mathematical Programming,
  185 (2021), pp.~1--35.

\bibitem{pena2008nash}
{\sc J.~Pena}, {\em Nash equilibria computation via smoothing techniques},
  Optima, 78 (2008), pp.~12--13.

\bibitem{pock2011diagonal}
{\sc T.~Pock and A.~Chambolle}, {\em Diagonal preconditioning for first order
  primal-dual algorithms in convex optimization}, in 2011 International
  Conference on Computer Vision, IEEE, 2011, pp.~1762--1769.

\bibitem{powell1969method}
{\sc M.~J.~D. Powell}, {\em A method for nonlinear constraints in minimization
  problems}, in Optimization ({S}ympos., {U}niv. {K}eele, {K}eele, 1968),
  Academic Press, London, 1969, pp.~283--298.

\bibitem{rudin1992nonlinear}
{\sc L.~I. Rudin, S.~Osher, and E.~Fatemi}, {\em Nonlinear total variation
  based noise removal algorithms}, Physica D: nonlinear phenomena, 60 (1992),
  pp.~259--268.

\bibitem{tibshirani2005sparsity}
{\sc R.~Tibshirani, M.~Saunders, S.~Rosset, J.~Zhu, and K.~Knight}, {\em
  Sparsity and smoothness via the fused lasso}, Journal of the Royal
  Statistical Society Series B: Statistical Methodology, 67 (2005),
  pp.~91--108.

\bibitem{tomioka2011statistical}
{\sc R.~Tomioka, T.~Suzuki, K.~Hayashi, and H.~Kashima}, {\em Statistical
  performance of convex tensor decomposition}, Advances in neural information
  processing systems, 24 (2011).

\bibitem{tseng2008accelerated}
{\sc P.~Tseng}, {\em On accelerated proximal gradient methods for
  convex-concave optimization}, submitted to SIAM Journal on Optimization,
  (2008).

\bibitem{vladarean2021first}
{\sc M.-L. Vladarean, Y.~Malitsky, and V.~Cevher}, {\em A first-order
  primal-dual method with adaptivity to local smoothness}, Advances in neural
  information processing systems, 34 (2021), pp.~6171--6182.

\bibitem{yang2011alternating}
{\sc J.~Yang and Y.~Zhang}, {\em Alternating direction algorithms for
  $\backslash$ell\_1-problems in compressive sensing}, SIAM journal on
  scientific computing, 33 (2011), pp.~250--278.

\bibitem{ye2011computational}
{\sc X.~Ye, Y.~Chen, and F.~Huang}, {\em Computational acceleration for {MR}
  image reconstruction in partially parallel imaging}, Medical Imaging, IEEE
  Transactions on, 30 (2011), pp.~1055--1063.

\bibitem{ye2011fast}
{\sc X.~Ye, Y.~Chen, W.~Lin, and F.~Huang}, {\em Fast {MR} image reconstruction
  for partially parallel imaging with arbitrary k-space trajectories}, IEEE
  Transactions on Medical Imaging, 30 (2011), pp.~575--585.

\end{thebibliography}

\end{document}